\documentclass[a4paper,12pt]{article}
\usepackage{amsmath, amsthm, amsfonts, amssymb, bbm}
\usepackage{graphicx, psfrag}
\usepackage[usenames,dvipsnames]{xcolor}
\usepackage{graphicx}
\usepackage{ulem}
\usepackage{cite}
\usepackage{hyperref}
\usepackage{stackrel}
\usepackage{framed,color}
\definecolor{shadecolor}{rgb}{0.93, 0.79, 0.69}
\usepackage[utopia]{mathdesign}
\usepackage{mathtools}
\usepackage{mhchem}
\usepackage{gensymb}

\definecolor{columbiablue}{rgb}{0.61, 0.87, 1.0}
\definecolor{mycolor}{rgb}{.1, .7, 1}
\usepackage[margin=0.9in]{geometry}

\usepackage{amsmath,lineno,centernot}
\usepackage{soul}
\usepackage{enumerate}
\usepackage{titlesec}
\usepackage{cancel}

\usepackage{ulem}

\usepackage{mathtools}

\newcommand{\La}{\Lambda}
\newcommand{\eps}{\epsilon}
\newtheorem*{acknowledgements}{Acknowledgements}
\newcommand{\und}{\underline}

\newcommand{\PP}{\mathbf P}

\usepackage{tikz}
\usetikzlibrary{arrows,decorations.pathmorphing,backgrounds,positioning,fit,petri}
\usetikzlibrary{shapes}
\usetikzlibrary{arrows.meta}

\DeclareMathAlphabet{\mathpzc}{OT1}{pzc}{m}{it}

\newtheorem{proposition}{Proposition}[section]
\newtheorem{theorem}[proposition]{Theorem}
\newtheorem{lemma}[proposition]{Lemma}
\newtheorem{definition}[proposition]{Definition}
\newtheorem{corollary}[proposition]{Corollary}

\newtheorem{cla}[proposition]{Claim}

\newtheorem{remark}[proposition]{Remark}

\numberwithin{equation}{section}
\definecolor{verdescuro}{RGB}{0,120,0}

\allowdisplaybreaks

\title{Space-time bounds for the stirring process\\ with births and deaths}
\author{P. Birmpa \thanks{Department of Actuarial Mathematics \& Statistics, Heriot-Watt University and the Maxwell Institute for Mathematical Sciences, Mayfield Rd, Edinburgh EH9 3FD, UK, Email {\tt P.Birmpa@hw.ac.uk }} \and P. Gon\c calves\thanks{Instituto Superior T\'ecnico, Department of Mathematics, University of Lisbon, Av. Rovisco Pais 1, 1049-001, Lisbon. E-mail: {\tt pgoncalves@tecnico.ulisboa.pt}}
\and D. Tsagkarogiannis\thanks{Universit\`a dell'Aquila, DISIM, 67100, L'Aquila.  E-mail: {\tt dimitrios.tsagkarogiannis@univaq.it}}}

\date{}

\begin{document}
\maketitle 
\abstract{We consider a one-dimensional symmetric simple exclusion process on the segment
$\{-N,\ldots,N\}$, coupled to boundary dynamics that inject particles from the
right reservoir and remove particles from the left reservoir, each acting on a window
of size $K$. We derive sharp bounds on the space and space--time $v$-functions of
arbitrary degree associated with the centered occupation variables. In particular, we establish bounds of order $N^{-1+2\zeta}$ for spatial $v$-functions of degrees $1$ and $2$, of order $N^{-1}$ for degrees $3,\ldots,K$, and of order $N^{-1-\zeta}$ for degrees at least $K+1$, for some sufficiently small $\zeta>0$. For space–time $v$-functions, we obtain bounds of order $N^{-1}$ for total degrees $2$ and $3$, together with improved bounds for higher degrees, incorporating the dependence on the time separation. These estimates provide the quantitative control needed to study the non-equilibrium fluctuations of the model.}

\section{Introduction}
The derivation of macroscopic equations from microscopic particle dynamics is a long-standing problem in mathematical physics. Interacting particle systems (IPS) provide a stochastic framework for addressing this problem, allowing one to study how macroscopic behaviour emerges from local random interactions. Introduced in the mathematical literature by Spitzer~\cite{Spi}, IPS are Markov processes describing particles evolving on a discrete approximation of the macroscopic space according to prescribed local rules; we refer to \cite{liggettinteracting} for a general introduction. The connection between the microscopic and macroscopic descriptions is obtained by introducing a scaling parameter $\epsilon=N^{-1}$ and studying the limit $\epsilon\to0$.

A fundamental example is the exclusion process, in which each lattice site can be occupied by at most one particle. Particles evolve through nearest-neighbour exchanges driven by independent Poisson clocks, subject to the exclusion constraint. Under diffusive scaling, the clocks are accelerated by a factor $N^2$. Exclusion processes and their variants provide a basic class of models for studying hydrodynamic limits and fluctuations out of equilibrium. We note that the number of particles is conserved under this dynamics.  We supplement this exchange dynamics with a Glauber dynamics that injects and removes particles in the system. The choice of the macroscopic space in this article is the one-dimensional interval $[-1,1]$ and the corresponding microscopic 
space is the space $\{-N,-N+1,\dots, N-1,N\}$ . The Glauber dynamics acts on two set of sites : those at distance less than or equal $K\geq 2$ to the points $-N,N$. The injection and removal rates are defined in the following way: from the right of the system a particle can be injected in the system at a  site $y$ in ${N-K+1,\dots, N}$ if the site $y$ is the first empty site on the window ${N-K+1,\dots, N}$ counting it from right to left. A particle in the left window of size $K$, namely in the set of sites $\{-N,\dots, -N+K-1\}$  can be removed from the site $z\in \{-N,\dots, -N+K-1\}$ if $z$ is the first occupied site in that window counting it from left to right. Note that the superposition of this dynamics with the exchange dynamics described previously, destroys the conservation law of the system, though the number of particles is the quantity that we want to analyze. 
We also note  that this dynamics was first introduced in \cite{de2011current} but the boundary dynamics was slowed down with respect to  the exchange dynamics by a factor of $\eps$ and later the same dynamics slowed by a factor $\eps^\theta$, with $\theta>0$ 
was also explored in \cite{erignoux2020hydrodynamics,erignoux2019hydrodynamics}.
\bigskip
 
\textbf{The hydrodynamic limit and the fluctuations out of equilibrium.} {One of the first intriguing questions about this model is the description of the space--time evolution of the particle density. For that purpose, we } consider the empirical measure which assigns mass $\eps$ to each occupied site.
The goal is to identify the macroscopic PDE governing its evolution, namely the hydrodynamic equation. Similarly, one can also study the fluctuations around the solution of the hydrodynamic equation by considering the fluctuation field which is the empirical measure properly centered and rescaled with $\sqrt\epsilon$.
In both cases, a useful starting point is the discrete density profile
$
\rho_t^\eps(x)=\mathbb E_{\mu^\eps}^\eps[\eta_t(x)], 
$ which, by Kolmogorov's forward equation, satisfies
$
\partial_t\rho_t^\eps(x)
=
\mathbb E_{\mu^\eps}^\eps[L_\eps\eta_t(x)].
$
The main difficulty in closing the right-hand side of the last equation with respect to the discrete density profile is caused by the boundary dynamics: the corresponding rates depend on the configuration in the boundary windows and generate higher-order products of occupation variables. Consequently, the evolution equation for $\rho_t^\eps$ is not closed and involves higher-order moments.

To overcome this difficulty, one introduces an auxiliary profile $\rho_\eps(x,t)$ satisfying the same evolution equation, but with the occupation variables appearing in the boundary terms treated as uncorrelated. This factorizes the corresponding expectations and yields a closed evolution equation for $\rho_\eps$. Thus, unlike $\rho_t^\eps$, the auxiliary profile can be analyzed directly.

In \cite{de2011current}, assuming that $\rho_\eps(\cdot,0)$ converges weakly to a macroscopic profile $u_0\in L^\infty([-1,1],[0,1])$, the hydrodynamic limit was obtained by proving that $\rho_\eps(x,t)$ is uniformly close to $\rho_t^\eps(x)$ and converges to the solution of the heat equation with the appropriate boundary conditions; see Theorem~\ref{hydrodynamiclimit}. Related hydrodynamic limits for slowed boundary dynamics were later established in \cite{erignoux2020hydrodynamics,erignoux2019hydrodynamics}. 
The comparison between $\rho_\eps$ and $\rho_t^\eps$ relies on stronger estimates for the correlations of the occupation variables. These correlation estimates for the hydrodynamic limit were established in \cite{de2012truncated}. However, for the case of the fluctuations these estimates are not enough due to the weaker rescaling by $\sqrt\epsilon$. The purpose of the present work is to overcome this difficulty.

 \bigskip
 
{\textbf{The $v$-function.}} 
 Given an initial measure $\mu^\eps$, the correlation of particles at sites $x,y$ is defined as the expectation with respect to that given measure of the centered variables $\eta_t(x)-\rho_t^\eps(x)$ and  $\eta_t(y)-\rho_t^\eps(y)$. As mentioned above, the evolution equation for $\rho_t^\epsilon$ involves higher–order moments, which makes it more convenient to consider instead the correlations of the centered variables {$\eta_t(x)-\rho_\epsilon(x,t)$ and $\eta_t(y)-\rho_\epsilon(y,t)$}. These correlations are referred to as the $v$-functions.
The definition above concerns only two spatial variables, but it extends in a straightforward manner to any number of variables. For that reason, the $v$-functions will be denoted precisely by {$v^\eps_n(\underline{x},t)$}, where $n$ represents the number of spatial variables and the cardinality of $\underline{x}$ is $n$. From Kolmogorov's identity one can get an evolution equation for the $v$-functions 
 which allows to apply Duhamel's formula and give a probabilistic interpretation of the solution in terms of the stirring process. As the 
Markov generator increases the degree of the $v$-function, the probabilistic interpretation uses a branching process that keeps track of the action of the Glauber dynamics, in the sense that it takes into account how many particles are created/destroyed at the boundary windows. 
The estimate for the $v$-functions obtained in \cite{de2012truncated}, following the strategy developed in \cite{demasi2006mathematical}, keeps track of the number $n$ of spatial variables and is different depending on whether we consider short or long times. 

We recall this estimate here: for a fixed  integer number $K$  such that $K\geq 2$ as given in  Section \ref{sec:model}.
\begin{equation}\label{apriori_introduction}
{|v_n^\eps(\und x,t)|}\leq
\begin{cases}
c_n(\epsilon^{-2}t)^{-c^* n}, & \text{for } t\leq\epsilon^{\beta^*}\\
c_n\epsilon^{(2-\beta^*)c^* n}, & \text{for } \epsilon^{\beta^*}\leq t\leq\tau\log\epsilon^{-1}
\end{cases}
\end{equation}
where {$\beta^*<2$}, $\tau>0$ and $c^*<\frac{1}{4(K+2)}$. We observe that when $n$ is sufficiently large and $\eps<1$, the estimate can become of order $\epsilon$ or even smaller. However, for small $n$, the estimate is not so good since the exponent of $\eps$, $c^* n$ is small. 
{
It is enough to prove the hydrodynamic limit, see \cite{de2011current} but it is not good enough for the study of the limiting equation for the fluctuation field which scales with $\sqrt\epsilon$. An improvement is necessary and this is the purpose of the present paper.
}
In this work,  we obtain  better estimates for the case of small $n$, giving up the more delicate $n$-dependent estimate. However, the latter will still be used for proving the new estimate. In particular, we show that for $n=1,2$ the bound is very close to $\epsilon$, for $3 \leq n \leq K$ it is  of order $\epsilon$, and for $n \geq K+1$ it is of order $\epsilon^{1+\zeta}$ with $\zeta > 0$ for times $0<t\leq T$, for some $T>0$. Furthermore, we provide a bound which holds for a fixed time horizon $T$ in contrast to the aforementioned result in \cite{de2012truncated}. The extension to times $T\log\epsilon^{-1}$ would require more work.

\bigskip

{\textbf{The space-time $v$-function.}} Given an initial measure $\mu^\epsilon$, the one--point space--time correlation of
particles at sites $y$ and $x$ involves two times $s<r$, respectively: particles are located
at site $y$ at time $s$, and at site $x$ at time $r$. It is defined as the
expectation, with respect to $\mu^\epsilon$, of the centered occupation variables
$\eta_s(y)-\rho_\epsilon(s,y)$ and
$\eta_r(x)-\rho_\epsilon(r,x)$. This definition
extends in a straightforward way to any number of spatial variables at the two
times $s$ and $r$. 

As in the case of the space $v$-functions, we derive a discrete evolution
equation for the space--time $v$-functions by fixing the locations of the
particles at time~$s$ and analyzing the evolution of the positions at the
later time~$r$. The integral form of the discrete evolution equation includes all possible scenarios of particle's motion within the system between
times $s$ and $r$. In contrast to the purely spatial case, the integral
form of the discrete equation now contains an initial term at time~$s$, whose estimation
requires bounding the expectation, with respect to the stirring process, of
the $v$-function at time $s$ evaluated at the (random) configuration occupied at
that time. This brings additional technical difficulties, but at the same time allows for a more detailed analysis and we show that this expectation is of order $\epsilon$ for $n=2,3$, and of order
$\epsilon^{1+\zeta}$ for $n\ge 4$ for some $\zeta>0$. This yields an additional
improvement of the overall estimates and, in turn, strengthens the final bounds
on the space--time $v$-functions.

\bigskip

{\textbf{Future problems.}} Beyond the present model, the techniques developed here provide a general framework for estimating space correlation $v$-functions in a broader class of interacting particle systems once an a priori estimate, such as \eqref{apriori_introduction}, has been established.  The $n$-dependent estimate \eqref{apriori_introduction} is essential in order to control the scenario of creating many particles. But then further improvement may be required according to the application. This can be obtained by a more detailed analysis of the "future" of the particles as it was performed here for the particular case of the fluctuations in which we had to overcome the difficulty of a weaker scaling. In a similar fashion, this strategy of obtaining an improvement by an exhaustive estimation can be in principle applied to other sources of complications such as including weak
asymmetries in the bulk
or for particle systems with generic boundary conditions as in \cite{landimmangisalvador}. The key ingredient is detailed estimates of the kernel of the bulk generator. In fact, estimates for $v$-functions have been obtained in a range of examples including the
weakly asymmetric exclusion process studied in \cite{de1989weakly} as well as
\cite{ferrari1991symmetric, calderoni1989transient, lebowitz1988convergence, de1992kinetic}. A further discussion regarding the main steps of the proof in a general setting and of the structural properties that a model should satisfy for this method to be applicable will be presented in Section~\ref{integral inequalities for space  v-function}.

\begin{acknowledgements}
\small{The research of P.G. is partially funded by Fundação para a Ciência e Tecnologia (FCT), through the ERC/FCT SAUL project. D.T. acknowledges financial support from the Italian Research Funding Agency
(MIUR) through the PRIN project “Emergence of condensation-like phenomena in interacting particle systems: kinetic and lattice models”, grant n. 202277WX43, as well as from Istituto Nazionale di Alta Matematica (INDAM).
The authors thank the warm hospitality of the University of Minho and Instituto Superior Técnico in Portugal, GSSI in Italy and the Maxwell institute for Mathematical Sciences, Edinburgh, UK, where parts of this work were developed. The authors are also indebted to Anna De Masi and Errico Presutti for many insightful and clarifying discussions. Work funded by FCT/Portugal and the Recovery and Resilience Plan (PRR) through projects $\textrm{UID}/04459/2025 (DOI: 10.54499/UID/04459/2025)$ and $UID/PRR/04459/2025 (DOI:10.54499/UID/PRR/04459/2025)$.}
\end{acknowledgements}


\section{The model and its hydrodynamic limit}
\label{sec:1}
In this section, we introduce the interacting particle system considered throughout the paper and recall its main hydrodynamic properties. We first describe the microscopic dynamics, including the stirring process and the boundary mechanisms, and then recall the corresponding hydrodynamic limit, which identifies the macroscopic evolution of the particle density.

\subsection{The model}
\label{sec:model}

Let $N \in \mathbb{N}$ represent the scaling parameter, which we will consider in the limit  $N \to \infty$. Fix an arbitrary integer $1 \leq K \leq N$. We examine the stirring process on the discrete subspace $\Lambda_N := [-N, N] \subset \mathbb{Z}$. Particle configurations are represented by elements $\eta \in  \{0,1\}^{\Lambda_N}$, where $\eta(x) = 1$ indicates the presence of a particle at site $x$, and $\eta(x) = 0$ otherwise. The stirring process, incorporating births and deaths, is a Markov process $\{\eta_{t}\}_{t \geq 0}$ fully characterized by its infinitesimal generator given by 
      \begin{equation}
        \label{eq:gen_total}
L_\epsilon:=\epsilon^{-2}\big(L_0+\epsilon L_{b,-}+\epsilon L_{b,+}\big), \quad \epsilon\equiv \frac 1N,
     \end{equation}
where $L_0$ is the generator of the stirring process in $\Lambda_N$, namely
      \begin{equation}
        \label{eq:gen_stir}
 L_0 f(\eta):=\frac 12\sum_{x=-N}^{N-1} [f(\eta^{(x,x+1)})-f(\eta)]
     \end{equation}
and 
      \begin{equation}
        \label{eq:gen_bound}
L_{b,\pm} f(\eta):= \frac{j}{2}
\sum_{x\in I_\pm}D_{\pm}\eta(x) [f(\eta^{(x)})-f(\eta)\Big],
\end{equation}
with $j>0$, $I_+=[N-K+1,N]$, $I_-=[-N,-N+K-1]$,
\begin{equation}\label{eq:boundary_operators}
\begin{split}
& D_+u (x)= [1-u(x)]\,u(x+1)\,u(x+2)\,\cdots u(N), \; x\in I_+,
       \\
& D_-u (x)=  u(x)\,[1-u(x-1)][1-u(x-2)]\,\cdots\,[1-u(-N)], \;  x\in I_-\,.
\end{split}
     \end{equation}
      Note that the generator includes a factor of $\epsilon^{-2}$, indicating that the process is considered on the diffusive time scale. In \eqref{eq:gen_bound}, $\eta^{x,x+1}$, with $-N \leq x \leq N-1$, represents the configuration obtained from $\eta$ by exchanging the occupation variables $\eta(x)$ and $\eta(x+1)$:
\[\eta^{(x,x+1)} (y) =
\begin{cases}
\eta(x+1),& y=x\;,\\
\eta(x), & y=x+1\;,\\
\eta(y), & y\neq x,x+1\;,\\
\end{cases}
\]
while $\eta^{(x)}$, for $x\in I_\pm $, is  the configuration obtained by flipping the
occupation variable $\eta(x)$ :
\[
\eta^{(x)}(y) =
\begin{cases}
\eta(y),& y\neq x\;, \\
1-\eta(y), & y=x\,.\\
\end{cases}
\]

The operator $N^2L_0$ is the generator of the stirring process, where neighboring particles exchange positions at rate $N^2/2$. The operator $NL_{b,\pm}$ represents the generator of a birth (resp. death) process: at rate $jN/2$, a particle is created (resp. destroyed) at the first empty (resp. occupied) site starting from $N$ (resp. $-N$) in $I_+$ (resp. $I_-$). If no site is empty (resp. occupied), the birth (resp. death) process is aborted, see the figure below.

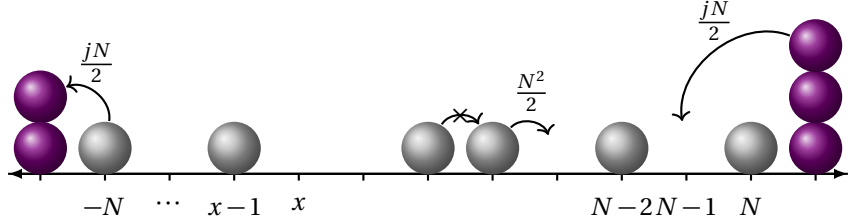
\begin{figure}[htb]
\begin{center}
\begin{tikzpicture}[thick, scale=0.85][h!]
\draw[latex-] (-6.5,0) -- (6.5,0) ;
\draw[-latex] (-6.5,0) -- (6.5,0) ;
\foreach \x in  {-6,-5,-4,-3,-2,-1,0,1,2,3,4,5,6}
\draw[shift={(\x,0)},color=black] (0pt,0pt) -- (0pt,-3pt) node[below] 
{};

       \node[ball color=violet, shape=circle, minimum size=0.7cm] (U) at (-6.,1.2) {};
       \node[ball color=violet, shape=circle, minimum size=0.7cm]  at (-6.,0.4) {};
       \node[ball color=violet, shape=circle, minimum size=0.7cm] (R) at (6.,0.4) {};
         \node[ball color=violet, shape=circle, minimum size=0.7cm] (S) at (6.,1.2) {};
         \node[ball color=violet, shape=circle, minimum size=0.7cm] (T) at (6.,2.0) {};
        
        \node[shape=circle,minimum size=0.7cm] (Q) at (-6.,2) {};
        \node[shape=circle,minimum size=0.7cm] (P) at (-5.,1.2) {};
           \node[shape=circle,minimum size=0.7cm] (PP) at (-5.,0.4) {};
        \node[shape=circle,minimum size=0.7cm] (M) at (-6.,1.2) {};
        \node[shape=circle,minimum size=0.7cm] (N) at (-2.,1.2) {};
        \node[shape=circle,minimum size=0.7cm] (J) at (2.,0.2) {};
            \node[shape=circle,minimum size=0.7cm] (JJ) at (4,0.4) {};
    
\node[ball color=black!30!, shape=circle, minimum size=0.7cm] (C) at (-5,0.4) {};

\node[ball color=black!30!, shape=circle, minimum size=0.7cm] (C) at (-3,0.4) {};    
    
\node[ball color=black!30!, shape=circle, minimum size=0.7cm] (C) at (0,0.4) {};

\node[ball color=black!30!, shape=circle, minimum size=0.7cm] (D) at (1,0.4) {};

\node[ball color=black!30!, shape=circle, minimum size=0.7cm] (L) at (3.,0.4) {};

\node[ball color=black!30!, shape=circle, minimum size=0.7cm] (H) at (5.,0.4) {};

\draw[] (-5,-0.2) node[below] {\footnotesize{$-N$}};
\draw[] (-4,-0.2) node[below] {\footnotesize{$\cdots$}};
\draw[] (-3,-0.2) node[below] {\footnotesize{$x-1$}};
\draw[] (-2,-0.2) node[below] {\footnotesize{$x$}};

\draw[] (5,-0.2) node[below] {\footnotesize{$N$}};
\draw[] (4,-0.2) node[below] {\footnotesize{$N-1$}};
\draw[] (3,-0.2) node[below] {\footnotesize{$N-2$}};

\node[shape=circle,minimum size=0.7cm] (Q) at (-6.,2.0) {};
\node[shape=circle,minimum size=0.7cm] (W) at (3,2.0) {};

setinhas
\path [->] (PP) edge[bend right=60] node[above] {$\tfrac{jN}{2}$} (U);
\path [->] (T) edge[bend right=60] node[above] {$\tfrac{jN}{2}$} (JJ);
\path [->] (C) edge[bend left =60] node[] {$\times$} (D);

\path [->] (D) edge[bend left =60] node[above] {$\tfrac{N^2}{2}$} (J);
              
\node[shape=circle,minimum size=0.7cm] (K) at (0,0.4) {};
\node[shape=circle,minimum size=0.7cm] (G) at (2,0.4) {};
\node[shape=circle,minimum size=0.7cm] (E) at (-1.,0.5) {};
\node[shape=circle,minimum size=0.7cm] (EE) at (-3.,0.5) {};  
\end{tikzpicture}
\end{center}
\end{figure}
\subsection{Hydrodynamic limit}
\label{subsec:2}
In this section, we review the hydrodynamic limit of our model, as developed in \cite{de2011current}. For $x\in\Lambda_N$  and $t\geq 0$, let us denote 
\begin{equation}\label{eq:exp_eta}
\rho^\epsilon_t(x)=\mathbb{E}^{\epsilon}_{\mu^\epsilon}[\eta_{t}(x)].\end{equation}
We observe that the function $\rho^\epsilon_t(x)$, from Kolmogorov's forward equation, satisfies an evolution equation given by 
\begin{equation}\label{eq:kolm_disc_prof}
\partial_t \rho_t^\epsilon (x)=\frac 12\Delta_\epsilon  \rho_t^\epsilon (x)+\frac {\epsilon^{-1}j}{2} \textbf{1}_{x\in I_{+}}\mathbb E^\epsilon_{\mu^\eps}[D_+\eta_{t}(x)]- \frac {\epsilon^{-1}j}{2} \textbf{1}_{x\in I_{-}}\mathbb E^\epsilon_{\mu^\eps}[D_-\eta_{t}(x)].
\end{equation}
Since the boundary generators increase the degree of polynomial functions of $\eta$, we do not get a closed equation for $\partial_t \rho^\epsilon_t(x)$. Therefore we introduce the function $\rho_\epsilon (x,t)$ as the
solution of
\begin{equation}\label{eq:linearized}
	\begin{cases}
	&\partial_t\rho_\epsilon(x,t)= \frac 12 \Delta_\epsilon\rho_\epsilon(x,t)+
\epsilon^{-1} \frac j2\Big(\mathbf 1_{\{x\in I_+\}} D_+\rho_\epsilon(x,t)
-\mathbf 1_{\{x\in I_-\}}  D_-\rho_\epsilon(x,t)\Big)\\
	&{\rho_\epsilon(x,0)=u_0(\epsilon x)},
	\end{cases}
	\end{equation}
where $\Delta_\epsilon=\epsilon^{-2}\Delta$, $\Delta$
the discrete Laplacian in $\Lambda_
N$ with reflecting boundary conditions:
    \begin{eqnarray}
&&\Delta u(x)=   u(x+1)+ u(x-1)-2u(x),\qquad |x|<\epsilon^{-1}
\\&&\Delta u(\pm N)= u(\pm(N-1))-u(\pm N).
         \label{def_laplacian}
        \end{eqnarray}

\begin{theorem}\cite [Theorem 1] {de2011current}\label{hydrodynamiclimit}
Suppose that the initial datum $\rho_{\epsilon}(\cdot,0)$ defined on $\Lambda_{N}$, and taking values in $[0,1]$ converges weakly as $\epsilon\to 0$ to $u_0(\cdot)\in L^{\infty}([-1,1],[0,1])$ in the sense
\begin{equation}\label{initial}
\lim_{\epsilon\to 0}\epsilon\sum_{x\in\Lambda_{N}}\rho_{\epsilon}(x,0)\phi(\epsilon x)=\int_{[-1,1]}u_0(r)\phi(r)dr,\quad\text{for every}\;\phi\in L^{\infty}([-1,1],\mathbb{R}).
\end{equation}
Then, there is $\rho(x,t)$, $x\in[-1,1]$, $t>0$ so that for any $t_1>t_0>0$
\begin{equation}\label{eq:hydro_conv}
\lim_{\epsilon\to 0}\sup_{x\in\Lambda_N}\sup_{t_0\leq t\leq t_1}|\rho_{\epsilon}(x,t)-\rho(\epsilon x,t)|=0,
\end{equation}
where the function $\rho(x,t)$  is the unique solution of the integral equation
\begin{eqnarray}\label{integralformofequation}
\rho(x,t)&=&\int_{[-1,1]}P_{t}(x,r')u_0(r')dr'+\frac{j}{2}\int_0^t P_{s}(x,1)(1-\rho(1,t-s)^{K}) \nonumber\\
&&-P_{s}(x,-1)(1-(1-\rho(-1,t-s))^{K})ds
\end{eqnarray}
and $P_{t}(x,r')$ is the density kernel of the semigroup (also denoted as $P_t$) with generator
$\Delta/2$, $\Delta$ the Laplacian in $[-1, 1]$ with reflecting, Neumann, boundary conditions.
\end{theorem}

\begin{remark}\label{rem:smoothness}
We note that the solution $\rho(x,t)$ of the previous equation is smooth. This has been observed in Remark 2.4 of \cite{de2011current}.
\end{remark}

\section{Results on space and space-time correlation $v$ estimates}\label{sec:def_space_and_space_time_v}In this section, we state the main results of the paper concerning estimates for the space and space--time $v$-functions. The two main results are Theorems~\ref{thm:vestimate2} and~\ref{thm:vestimate3}. In Theorem~\ref{thm:vestimate2}, we establish improved estimates for the space $v$-functions, sharpening the bounds obtained in Theorem~2.1 of \cite{de2012truncated}, while Theorem~\ref{thm:vestimate3} provides the corresponding estimates for the space--time $v$-functions. As discussed above, these quantities arise naturally in the analysis of correlations of the centered occupation variables: the space $v$-functions measure correlations at a fixed time, whereas their space--time counterparts describe correlations across different times.

We first introduce some notation. Recall  that we denote by $\mathbb P _{\mu^\epsilon}$ the probability measure in the Skorohod space denoted by $\mathcal D([0,T], \Omega_N)$ induced by the  Markov process $\{\eta_{t}\}_{t\geq{0}}$ and the initial probability measure $\mu^\epsilon$ and  $\mathbb E^\epsilon _{\mu^\epsilon}$ denotes the expectation with respect to $\mathbb P_{\mu^\epsilon}$ and that $
\rho^\epsilon_t(x)=\mathbb{E}^{\epsilon}_{\mu^\epsilon}[\eta_{t}(x)]$.  In the particular case where $\mu^\epsilon$  is the point mass at a configuration $\eta$, we write $\mathbb E^\epsilon _{\eta}$.

\begin{definition}\label{definitionvestimate}We fix an initial product measure $\mu^{\epsilon}$ on  $\Lambda_N$.
Suppose that the process $\{\eta_{t}\}_{t\geq 0}$ starts with $\mu_{\epsilon}$ and $\rho_{\epsilon}(t,x) $ is the solution of \eqref{eq:linearized}. We define the space  $v$-function as:
\begin{equation}\label{def: v function original}
v_{n}^{\epsilon}(\underline{x},t|\mu^{\epsilon}):=\mathbb{E}^{\epsilon}_{\mu^\epsilon}\left[\prod_{i=1}^{n}\left(\eta_{t}(x_i)-\rho_{\epsilon}(x_{i},t)\right)\right],\quad \underline{x}\in\Lambda_{N}^{n,\neq},n\geq1,
\end{equation}
where $\Lambda_{N}^{n,\neq}$ is the set of all sequences $(x_1,\dots,x_n)\in \Lambda_{N}^{n}$ with $x_i\neq x_j$. For brevity, we
shall write $v_{n}^{\epsilon}(\underline{x},t)\equiv v_{n}^{\epsilon}(\underline{x},t|\mu^{\epsilon})$. Note that in the particular case where $\mu^\epsilon$  is the point mass at a configuration $\eta$, we write $v_{n}^{\epsilon}(\underline{x},t|\eta)$. When $n=1$, $\underline{x}=x$ then 
\begin{equation} \label{eq:important_1}
v_{1}^{\epsilon}(x,t):=\rho_t^{\epsilon}(x)-\rho_{\epsilon}(x,t).
\end{equation}
When $\underline{x}=\und\emptyset$ is the empty configuration (i.e., no particles in the system), then 
\begin{equation}\label{v:emptyset}
v_0^{\eps}(\und \emptyset,t)=1.
\end{equation}
\end{definition}
We observe that as a consequence of  Remark \ref{rem:v_1_null} we have  $v^\eps_1(x,0)\equiv 0$. This will play an important role in what follows.

The notation $v_{n}^{\epsilon}(\underline{x},t)$ with $\und x \in
\La_N^{n,\neq}$ is adopted for convenience when working with the duality and the labeled process (Definition \ref{Def:labeled process} in Appendix \ref{app: derivation of II for space v}). However, $v_{n}^{\epsilon}(\underline{x},t)$ can equivalently be written as
$v_{n}^{\epsilon}(X,t)$, with $X$ a non-empty subset of $\La_N$ since it is symmetric with respect to permutations of the coordinates in $\underline x$. Throughout the paper, we will interchange between the two definitions whenever it simplifies the presentation. We first recall the estimate established in \cite{de2012truncated}. 

\begin{theorem}\cite [Theorem 2.1] {de2012truncated}\label{vestimate1}
Let  $K$ be a fixed  integer number such that $K\geq 2$ as given in  Section \ref{sec:model}. Then, there exists $\tau>0$ such that for any $0<c^*<\frac{1}{4(K+2)}$ the following holds: For any $\beta^*>0$ and for any positive integer $n$, there is a constant $c_n<\infty$, so that for every $\epsilon>0$, any initial product measure $\mu^{\epsilon}$,
\begin{equation}\label{v}
\sup_{\underline{x}\in\Lambda_{N}^{n,\neq}}|v_{n}^{\epsilon}(\underline{x},t|\mu^{\epsilon})|\leq
\begin{cases}
c_n(\epsilon^{-2}t)^{-c^* n}, & \text{for } t\leq\epsilon^{\beta^*}\\
c_n\epsilon^{(2-\beta^*)c^* n}, & \text{for } \epsilon^{\beta^*}\leq t\leq\tau\log\epsilon^{-1}.
\end{cases}
\end{equation}
\end{theorem} 
This estimate becomes of order $\epsilon$, or even smaller, when $n$ is sufficiently large. This is stated in the following corollary.

\begin{corollary}\label{cor:N*}
Let  $K\geq 2$ be a fixed  integer number,  $\zeta>0$ and $0<c^*<\frac{1}{4(K+2)}$. For $\beta^*<2$ and $\eps<1$, there exists $N^*$ positive integer (large enough) satisfying 
\begin{equation}\label{N*}
(2-\beta^*)c^*(N^*+1)>1+\zeta
\end{equation}
and thus
 \begin{equation}\label{est: v for large n}
 \sup_{\eps^{\beta^*}<t\leq T}\sup_{\underline{x}\in\Lambda_N^{n,\neq}}|v^{\epsilon}_{n}(\underline{x},t)|\leq c\epsilon^{1+\zeta},\quad n\geq N^*+1.
\end{equation} 
\end{corollary}
However, for small values of $n$, the exponent $c^*n$ is too small to yield the optimal order in $\epsilon$.

Our main contribution is to obtain sharper estimates for small values of $n$, while sacrificing the more delicate dependence on $n$ appearing above. More precisely, for times $0<t\leq T$, with $T>0$ fixed, we prove that the space and space-time $v$-functions are of order $\epsilon^{1-2\zeta}$ for $n=1,2$, of order $\epsilon$ for $3\leq n\leq K$, and of order $\epsilon^{1+\zeta}$ for $n\geq K+1$, where $\zeta>0$. Moreover, our estimates retain suitable time singularities so that they remain integrable after time integration, a property that is essential in the subsequent analysis. We emphasize that our results hold on a fixed time horizon $[0,T]$, whereas extending them to times of order $T\log\epsilon^{-1}$ would require additional arguments. Bounds of similar order are obtained for the space-time $v$-estimate. The first result on space $v$-functions can be stated as follows.

\begin{theorem}[Space correlation $v$-estimate]
\label{thm:vestimate2}
Let $u_0:[-1,1]\to[0,1]$ be a smooth initial profile for \eqref{eq:linearized}. Let $K\geq2$ be a fixed integer and let $\zeta>0$. Assume that the initial product measure $\mu^{\epsilon}$ satisfies
\begin{equation}\label{Ass1}
\rho^\epsilon_0(y)=\mathbb E^\eps_{\mu^{\epsilon}}[\eta(y)]
=u_0(\epsilon y), \qquad y\in\Lambda_N.
\end{equation}
Then, for every $0<t\leq T$,
\begin{equation}\label{newvestimate}
\sup_{\underline{x}\in\Lambda_{N}^{n,\neq}}
|v_{n}^{\epsilon}(\underline{x},t\mid\mu^{\epsilon})|
\lesssim
\epsilon^{1-2\zeta}\mathbf{1}_{\{n=1,2\leq K\}}
+\epsilon\mathbf{1}_{\{3\leq n\leq K\}}
+\epsilon^{1+\zeta}\mathbf{1}_{\{n\geq K+1\}}.
\end{equation}
\end{theorem}
\begin{remark}\label{rem:ind_time_space_v}
Above and in the following, $\lesssim$ denotes an inequality that is correct up to a
multiplicative universal constant, that is, independent of $N, \epsilon, t$. This means that the result above is not time dependent. This will be crucial ahead. 
\end{remark}
\begin{remark}\label{rem:v_1_null}
    From \eqref{Ass1}, we conclude that $\rho_0^\eps(y)=\rho_\eps(0,y)$ for all $y\in \Lambda_N.$
\end{remark}

In contrast to the weaker assumption on the initial state in \eqref{initial},  assumption \eqref{Ass1} is essential for the proofs of Theorems~\ref{thm:vestimate2} and~\ref{thm:vestimate3}. Indeed, both proofs require a refined control of the discrete gradient of the solution to \eqref{eq:linearized}, as established later in  Lemma \ref{lem: difference of rhos}. This estimate improves upon (3.8)--(3.9) in \cite{de2012truncated}; see also Proposition~\ref{prop:disc_grad_rho_eps}, i.e., 
 \begin{equation*}
 \sup_{x,y\in \La_N: |x-y|\leq 1}|\rho_\epsilon(x,t)-\rho_\epsilon(y,t)| \le \frac {c}{(\epsilon^{-2}t)^{1/2-\zeta}+1}
\quad \text{{\rm for any} $t\le \tau \log \epsilon^{-1}$}.
                \end{equation*}

Next we proceed to space-time $v$-functions  that we define as follows. \begin{definition}\label{def:space time correlations}
Let $\underline{x}=(x_1,\dots,x_n)\in \Lambda_{N}^{n,\neq}$ and $\underline{y}=(y_1,\dots,y_m)\in \Lambda_{N}^{m,\neq}$ with $n,m\geq 1$. For any initial product measure $\mu^{\epsilon}$ and times $0<s<r$, the space-time correlation $v$-function is given by 
\begin{equation}\label{def:space time corr fun}
v^{\epsilon,m,n}_{s,r}(\underline{y}, \underline{x}|\mu^{\epsilon}):=\mathbb{E}^{\epsilon}_{\mu^\epsilon}\left[\prod_{j=1}^{m}\bar\eta_s(y_j)\prod_{i=1}^{n}\bar\eta_r(x_i)\right],
\end{equation}
where  $\bar\eta_t(x):=\eta_t(x)-\rho_\eps(x,t)$, $t\geq 0$.
\end{definition}
 The subscript indicates the times $s,r$ involved in the space-time correlation $v$-function. The inputs $(\underline{y},\underline{x})$ are the positions of the $m$ particles at time $s$, denoted by  $\underline{y}=(y_1,\dots,y_m)$ and the positions of the $n$ particles at time $r$, denoted by  $\underline{x}=(x_1,\dots,x_n)$.
The superscript reflects the dependence on $\epsilon$, which may sometimes be omitted, and the number of particles at time $s$ and at time $r$, i.e., $m$ and $n$ respectively.

\begin{theorem}[Space-time correlation $v$-estimate]
	\label{thm:vestimate3}
Let  $K\geq 2$ be a fixed  integer number, $n,m\geq 1$ and $\zeta>0$. For any initial product measure $\mu^{\epsilon}$ satisfying \eqref{Ass1}, and times  $0<s<r\leq T$, there exists $\tilde\Delta>0$ (independent of $\eps$), such that when $r-s\leq \tilde\Delta$
\begin{equation}\label{bound:space_time_v_estimate1}
\sup_{\underline{y}\in\Lambda_{N}^{m,\neq}}\sup_{\underline{x}\in\Lambda_{N}^{n,\neq}}|v^{\epsilon,m,n}_{s,r}(\underline{y}, \underline{x}|\mu^{\epsilon})|\lesssim 
\begin{cases}
 \epsilon(r-s)^{-1/2},&n+m=2,3\\
 \epsilon^{1+\zeta}(r-s)^{-1+\zeta},&n+m\geq 4.
\end{cases}
\end{equation}
When $r-s> \tilde\Delta$,   
\begin{equation}\label{bound:space_time_v_estimate2}
\sup_{\underline{y}\in\Lambda_{N}^{m,\neq}}\sup_{\underline{x}\in\Lambda_{N}^{n,\neq}}|v^{\epsilon,m,n}_{s,r}(\underline{y}, \underline{x}|\mu^{\epsilon})|\lesssim
\begin{cases}
 \epsilon,&n+m=2,3\\
 \epsilon^{1+\zeta/2},&n+m\geq 4.
\end{cases}
\end{equation}
\end{theorem}
The proof of the theorem is given in Section \ref{sec:proof of space time}. The above estimate is desirable in the sense that it depends on the time difference $r - s$, which may be extremely small, on the order of $\epsilon$  or even smaller. Consequently, the factors $(r - s)^{-1/2}$ or $(r - s)^{-(1 - \zeta)}$ can become very large, and substituting their orders directly may lead to an overall excessively large bound. Therefore, when $s$ and $r$ are close, the estimate is expected to depend explicitly on the time difference $r - s$. 
On the other hand, when $r>s+\tilde\Delta$, $s,r$ are not close anymore and the bound can be independent of $r-s$.

\section{Proof of Theorem~\ref{thm:vestimate2}}\label{integral inequalities for space  v-function}

In this section, we prove Theorem~\ref{thm:vestimate2}. We organize the proof into five main steps, which isolate the essential ingredients of the argument and, at the same time, indicate how the method may be adapted to derive correlation estimates for other interacting particle systems.

\medskip

\noindent{\bf Discrete Evolution equation (Step 1.)} Starting from Kolmogorov's equation, we derive the evolution equation for the space $v$-functions. More precisely, for $X$ such that $|X|=n$,
\begin{equation}\label{formula: equation for v}
\frac{d}{dt} v_n^\eps(X,t)=\mathbb E^\eps_{\mu^\eps}\left[L_\eps
\prod_{x\in X}\bigl(\eta_t(x)-\rho_\eps(t,x)\bigr)\right]
+
\mathbb E^\eps_{\mu^\eps}\left[\frac{\partial}{\partial t}
\prod_{x\in X}\bigl(\eta_t(x)-\rho_\eps(t,x)\bigr)\right].
\end{equation}
Computing the right-hand side of \eqref{formula: equation for v}, we obtain the discrete evolution equation
\begin{equation}\label{10_discussion}
\frac{d}{dt} v_n^\eps(X,t)
=
\eps^{-2}(L_0v_n^\eps)(X,t)
+
(C_\eps v_n^\eps)(X,t),
\end{equation}
where $L_0v_n^\eps$ denotes the action of $L_0$ on $v_n^\eps(\cdot,t)$, viewed as a function of the particle configuration $X$. Moreover, $\eps^{-2}L_0$ is the generator of the symmetric stirring process, denoted by $(X(\lambda))_{\lambda\ge0}$, reflected at the boundary sites $\{-N,N\}$. The evolution equation \eqref{10_discussion} naturally separates the contribution of the stirring process from the interaction mechanisms encoded in the operator $C_\eps$. While the explicit formula for $C_\eps$ is given later and depends on the particular interacting particle system under consideration, the derivation of the evolution equation follows the same general principle. In the present model, $(C_\eps v_n^\eps)$ accounts for the following {\bf scenarios}:
\begin{itemize}
    \item[{\it Scenario 1.}] Two particles meet (i.e., occupy neighboring lattice sites), and one of them dies;
    \item[{\it Scenario 2.}] Two particles meet, and both die;
    \item[{\it Scenario 3.}] One or more particles reach the boundaries $I_{\pm}$, and either no new particles are created or new ones are born.
\end{itemize}
The explicit expression of $(C_\eps v_n^\eps)$ is derived later as we analyze the correlation functions.

\medskip

\noindent{\bf Integral solution (Step 2.)} 
Applying Duhamel's formula to the evolution equation yields
\begin{equation}
\label{14_for_discussion}
v_n^\eps(X,t)
=
\int_0^t d\lambda
\sum_Y
\PP_\eps(X\overset{\lambda}\rightarrow Y)
(C_\eps v_n^\eps)(Y,t-\lambda).
\end{equation}
because $v_n^\eps(X,t)=0$. This representation is a direct consequence of the evolution equation and forms the basis for the subsequent estimates. It is therefore available whenever an evolution equation of the form \eqref{10_discussion} can be established.

For the stirring process introduced in Appendix~\ref{app: derivation of II for space v}, we write, for configurations $X,Y\subset\Lambda_N$ with $|X|=|Y|$,
\[
\PP_\eps(X\overset{\lambda}\rightarrow Y)
=
\PP_\eps(X(\lambda)=Y\,|\,X(0)=X).
\]
In particular, if $X=\{x\}$ and $Y=\{y\}$, then
\[
\PP_\eps(X\overset{\lambda}\rightarrow Y)
=
P_\lambda^{(\eps)}(x,y),
\]
where $P_\lambda^{(\eps)}$ is defined in \eqref{a3.6}. Finally, we order the sites of $X$ arbitrarily as $\underline{x}=(x_1,\ldots,x_n)$ and identify
\[
v_n^\eps(\underline{x},t)=v_n^\eps(X,t).
\]

\medskip

{\bf Estimates for the interaction scenarios (Step 3.)} The scenarios (events) introduced in Step 1. are mainly controlled using tools developed in \cite{de2011current} and \cite{de2012truncated}. 
These methods rely on Gaussian kernel approximations for the aforementioned events, in particular by bounding the transition probability of the stirring process from a labeled configuration $\underline{x} = (x_1, \dots, x_n)$ to $\underline{y} = (y_1, \dots, y_n)$ by the product of the corresponding probabilities for $n$ independent random walks, each estimated via Gaussian kernels. 
Since the Gaussian kernels are further bounded by $\frac{1}{[\epsilon^{-2} \lambda]^{1/2}}$  after diffusive scaling, that is,
\begin{equation}\label{eq:psi_n}
\mathbf{P}_{\epsilon}(\underline{x} \xrightarrow{\lambda} \underline{y}) 
\leq 
\prod_{i=1}^{n} \sum_{j=1}^{n} P^{(\epsilon)}_{\lambda}(x_i, y_j)
\lesssim 
\frac{1}{[\epsilon^{-2}\lambda]^{n/2}} 
\qquad \text{(Liggett's inequality, \cite{liggettinteracting})},
\end{equation}
the product of more than one such kernels is not necessarily integrable over $[0, T]$, e.g., $\int_0^T d\lambda\frac{1}{[\epsilon^{-2}\lambda]^{n/2}}$, when $n\geq 2$. On the other hand, as observed in Remark~\ref{rem:imp_estimate}, a Gaussian kernel can be bounded by
$\frac{1}{[\epsilon^{-2}\lambda]^{1/2}+1}$. Applying this estimate to each Gaussian kernel in the product yields an integrable bound over $[0,T]$. However, the resulting estimate no longer preserves the desired order in~$\epsilon$. 
This presents a major challenge in proving that the space $v$-function is at least of order $\epsilon$ for $n \geq 3$, and, as we shall see later, at least of order $\epsilon$ for the space--time $v$-function as well. 
To overcome this difficulty, we estimate the possible scenarios in a different manner.
Depending on the number of particles that die and/or are born, the resulting integral inequality involves a $v$-function whose degree corresponds to the number of remaining particles. Therefore, the evolution equation for the space $v_{n}^{\epsilon}$- function includes terms involving $v_{j}^{\epsilon}$, where, according to the scenarios described above, $j = n-2, n-1$, or $j = n, n+1, \dots, n+K-1$. For this reason, to fully estimate all terms appearing in the integral solution of \eqref{formula: equation for v}, we need to consider the integral solutions of all $v_{j}^{\epsilon}$. In this way, all $v_{j}^{\epsilon},\;j\geq 1$ are estimated simultaneously, showing that as the degree of $v_{j}^{\epsilon}$ increases, the estimates are bounded by the order $\eps^{1+\zeta}$. This is because, when $n\geq K+1$, the scenarios contributing terms of order at most $\epsilon$ no longer appear in the integral representation of \eqref{formula: equation for v}. This comes from the actual calculation of \eqref{formula: equation for v}. Consequently, only the scenarios contributing terms of order $\epsilon^{1+\zeta}$ remain.
The structure of the proof goes as follows: we start with the analysis of the evolution equation of 
$v$-functions, then proceed with the estimation of the scenarios mentioned above and conclude with simultaneous estimation of the vector of all $v_j^{\eps}$'s. Although the estimates depend on the specific interaction mechanisms of the model, the overall strategy remains the same and can be applied to other interacting particle systems once the corresponding operator $C_\eps$ has been identified.
\medskip

\noindent{\bf Derivation of the integral inequalities (Step 4.)} In this step, we derive the integral inequalities for the $v$-functions by combining the estimates obtained for the different contributions in the integral representation discussed in Step~2. As explained above, these estimates rely primarily on Gaussian kernel bounds, which are independent of the particular particle configuration. This allows us to pass to the supremum over configurations and derive integral inequalities for quantities of the form
\[
\sup_{\underline{x}\in\Lambda_N^{\cdot,\neq}}
\bigl|v_{\cdot}^{\eps}(\underline{x},t)\bigr|.
\]
Integral inequalities for the space $v$-functions were already established in Section~5 of \cite{de2012truncated}. In the present work, however, several of the underlying scenarios require a different treatment in order to obtain the improved estimates sought here, particularly in the small-degree regime, as discussed in the Introduction. In the analysis that follows, we therefore focus on those terms for which a more refined argument is needed, while treating the remaining contributions along the lines of \cite{de2012truncated}.
\medskip

\noindent {{\bf Step 5.}
This is the final step of the proof, and its main idea goes back to \cite{de1989weakly}. Starting from the integral inequalities derived in Step~4, we rescale each $v$-function by the inverse of the power of $\epsilon$ corresponding to the estimate we aim to prove and then take the supremum over time. For example, when $3\leq n\leq K$, we consider the integral inequality for
\[
\epsilon^{-1}
\sup_{\underline{x}\in\Lambda_N^{n,\neq}}
|v_n^\epsilon(\underline{x},t)|
\]
and subsequently take the supremum over $0\leq t\leq T$.

Taking the time supremum is particularly useful, as it allows us to carry out the remaining time integrations in the terms inherited from the Duhamel representation. In this way, the integral inequalities of Step~4 are reduced to a system of inequalities for the rescaled quantities
\[
\Big(
\epsilon^{-1+2\zeta}\mathbf{1}_{\{n=1,2\}}
+\epsilon^{-1}\mathbf{1}_{\{3\leq n\leq K\}}
+\epsilon^{-1-\zeta}\mathbf{1}_{\{n\geq K+1\}}
\Big)
\sup_{0\leq t\leq T}
\sup_{\underline{x}\in\Lambda_N^{n,\neq}}
|v_n^\epsilon(\underline{x},t)|.
\]
We then show that these quantities are uniformly bounded in $\epsilon$, which yields the estimates stated in Theorem~\ref{thm:vestimate2}. 
\medskip

\noindent {\bf Applicability to other interacting particle systems.} The above argument also suggests how the method can be adapted to other interacting particle systems. For an exchange dynamics, the main model-dependent ingredient is the operator $C_\eps$, which determines the possible interaction mechanisms and, consequently, which degrees of the $v$-functions are coupled through the corresponding evolution equations. Once these contributions are identified, the same strategy can be followed: one derives the Duhamel representation, estimates separately the different terms generated by $C_\eps$ using suitable bounds for the transition probabilities of the underlying particle dynamics, and finally closes the resulting system of integral inequalities by estimating the relevant $v$-functions simultaneously. Thus, the method is applicable whenever the underlying dynamics admit sufficiently strong transition-kernel estimates and the interaction terms generated by $C_\eps$ can be organized into a system of integral inequalities that closes across the different correlation degrees.
\medskip

\subsection{An evolution equation for the space $v$-function (Step 1)}
We first recall, for example,  from Lemma 5.1 of \cite{de2012truncated}, the evolution equation for  space  $v$-functions, for $X$ such that $|X|=n$, given by \eqref{formula: equation for v}. We calculate its right-hand side, obtaining the discrete equation for  $v_n^\eps(X,t)$ given by \eqref{10_discussion}, i.e., 
$$\frac{d}{dt} v_n^\eps(X,t)=\eps^{-2} (L_0 v_n^\eps)(X,t)+(C_{\eps}v_n^\eps)(X,t)$$ 
We remind that $L_0v_n^\eps$ denotes the action of $L_0$  on  $ v_n^\eps(\cdot,t)$  as a  function of  $X$  with $X$ regarded as a particle configuration. The operator    $(C_{\eps}v_n^{\eps})$ is the result of a linear operator acting on the hierarchy of $v_n^{\eps}$, given as

 \begin{equation}
       \label{12aa}
 (C_{\eps}v_n^\eps)(X,t):=
\eps^{-2}(Av_n^\eps)(X,t)+\eps^{-1}(Bv_n^\eps)(X,t),
       \end{equation}
with $(Av_n^\eps)(X,t)=0$ if $|X|<2$, while for $|X|\ge 2$
\begin{equation}\label{11}
\begin{split}
(Av_n^\eps)(X,t) & = 
(A^1v_n^\eps)(X,t) +(A^2v_n^\eps)(X,t) 
 \\
&:=\sum_{x,y\in X}\mathbf 1_{\{|x-y|=1\}} 
\left[\rho_\eps(t,x)-\rho_\eps(t,y)\right] \left[v_{n-1}^\eps(X\setminus x,t)-v_{n-1}^\eps(X\setminus y,t)\right]
\\
&+\sum_{x,y\in X}\mathbf 1_{\{|x-y|=1\}} \left( -\frac 12 \left[ \rho_\eps(t,x)-\rho_\eps(t,y)
\right]^2 v_{n-2}^\eps(X\setminus(x\cup y),t)\right)
\end{split}
       \end{equation}
and $Bv_n^\eps=B_+v_n^\eps+B_-v_n^\eps$, with $(B_{\pm} v_n^\eps)(X,t)$ that can be written as
     \begin{equation}
       \label{12}
       \begin{split}
(B_{\pm}v_n^\eps)(X,t)&:=\sum_{q=1}^K\sum_{ Z'_q\subset I_{\pm}}b_\eps\Big([X\cap I_{\pm}],Z'_q,t\Big) \\
&\hskip2cm\times v_{n-1+q}^{\eps}\Big(X\setminus [X\cap I_{\pm}]\cup Z'_q,t\Big)\mathbf{1}_{\{|X\cap I_{\pm}|=1\}} \mathbf{1}_{\{|Z'_q|=q\}}\,\\
&\hskip0.5cm+\sum_{p=2}^{\min\{{|X|},K\}}\sum_{q=0}^K\sum_{ Z'_q\subset I_{\pm}}\mathbf{1}_{\substack{\{|X\cap I_{\pm}|=p,\},\\ X\neq I^{|X|}_\pm }} \mathbf{1}_{\{|Z'_q|=q\}}\,
b_\eps\Big([X\cap I_{\pm}],Z'_q,t\Big)\\
&\hskip2cm\times v_{n-p+q}^{\eps}\Big(X\setminus [X\cap I_{\pm}]\cup Z'_q,t\Big)\\
&\hskip 0.5cm+\mathbf{1}_{\{X= I_{\pm}^{|X|}\}}b_\eps\Big({I_{\pm}^{|X|}},\emptyset,t\Big)
\end{split}
       \end{equation}
       with 
\begin{equation}\label{In}
I^l_{+}=\{N-(l-1),\dots, N\}\;\textrm{ and } \;I^l_{-}=\{-N,\dots, -N+(l-1)\},\quad 1\leq  l\leq K,
\end{equation}
being the last $l$ sites in the right/left boundary, and 
the coefficients $b_\eps(Z,Z',t)$  satisfy the following properties:
       \begin{equation}
       \label{12ab}
 \text{for any integer $M$:}\,\, \sup_{t,|Z|\leq M, |Z'|\leq M}|b_\eps(Z,Z',t)| <\infty.
       \end{equation}
The last term in \eqref{12} is displayed separately from the others since it will be treated with a different strategy  as it does not involve any order of a $v$-function. Observe that in \eqref{12},  $p=|X\cap I_{\pm}|\in\{1,\dots, n\}$,  represents the number of particles that go to the boundary and die,  and $q=|Z'_q|\in\{0,\dots, K\}$, represents the number of particles that are born at the boundary.  Therefore, the first two terms in \eqref{12} distinguish the following two cases:
\begin{itemize}
\item[$p=1$:] Exactly one particle reaches the boundary and is removed, while $q\in\{1,\dots,K\}$ new particles are introduced.
\item[$p\geq 2$:] At least two particles reach the boundary and are removed. In this case, $q\in\{0,\dots,K\}$, allowing for the possibility that no new particles are born.
\end{itemize}
This distinction will be useful in the subsequent analysis, where the two cases are treated separately.
 \begin{remark}[\bf Interpretation of the operator $C_{\eps}v_n^\eps$]\label{rem:interpretation} The operator $(C_{\epsilon}v_n^\eps)$ describes how particles move within the system. Specifically, $Av_n^\eps$ represents two different scenarios described in the operators $A^1v_n^\eps$ and $A^2v_n^\eps$: either two particles meet - this is indicated by the indicator function $\mathbf 1_{\{|x-y|=1\}}$  - and either $x$ or $y$ dies,  or  both $\{x,y\}$ die,  and this is  represented by $v^{\epsilon}_{n-1}(X\setminus x,t)$ or $v^{\epsilon}_{n-1}(X\setminus y,t)$ or $v^{\epsilon}_{n-2}(X\setminus(x\cup y),t)$, respectively. The operator $B_{\pm}v_n^\eps$ accounts for cases where one or more particles reach the boundary  $I_{\pm}$, i.e., the particles in $X\cap I_{\pm}=Z_p$, die, and possibly result in the birth of 0 or more particles, i.e., the particles in $Z_q'$. This is described by $v_{n-p+q}\left(X\setminus [X\cap I_{\pm}]\cup Z_q',t\right)$ subject to the constraints \eqref{12ab}. 
        \end{remark}

\subsection{Derivation of the integral solution (Step 2)} Given the evolution equation for the space $v$-function we use the stirring process defined in Section \ref{app: derivation of II for space v}.
While in \cite{de2012truncated} the integral form of the solution  of \eqref{10_discussion} is \eqref{14_for_discussion}, because the initial condition starts at time 0 and is given by $v_n^\eps(X,0)=0$, i.e., $$v_n^\eps(X,t)
=
\int_0^t d\lambda
\sum_Y
\PP_\eps(X\overset{\lambda}\rightarrow Y)
(C_\eps v_n^\eps)(Y,t-\lambda),$$ here we present the integral form for any starting time $0< t^*<t$.  We denote by $v_n^\epsilon(X, t; t^*)$ the space $v$-function starting from time $t^*$, 
and by $v_n^\epsilon(X, t)$ the one starting from $t = 0$. As mentioned in the beginning of the section, we order arbitrarily the sites of $X$ as $\und x=(x_1,\dots,x_n)$ (which is a labeled configuration now, see definition in Appendix \ref{app: derivation of II for space v}) and set $v_n^{\epsilon}(\und x,t) := v_n^{\epsilon}(X,t)$ or $v_n^{\epsilon}(\und x,t;t^*) := v_n^{\epsilon}(X,t, t^*)$. In turn, ${\bf E}_{\epsilon, X}$ is written as ${\bf E}_{\epsilon, \underline{x}}$. By the evolution equation \eqref{10_discussion} and Duhamel's formula, we get
  \begin{equation}
         \label{14*}
v_n^\eps(\und x,t;t^*)={\bf E}_{\epsilon,\und x}\left[v_n^\eps(\und x(t-t^*),t^*)\right]+\int_{0}^{t-t^*}\,d\lambda
{\bf E}_{\epsilon,\und x}\left[(C_{\eps}v_n^\eps)(\und x(\lambda),t-\lambda)\right]
         \end{equation}
         which, by taking into account all the possible states for the stirring process, is the same as 
         \begin{equation}
         \label{14}
v_n^\eps(\und x,t;t^*)=\sum_{\und y\in \Lambda_N^{n,\neq}}\PP_\eps(\und x\overset{t-t^*}\rightarrow \und y)v_n^\eps(\und y,t^*)+\int_{0}^{t-t^*}\,d\lambda
\sum_{\und y\in \Lambda_N^{n,\neq}}\PP_\eps(\und x\overset{\lambda}\rightarrow \und y)(C_{\eps}v_n^\eps)(\und y,t-\lambda).
         \end{equation}
Recall that $\Lambda_{N}^{n,\neq}$ denotes the set of all sequences $(x_1, \dots, x_n) \in \Lambda_N^n$ with $x_i \neq x_j$ for all $i \neq j$, i.e., all coordinates of $\underline{y}$ are distinct. For simplicity of notation, in all subsequent summations we will write $\sum_{\underline{y}}$, which is to be understood as $\sum_{\underline{y} \in\Lambda_N^{n,\neq}}$.  We also note that the dependence on $\underline{x}$ in ${\bf E}_{\epsilon, \underline{x}}$ will be omitted whenever it is clear from the context. Similarly, we omit $t^*$ from the notation $v_n^\epsilon(\und x, t; t^*)$ and write $v_n^\epsilon(\und x, t)$ whenever no confusion arises. 
      
  \begin{remark}\label{rem:In}
After computing \eqref{formula: equation for v}, we observe that when $n\leq K$ and $X$ coincide with $\{N-(l-1),\dots, N\}$ (or $\{-N,\dots, -N+(l-1)\}$), i.e., when all particles reach  the last $l$ sites to the right in the boundary and die, one of the possible scenarios is that no new particles are born. This occurs when $l=n$ and  $X=I^l_\pm=I^n_\pm$ (and thus $X\setminus [X\cap I_{\pm}]=\emptyset$) and $Z_q'=\emptyset$ (and thus $q=0$  in which case the term simplifies to $b\Big(I^l_\pm,\emptyset,t\Big)$ since $v_{l-p+q}^{\eps}\Big(X\setminus I^l_\pm\cup \emptyset,t\Big)=v_{l-l+0}^{\eps}\Big(\emptyset,t\Big)=1$ by \eqref{v:emptyset}. This scenario is represented by $\mathbf{1}_{\{X= I_{\pm}^l\}}b_\eps\Big(I_{\pm}^l,\emptyset,t\Big)$
  and is exactly the third term in \eqref{12}. 
  \end{remark}

We recall \eqref{12aa} and we bound the rightmost term in \eqref{14} by the sum of the next two terms
\begin{equation}
         \label{14_star3}
\int_{0}^{t-t^*}\,d\lambda\sum_{\underline{y}}\mathbf{P}_{\epsilon}(\underline{x}\overset{\lambda}\rightarrow \underline{y}) \epsilon^{-1}|(B_{\pm}v_n^\epsilon)(\underline{y},t-\lambda)|
         \end{equation}
and 
\begin{equation}
         \label{14+}
\int_{0}^{t-t^*}\,d\lambda\sum_{\underline{y}}\mathbf{P}_{\epsilon}(\underline{x}\overset{\lambda}\rightarrow \underline{y}) \epsilon^{-2}\left(|(A^1 v^\epsilon_n)(\underline{y},t-\lambda)|+|(A^2 v^\epsilon_n)(\underline{y},t-\lambda)|\right).
         \end{equation}
Following \eqref{14}, we now consider separately the contributions associated with the operators $A^1$ and $A^2$. Recall that these operators correspond to the two possible scenarios occurring when two particles become neighbors: either one of the two particles is removed, or both particles are removed. Accordingly, the resulting terms involve $v$-functions of degree $n-1$ and $n-2$, respectively. More precisely, for $s\in[0,T]$, $n\geq 2$, $i,j\in\{1,\dots,n\}$, and $0\leq\lambda<s$, we bound
\begin{equation}\label{eq:A1}
|\eps^{-2}(A^1 v^\epsilon_n)(\underline{y},s-\lambda)|
\leq
\sum_{i,j}\gamma_{0}^{\eps}(i,j,\lambda,s-\lambda)
\left(
v_{n-1}^{\eps}(\underline{y}^{(i)},s-\lambda)
-
v_{n-1}^{\eps}(\underline{y}^{(j)},s-\lambda)
\right),
\end{equation}
whereas
\begin{equation}\label{eq:A2}
|\eps^{-2}(A^2 v^\epsilon_n)(\underline{y},s-\lambda)|
\leq
\sum_{i,j}\gamma_{2}^\eps(i,j,\lambda,s-\lambda)
\left|
v^{\epsilon}_{n-2}(\underline{y}^{(i,j)},s-\lambda)
\right|.
\end{equation}
Here, the coefficients $\gamma_{0}^{\eps}$ and $\gamma_{2}^{\eps}$ encode the contribution arising from the discrete variation of the profile $\rho_\eps$ between the positions of the two neighboring particles. They are defined by
\begin{equation}\label{19.1}
\gamma_{0}^{\eps}(i,j,\lambda,s-\lambda)
:=
\mathbf{1}_{\{|x_i(\lambda)-x_j(\lambda)|=1\}}
\eps^{-2}
\left|
\rho_\eps(x_i(\lambda),s-\lambda)
-
\rho_\eps(x_j(\lambda),s-\lambda)
\right|
\end{equation}
and
\begin{equation}\label{19.2}
\gamma_{2}^{\eps}(i,j,\lambda,s-\lambda)
:=
\gamma_0^\epsilon(i,j,\lambda,s-\lambda)
\left|
\rho_\eps(x_i(\lambda),s-\lambda)
-
\rho_\eps(x_j(\lambda),s-\lambda)
\right|.
\end{equation}
Thus, the contribution of $A^1$ couples the degree-$n$ correlation function to correlations of degree $n-1$, while that of $A^2$ involves correlations of degree $n-2$. The estimates of the coefficients $\gamma_0^\eps$ and $\gamma_2^\eps$ will therefore play an important role in controlling these two contributions.
      
\subsubsection{Analysis of \eqref{14_star3}}By using  Lemma 5.2 in \cite{de2012truncated}, it is straightforward to show that: For any $n$  there is a constant $c$ so that for any $\underline{x}\in\La_N^{n,\neq}$, and any $\lambda<t \leq \log \epsilon^{-1}$, \eqref{14_star3}  is bounded by
\begin{equation}\label{13}
              \begin{split}
&\sum_{u=\pm}\sum_{j=1}^{n}\sum_{\substack{\underline{z}'\subset I_{u},\\\und z'\neq\emptyset}}{\psi_1^\eps(\lambda)} \mathbf{E}_\epsilon \Big[ |v_{n-1+q}^\epsilon\big(\underline{x}^{(j)}(\lambda)\cup\underline{z}',t-\lambda\big)| \Big]\\
&+\sum_{u=\pm}\sum_{\substack{J\subset \{1,\dots,n\},\\2\leq p\leq \min\{n,K\}}}\sum_{\underline{z}'\subset I_{u}} {\psi_p^\eps(\lambda)}\mathbf{E}_\epsilon \Big[  |v_{n-p+q}^\epsilon\big(\underline{x}^{(J)}(\lambda)\cup\underline{z}',t-\lambda\big)|\Big]\\
&+\epsilon^{-1}\sum_{u=\pm}\mathbf{E}_\epsilon \Big[{\bf 1}_{\{\und x(\lambda)=I_u^n\}}\Big]{\bf 1}_{\{n\leq K\}}
\end{split}
      \end{equation}
where $\underline{x}^{(J)}$  is the configuration obtained by erasing from $\und x$ all $x_j$, $j\in J$ and  $p=|J|$. We note that when $J=\{1,\dots, n\}$, $\underline{x}^{(J)}=\emptyset$ . Also \begin{equation}\label{def_psi_n}
  \psi_p^\eps(\lambda)= 
\frac{\eps^{-1}}{[\epsilon^{-2}\lambda]^{p/2} + 1}.\end{equation} 
Note that the indicators $\mathbf{1}_{\{\emptyset\neq\underline{x}^{(J)}(\lambda)\subset I_\pm^c\}}$ appearing in Lemma 5.2 in \cite{de2012truncated} do not appear anymore  because  we   simply bound them by 1. The last term in \eqref{13}, namely
$\epsilon^{-1}\mathbf{E}_\epsilon
\big[{\bf 1}_{\{\underline{x}(\lambda)=I_\pm^n\}}\big]$,
can be estimated as in Lemma~5.2 of \cite{de2012truncated}; more precisely, for $n\leq K$,
\begin{equation}\label{est:dangerous_boundary_term}
\epsilon^{-1}\mathbf{E}_\epsilon
\Big[{\bf 1}_{\{\underline{x}(\lambda)=I_u^n\}}\Big]
\lesssim \psi_n^\eps(\lambda).
\end{equation}
However, a different treatment of this term will be required at several points in the paper. We therefore retain its more general form at this stage and estimate it according to the specific needs of each subsequent argument. For $t^*>0$, the integral solution \eqref{14} can be estimated as 
  \begin{equation}
              \label{13.1.111_p=1}
              \begin{split}
&|v_n^{\epsilon}(\underline{x},t)|\lesssim \sum_{\und y}\PP_\eps(\und x\overset{t-t^*}\rightarrow \und y)|v_n^\eps(\und y,t^*)|\\
&\hspace{1cm}+\int_{0}^{t-t^*} d\lambda \sum_{u=\pm}\sum_{j=1}^{n}\sum_{\substack{\underline{z}'\subset I_{u},\\\und z'\neq\emptyset}}{\psi_1^\eps(\lambda)} \mathbf{E}_\epsilon \Big[ |v_{n-1+q}^\epsilon\big(\underline{x}^{(j)}(\lambda)\cup\underline{z}',t-\lambda\big)| \Big]\\
&\hspace{1cm}+\int_{0}^{t-t^*} d\lambda\sum_{u=\pm}\sum_{\substack{J\subset \{1,\dots,n\},\\2\leq p\leq \min\{n,K\}}}\sum_{\underline{z}'\subset I_{u}} {\psi_p^\eps(\lambda)}\mathbf{E}_\epsilon \Big[  |v_{n-p+q}^\epsilon\big(\underline{x}^{(J)}(\lambda)\cup\underline{z}',t-\lambda\big)|\Big]\\
&\hspace{1cm}+\int_{0}^{t-t^*} d\lambda\epsilon^{-1}\sum_{u=\pm}\mathbf{E}_\epsilon \Big[{\bf 1}_{\{\und x(\lambda)=I_u^n\}}\Big]{\bf 1}_{\{n\leq K\}}\\
&\hspace{1cm}+\mathbf 1_{\{n\geq 2\}}\int_{0}^{t-t^*} d\lambda\;\bigg(\mathbf{E}_\epsilon\Big[\epsilon^{-2}\left\{|(A^1 v^{\epsilon}_{n})(\underline{x}(\lambda),t-\lambda)|+|(A^2 v^{\epsilon}_{n})(\underline{x}(\lambda),t-\lambda)|\right\} \Big].
\end{split}
  \end{equation}
 Here $A^1$ and $A^2$ are defined in \eqref{11}, \eqref{eq:A1} and \eqref{eq:A2}, $p=|J|$ and $q=|\und z'|$. Recall that when $t^*=0$, the first term on the right-hand side of
\eqref{13.1.111_p=1} vanishes, and hence the integral solution consists only of
the integral terms over $[0,t]$.
\begin{remark}[{\bf A few comments on \eqref{13.1.111_p=1}}]\label{A few comments} We discuss a few important points regarding \eqref{13.1.111_p=1}:
 \begin{itemize}
 \item We have distinguished the scenario that one particle  in $\und x$ reaches the boundary, dies and a new particle is born from all the other scenarios in $B_+$. The former scenarios have been retained in their original form as defined in \eqref{12} while the latter ones have been estimated as in \eqref{13}.
\item This decomposition is not essential for proving Theorem~\ref{thm:vestimate2}. 
As shown in Section~\ref{proof of main theorem}, the last three terms in~\eqref{13.1.111_p=1} 
are estimated using Liggett’s inequality together with Gaussian kernel bounds 
(see Lemma~\ref{lem: bound T1} and the estimates in Appendix~\ref{secN6}). 
Specifically, these terms yield an estimate of the form  {$\psi_p^\eps(\lambda)$} for $ p \ge 1$.
\end{itemize}
 \end{remark}

 \subsubsection{Analysis of \eqref{14+}} While \eqref{14+} accounts for both scenarios in which two particles meet--either both particles die, as described by the operator $A^2$, or only one of them dies, as given by $A^1$--we focus here on the latter, as it requires further analysis.
 
 \smallskip
 \noindent{\it Two particles meet and only one dies.} We now turn our attention to the scenario that two particles meet and only one dies. This is translated in the operator $A^1 v_n^\eps$, whose  definition is  given in  \eqref{11}.
 Note that $A^1$  depends on the discrete gradient of a $v$-function with one less particle, i.e., 
\begin{equation}\label{den:d_intial}d^{\epsilon}_{n-1}(\underline{x}^{(i)},\underline{x}^{(j)},t):=v_{n-1}^\eps(\underline{x}^{(i)},t)-v_{n-1}^\eps(\underline{x}^{(j)},t),
\end{equation}
where $\underline{x}^{(i)}$ and $\underline{x}^{(j)}$ are the configuration obtained by erasing from $\und x$ the $i$-th and the $j$-th particles, respectively. Recalling the formula  \eqref{14*} for the difference of the  $v$-functions, and by enlarging the space so that we can take the same expectation for the two $v$-functions, for $\und x=(x_1,\dots,x_{n+1})$, we write 
   \begin{equation}\label{DFD}
   \begin{split}
&d^{\epsilon}_{n}(\underline{x}^{(i)},\underline{x}^{(j)},t)={\bf E}_{\epsilon}\left[d^{\epsilon}_{n}(\underline{x}^{(i)}(t-t^*),\underline{x}^{(j)}(t-t^*),t^*)\right]\\
&\hskip1cm+\int_{0}^{t-t^*}\,d\lambda \,
{\bf E}_{\epsilon}\left[(C_{\eps}v_{n}^\eps)(\underline{x}^{(i)}(\lambda),t-\lambda)-(C_{\eps}v_{n}^\eps)(\underline{x}^{(j)}(\lambda),t-\lambda)\right].
         \end{split}\end{equation}
By introducing the stopping time $\tau_{i,j}$ defined as the first time $\tau>0$
when  
\begin{itemize}
\item[(i)] $|x_i(\tau)-x_j(\tau)|=1$  and \item[(ii)] at $\tau$ there is a mark (either active or  passive) between  $x_i(\tau)$ and $x_j(\tau)$ (see Definition~\ref{def:stoptime} and  Appendix~\ref{app: derivation of II for space v} and Section~5 of \cite{de2012truncated} for further details),
\end{itemize}
and using the fact that the expectations above act on antisymmetric functions, an argument analogous to that of Lemma~\ref{lem:17.3} or Lemma~4.3 of \cite{de2012truncated}, yields the following estimate for $n\geq 2$: 
\begin{eqnarray*}
&&|d^{\epsilon}_{n}(\underline{x}^{(i)},\underline{x}^{(j)},t)| \leq  {\bf E}_{\epsilon}\left[1_{\tau_{i,j}\geq \frac{t-t^*}{2}}|d^{\epsilon}_{n}(\underline{x}^{(i)}(t-t^*),\underline{x}^{(j)}(t-t^*),t^*)|\right]\nonumber\\ 
&&+\int_{0}^{t-t^*} ds\; \mathbf{E}_\epsilon \Big[\mathbf 1_{\{\tau_{i,j}\geq \frac s2\}}\sum_{\underline y} \Big\{\mathbf{P}_\epsilon\Big(\underline x^{(j)}\Big(\tfrac{s}{2}\Big)\overset{s/2}\rightarrow \underline y\Big)
+\mathbf{P}_\epsilon\Big(\underline x^{(i)}\Big(\tfrac{s}{2}\Big)\overset{s/2}\rightarrow \underline y\Big)\Big\}  |(C_\epsilon v_n)(\underline y,t-s)| \Big].
                   \end{eqnarray*}
\noindent We estimate the second term on the right-hand side of the display above using the same argument as in the derivation of \eqref{13.1.111_p=1}. This yields, for $n\geq1$,
\begin{equation}\label{17.3.3}
\begin{split}
& |d^{\epsilon}_{n}(\underline{x}^{(i)},\underline{x}^{(j)},t;t^*)|  \lesssim
\beta^{\eps}_{n}(\underline{x}^{(i)},\underline{x}^{(j)},t;t^*)\\
&\hskip 2cm+ \mathbf 1_{\{n\geq 2\}}\int_{0}^{t-t^*} d\lambda\;\epsilon^{-1}\sum_{\und y\in\{\underline{x}^{(i)},\underline{x}^{(j)}\}}\mathbf{E}_{\epsilon,\und y} \Big[\mathbf 1_{\{\tau_{i,j}>\frac \lambda2\}}{\bf 1}_{\und x(\lambda)=I^{n}_\pm}\Big],
\end{split}
\end{equation}
where for $\und x=(x_1,\dots,x_{n+1})$
\begin{equation}\label{17.3.3*}
\begin{split}
 \beta^{\eps}_{n}(\underline{x}^{(i)},&\underline{x}^{(j)},t;t^*):=
\PP_\eps\left(\tau_{i,j}>\frac{t-t^*}{2}\right)\sup_{\und x\in\Lambda_N^{n,\neq}}|v_{n}^\eps(\und x,t^*)|{\bf 1}_{\{t^*>0\}}\\
&+\int_{0}^{t-t^*} d\lambda \sum_{u=\pm}\sum_{j=1}^{n}\sum_{\substack{\underline{z}'\subset I_{u},\\\und z'\neq\emptyset}}{\psi_1^\eps(\lambda)} \mathbf{E}_\epsilon \Big[\mathbf 1_{\{\tau_{i,j}>\frac \lambda2\}} |v_{n-1+q}^\epsilon\big(\underline{x}^{(j)}(\lambda)\cup\underline{z}',t-\lambda\big)| \Big]\\
&+\int_{0}^{t-t^*} d\lambda\sum_{u=\pm}\sum_{\substack{J\subset \{1,\dots,n\},\\2\leq p\leq \min\{n,K\}}}\sum_{\underline{z}'\subset I_{u}} {\psi_p^\eps(\lambda)}\mathbf{E}_\epsilon \Big[ \mathbf 1_{\{\tau_{i,j}>\frac \lambda2\}} |v_{n-p+q}^\epsilon\big(\underline{x}^{(J)}(\lambda)\cup\underline{z}',t-\lambda\big)|\Big]\\
&+ \mathbf 1_{\{n\geq 3\}}\int_{0}^{t-t^*} d\lambda\;\mathbf{E}_\epsilon\Big[ \mathbf 1_{\{\tau_{i,j}>\frac \lambda2\}}\epsilon^{-2}|(A v^{\epsilon}_{n})(\underline{x}^{(i)}(\lambda),t-\lambda)| \Big]\\
&+ \mathbf 1_{\{n\geq 3\}}\int_{0}^{t-t^*} d\lambda\;\mathbf{E}_\epsilon\Big[ \mathbf 1_{\{\tau_{i,j}>\frac \lambda2\}}\epsilon^{-2}|(A v^{\epsilon}_{n})(\underline{x}^{(j)}(\lambda),t-\lambda)| \Big]
\end{split}
\end{equation}
 The validity of \eqref{17.3.3} follows from the observation that, for every $t\geq0$, the law of the process starting from $\underline{x}^{(i)}$ (that is, with particle $i$ removed) coincides with the marginal at time $t$ of the process started from the full configuration, after integrating out the position of particle $i$. Due to the indicator present in the second term of the right-hand side of \eqref{17.3.3}, when $\und x=(x_1,x_2)$, the following inequality holds
    \begin{equation}\label{17.3.3,n=1}
    |d^{\epsilon}_{1}(\underline{x}^{(1)},\underline{x}^{(2)},t;t^*)|  \lesssim \beta^{\eps}_1(\underline{x}^{(1)},\underline{x}^{(2)},t;t^*).
    \end{equation}
The last line of \eqref{13.1.111_p=1}  takes the form: 
 \begin{equation}\label{13.1.111 with gammas1}
 \begin{split}
&\mathbf{1}_{\{n\geq 2\}}\sum_{i,j}\int_{0}^{t-t^*} d\lambda\;\mathbf{E}_\epsilon\Big[ \gamma_{0}^\eps(i,j,\lambda, t-\lambda)\beta_{n-1}^{\eps}(\underline{x}^{(i)}(\lambda),\underline{x}^{(j)}(\lambda),t-\lambda,{t^*})\Big]\\
&\;\;\;\;\;\;+\mathbf{1}_{\{n\geq 3\}}\sum_{i,j}\int_{0}^{t-t^*} d\lambda\mathbf{E}_\epsilon\left[\gamma_{1}^{\eps, n-1}(i,j,\lambda, t-\lambda)\right]\\
&\;\;\;\;\;\;+\mathbf{1}_{\{n\geq 2\}}\sum_{i,j}\int_{0}^{t-t^*} d\lambda\;\mathbf{E}_\epsilon\Big[ \gamma_{2}^\eps(i,j,\lambda, t-\lambda)|v^{\epsilon}_{n-2}(\underline{x}^{(i,j)}(\lambda),t-\lambda)|\Big].
\end{split}
  \end{equation}
where    $\gamma_0^\eps (i,j,\lambda,t-\lambda)$ is defined in \eqref{19.1} and 
\begin{equation}\label{19.1*}
\gamma_{1}^{\eps,n-1}(i,j,\lambda, t-\lambda):=
 \eps^{-1}\gamma_{0}^{\eps}(i,j,\lambda, t-\lambda)\int_{0}^{t-\lambda}d\tau\sum_{\und y=\underline{x}^{(i)}(\lambda),\,\underline{x}^{(j)}(\lambda)}{\bf E}_{\epsilon,\underline{y}}\left[\mathbf 1_{\{\tau_{i,j}>\frac \tau2\}}{\bf 1}_{\underline{x}(\tau)=I_\pm^{n-1}}\right].
\end{equation}  
We now introduce the following quantities involved in the last two lines of integral inequality for $\beta_n^{\eps}$ given in \eqref{17.3.3*} and estimated in the next subsection.
\begin{definition}\label{def:tilde_gammas}
Fix $t\in[0,T]$. Let  $n\geq 2$, $i,j\in \{1,\dots,n\}$ and $0<\lambda<t$.  We define
\begin{eqnarray}
              \label{19.1**}              \hspace{-1cm}\widetilde\gamma_{0}^{\eps}(i,j,\lambda, t-\lambda)&=&
    \mathbf 1_{\{\tau_{i,j}>\frac \lambda2\}}\gamma_{0}^{\eps}(i,j,\lambda, t-\lambda)
             \end{eqnarray} 
             \begin{eqnarray}
              \label{19.2*}
  \widetilde\gamma_{2}^{\eps}(i,j,\lambda, t-\lambda)&=&
    \mathbf 1_{\{\tau_{i,j}>\frac \lambda2\}}\gamma_{2}^{\eps}(i,j,\lambda, t-\lambda)
             \end{eqnarray}
            
             and for  $1\leq p\leq \min\{n,K\}$
               \begin{equation}
              \label{19.3*}
   \widetilde\gamma_{p}^{\eps}(i,j,\lambda)=\mathbf 1_{\{\tau_{i,j}>\frac \lambda2\}}\psi_{p}^{\eps}(\lambda).
             \end{equation}
\end{definition}

\subsection{Estimation of the right-hand side of \eqref{14} (Step 3) }  \label{sec:gammas}
In this step, we present the strategy of estimating each scenario mentioned in Step 1, i.e., 
\begin{itemize}
    \item[{\bf Scenario 1.}] Two particles meet (i.e., occupy neighboring lattice sites), and one of them dies;
    \item[{\bf Scenario 2.}] Two particles meet, and both die;
    \item[{\bf Scenario 3.}] One or more particles reach the boundaries $I_{\pm}$, and either no new particles are created or new ones are born.
\end{itemize}
 through the estimation of $\gamma^\eps_{0}$, $\gamma_{1}^{\eps,n}$ and $\gamma_{2}^{\eps}$, we estimate their expectations which play a key role in deriving the integral inequalities for the space $v$-function. Throughout this section, we repeatedly use the following key remark  and Lemma \ref{lem: difference of rhos} in deriving the estimates.

\begin{lemma}\label{lem:gammas}
Let  $i,j\in \{1,\dots,n\}$, $\zeta>0$,  $\eps<1$ and $\und x=(x_1,\dots,x_n)$. Let $\gamma^\epsilon_0$,  $\gamma^\epsilon_{2}$ and  $\gamma_{1}^{\epsilon,n-1}$as  defined in \eqref{19.1}--\eqref{19.2} and \eqref{19.1*}. Under \eqref{Ass1}, for any {$0\leq \lambda\leq t$}, the following estimates hold:

\begin{equation}\label{est:gamma_0}
\mathbf{E}_{\epsilon,\und x} [\gamma^{{\epsilon}}_{0}(i,j,\lambda, t-\lambda)]\lesssim {\phi_0(\lambda)}:=\frac{1}{\lambda^{1/2}}+\frac{1}{\lambda^{1/2+\zeta}},
\end{equation} 
\begin{equation}\label{est:gamma_2}
\mathbf{E}_{\epsilon,\und x} [\gamma_{2}^{{\epsilon}}(i,j,\lambda, t-\lambda)]\lesssim \,{\phi^\epsilon_2(\lambda):=} \eps\left( \frac{1}{\lambda^{1/2}}+\frac{1}{\lambda^{3/4}}\right).
\end{equation}

 Moreover, for $n\geq 3$,
\begin{equation}\label{est:gamma_1}
\mathbf{E}_{\epsilon,\und x} [\gamma^{{\epsilon},n-1}_1(i,j,\lambda, t-\lambda)]\lesssim {\phi^{\epsilon,n-1}_1(\lambda)}:=
\begin{cases}
\eps
\phi_0(\lambda), & n= 3\\[6mm]
\frac{\epsilon^{-4\zeta}}{[\eps^{-2}\lambda]+1},&n\geq 4.
\end{cases}
\end{equation}
\end{lemma}

\begin{proof}
First, we estimate \eqref{est:gamma_0}, temporarily setting aside the factor $\eps^{-2}$. By Lemma \ref{lem: difference of rhos} and Liggett's inequality \eqref{Liggett_new}, we bound $\mathbf{E}_\epsilon [\gamma_{0}^\eps(i,j,\lambda, t-\lambda)]$ from above by a constant times:
\begin{equation*}
\begin{split}
&{\bf E}_{\epsilon,\und x}\left[\mathbf{1}_{\{|x_i(\lambda)-x_j(\lambda)|=1\}}\left(\eps+\int_0^{t-\lambda}d\tau\frac{\eps^{-1}}{[\eps^{-2}\tau]^{1/2}+1}\sum_{z\in I_{\pm}}G_{\eps^{-2}\tau}(x_i(\lambda),z)\right)\right]\nonumber\\
&{\lesssim \sum_{\pi\in S_n}\sum_{w}P^{(\epsilon)}_{\lambda}(x_{\pi(i)},w)P^{(\epsilon)}_{\lambda}(x_{\pi(j)},w+1)\left(\eps+\int_{0}^{t-\lambda}d\tau\frac{\eps^{-1}}{[\eps^{-2}\tau]^{1/2}+1}\sum_{z\in I_{\pm}}G_{\eps^{-2}\tau}(w,z)\right)}\nonumber\\
&\lesssim \frac{1}{[\eps^{-2}\lambda]^{1/2}+1} \sum_{\pi\in S_n}\sum_{w}P^{(\epsilon)}_{\lambda}(x_{\pi(i)},w)\left(\eps+\int_{0}^{t-\lambda}d\tau\frac{\eps^{-1}}{[\eps^{-2}\tau]^{1/2}+1}\sum_{z\in I_{\pm}}G_{\eps^{-2}\tau}(w,z)\right)\nonumber\\
&\lesssim \frac{1}{[\eps^{-2}\lambda]^{1/2}+1}\eps+ \frac{1}{[\eps^{-2}\lambda]^{1/2}+1}\int_{0}^{t-\lambda}d\tau\frac{\eps^{-1}}{[\eps^{-2}\tau]^{1/2}+1}\sum_{\pi\in S_n}\sum_{z\in I_{\pm}}G_{\eps^{-2}(\lambda+\tau)}(x_{\pi(i)},z)\nonumber\\x
&\lesssim \frac{1}{[\eps^{-2}\lambda]^{1/2}+1}\eps+ \frac{1}{[\eps^{-2}\lambda]^{1/2}+1}\int_{0}^{t-\lambda}d\tau\frac{\eps^{-1}}{[\eps^{-2}\tau]^{1/2}+1}\frac{1}{[\eps^{-2}(\lambda+\tau)]^{1/2}+1}\nonumber\\
&\lesssim \frac{1}{[\eps^{-2}\lambda]^{1/2}+1}\eps+ \frac{1}{[\eps^{-2}\lambda]^{1/2+\zeta}}\int_{0}^{t-\lambda}d\tau\frac{\eps^{-1}}{[\eps^{-2}\tau]^{1/2}+1}\frac{1}{[\eps^{-2}\tau]^{1/2-\zeta}}, \nonumber
\end{split}
\end{equation*}
where $S_n$ is the set of all permutations of $\{1,\dots,n\}$. To obtain $G_{\eps^{-2}(\lambda+\tau)}(x_{\pi(i)},z)$ in the third line of the right-hand side of the above calculation, we use \eqref{N4.5}, i.e., 
$$\sum_w P^{(\eps)}_{\lambda}(x_{\pi(i)},w)G_{\eps^{-2}\tau}(w,z)\lesssim \sum_w G_{\eps^{-2}\tau}(x_{\pi(i)},w)G_{\eps^{-2}\tau}(w,z)=G_{\eps^{-2}(\lambda+\tau)}(x_{\pi(i)},z),$$
followed by Kolmogorov's identity.  The role of Liggett's inequality \eqref{Liggett_new} is precisely to make Kolmogorov's identity applicable, allowing the transition probabilities to be merged into the form shown in the display above.   Furthermore, in the last inequality, we use that $\lambda+\tau> \max\{\tau,\lambda\}$  which implies that 
$$\frac{1}{(\lambda+\tau)^{1/2}}=\frac{1}{(\lambda+\tau)^{1/2-\zeta}}\frac{1}{(\lambda+\tau)^\zeta}<\frac{1}{\tau^{1/2-\zeta}}\frac{1}{\lambda^\zeta}.$$
Then the estimation follows directly. 
\medskip
Now we proceed estimating \eqref{est:gamma_2},  temporarily setting aside the factor $\eps^{-2}$. We can bound  $\mathbf{E}_\epsilon [\gamma^\eps_{2}(i,j,\lambda, t-\lambda)]$ by a constant times 
\begin{eqnarray}
\label{est:indicator_diff_rhos_squared}
&&{\bf E}_{\epsilon,\und x}\left[\mathbf{1}_{\{|x_i(\lambda)-x_j(\lambda)|=1\}}\left(\eps+\int_{0}^{t-\lambda}d\tau\frac{\eps^{-1}}{[\eps^{-2}\tau]^{1/2}+1}\sum_{z\in I_{\pm}}G_{\eps^{-2}\tau}(x_i(\lambda),z)\right)^2\right]\nonumber\\
&&\lesssim{\bf E}_{\epsilon,\und x}\left[\mathbf{1}_{\{|x_i(\lambda)-x_j(\lambda)|=1\}}\left(\eps^2+\left(\int_{0}^{t-\lambda}d\tau\frac{\eps^{-1}}{[\eps^{-2}\tau]^{1/2}+1}\sum_{z\in I_{\pm}}G_{\eps^{-2}\tau}(x_i(\lambda),z)\right)^2\right)\right]\nonumber\\
&&\lesssim \sum_{\pi\in S_n}\sum_{w}P^{(\epsilon)}_{\lambda}(x_{\pi(i)},w)P^{(\epsilon)}_{\lambda}(x_{\pi(j)},{w+1})\bigg(\eps^2\nonumber\\
&&+\int_{0}^{t-\lambda}d\tau_1\int_{\tau_1}^{t-\lambda}d\tau_2\frac{\eps^{-1}}{[\eps^{-2}\tau_1]^{1/2}+1}\frac{\eps^{-1}}{[\eps^{-2}\tau_2]^{1/2}+1}{\sum_{z_1\in I_{\pm}}G_{\eps^{-2}\tau_1}(w,z_1)\sum_{z_2\in I_{\pm}}G_{\eps^{-2}\tau_2}(w,z_2)\bigg)}.\nonumber
\end{eqnarray}
We note now that the contribution of the transition probabilities can be combined with the heat kernels in the following fashion: 
\begin{equation*}
\begin{split}
&\sum_{w}P^{(\epsilon)}_{\lambda}(x_{\pi(i)},w)P^{(\epsilon)}_{\lambda}(x_{\pi(j)},{w+1})\sum_{z_1\in I_{\pm}}G_{\eps^{-2}\tau_1}(w,z_1)\sum_{z_2\in I_{\pm}}G_{\eps^{-2}\tau_2}(w,z_2)\nonumber\\
&\lesssim\sum_{z_1\in I_{\pm}}\sum_{z_2\in I_{\pm}}\underbrace{\sum_{w}P^{(\epsilon)}_{\lambda}(x_{\pi(i)},w)G_{\eps^{-2}\tau_1}(w,z_1)}_{\lesssim G_{\eps^{-2}(\tau_1+\lambda)}({x_{\pi(i)}},z_1)}\underbrace{P^{(\epsilon)}_{\lambda}(x_{\pi(j)},{w+1})}_{\lesssim\frac{1}{[\eps^{-2}\lambda]^{1/2}+1}}\underbrace{G_{\eps^{-2}\tau_2}(w,z_2)}_{\lesssim\frac{1}{[\eps^{-2}\tau_2]^{1/2}+1}},\nonumber
\end{split}
\end{equation*}
from where it follows that the last display is bounded from above by a constant times 
\begin{eqnarray}
 \frac{\eps^2}{[\eps^{-2}\lambda]^{1/2}}+\frac{1}{[\eps^{-2}\lambda]^{1/2}}\int_{0}^{t-\lambda}d\tau_1\int_{\tau_1}^{t-\lambda}d\tau_2\frac{\eps^{-1}}{[\eps^{-2}\tau_1]^{1/2}+1}\frac{\eps^{-1}}{[\eps^{-2}\tau_2]+1}{\sum_{\pi\in S_n}}\sum_{z_1\in I_{\pm}}G_{\eps^{-2}(\tau_1+\lambda)}({x_{\pi(i)}},z_1)\nonumber.
\end{eqnarray} 
We now estimate the rightmost term in the last line of the last display: we first bound the heat kernel as $G_{\eps^{-2}(\tau_1+\lambda)}(x_{\pi(i)},z)\lesssim \frac{1}{[\eps^{-2}(\lambda+\tau_1)]^{1/2}}$, then use that 
$$\frac{1}{[\eps^{-2}(\lambda+\tau_1)]^{1/2}}\leq \frac{1}{[\eps^{-2}\lambda]^{1/4}}\frac{1}{[\eps^{-2}\tau_1]^{1/4}},$$ because $\lambda+\tau_1>\tau_1$, $\lambda+\tau_1>\lambda$ and $\tau_2>\tau_1$, and 
$$\frac{1}{[\eps^{-2}\tau_2]}= \frac{1}{[\eps^{-2}\tau_2]^{1-\zeta}}\frac{1}{[\eps^{-2}\tau_2]^{\zeta}}\leq \frac{1}{[\eps^{-2}\tau_2]^{1-\zeta}}\frac{1}{[\eps^{-2}\tau_1]^{\zeta}}.$$
Note that the exponents of $(\lambda+\tau_1)^{-1/2}$ and $\tau_2^{-1/2}$ can be distributed in different ways to ensure that the integrand in the double integral remains integrable. The choices presented above constitute one such example that satisfies this condition.  Overall, we have:
\begin{eqnarray*}
\frac{1}{[\eps^{-2}\lambda]^{1/2}}\int_0^{t-\lambda}d\tau_1\int_{\tau_1}^{t-\lambda}d\tau_2\frac{\eps^{-1}}{[\eps^{-2}\tau_1]^{1/2}+1}\frac{\eps^{-1}}{[\eps^{-2}\tau_2]+1}\frac{1}{[\eps^{-2}(\lambda+\tau_1)]^{1/2}}\lesssim \frac{\eps^3}{\lambda^{3/4}}.
\end{eqnarray*}
\medskip

We now proceed with estimating \eqref{est:gamma_1}. 
             First, it is straightforward  to check that when $n=3$ the integral in time appearing in the  expectation  of $\gamma_{1}^{\eps, n-1}$ takes the form $$\int_{0}^{t-\lambda}d\tau\,{\bf E}_{\epsilon,y}\left[\mathbf 1_{\{\tau_{i,j}>\frac{\tau}{2}\}}{\bf 1}_{\{\und x(\tau)=(N-1,N)\}}\right]\lesssim \eps^2$$ because we first condition on $\mathcal{F}(\frac{\tau}{2})$, being the $\sigma$-algebra generated by the active/passive marks in the interval $[0,\frac{\tau}{2})$, then estimate the expectation ${\bf 1}_{\und x(\tau)=(N-1, N)}$ and finally  estimate $\mathbf 1_{\{\tau_{i,j}>\frac{\tau}{2}\}}$ using \eqref{important*m=3}; see a similar calculation in the proof of Lemma \ref{lem:widetilde_gammas}.  
             To be more precise, we start with the internal integral in \eqref{19.1*} when $n=3$:
             \begin{eqnarray*}
             &&\int_{0}^{t-\lambda}d\tau\,{\bf E}_{\epsilon,y}\left[\mathbf 1_{\{\tau_{i,j}>\frac{\tau}{2}\}}{\bf 1}_{\und x(\tau)=(N-1,N)}\right]=\int_{0}^{t-\lambda}{\bf E}_{\epsilon,y}\left[\mathbf 1_{\tau_{i,j}>\frac{\tau}{2}}{\bf E}_{\epsilon}\left[{\bf 1}_{\und x(\tau)=(N-1,N)}|\mathcal{F}\left(\tfrac{\tau}{2}\right)\right]\right]\nonumber\\
             &&\hspace{2cm}\lesssim \int_{0}^{t-\lambda}d\tau\, {\bf E}_{\epsilon,y}\left[\mathbf 1_{\tau_{i,j}>\frac{\tau}{2}}\frac{1}{[\eps^{-2}(\tau/2)]+1}\right]\lesssim \int_{0}^{t-\lambda}d\tau\, \frac{1}{[\eps^{-2}(\tau/2)]^{3/2}+1}\lesssim\eps^{2},
             \end{eqnarray*}
            where the last line follows from Remark \ref{rem:comment_integration}. Now, the remaining follows from the estimates of $\gamma_0^{\eps}$. Putting everything together we get that for $n=3$ 
            $$\mathbf{E}_\epsilon [\gamma_{1}^{\eps, n-1}(i,j,\lambda)]
    \overset{n=3}=\mathbf{E}_\epsilon [\gamma_{1}^{2,\eps}(i,j,\lambda)]\lesssim\eps\phi_0(\lambda).$$
For $n\geq 4$ and $i,j\subset \{1,\dots,n\}$, 
we need to control
\begin{eqnarray*}
&\int_{0}^{t-\lambda}\,d\tau
\mathbf{E}_\epsilon\Big[ 
 \eps^{-1} \mathbf{1}_{\{|  x_i(\lambda)- x_j(\lambda)|=1\}}\eps^{-2}\left|\rho_\eps(x_i(\lambda), t-\lambda)-\rho_\eps(x_j(\lambda), t-\lambda)\right|\\&\times \int_{0}^{t-\lambda}d\tau\sum_{\und y=\underline{x}^{(i)}(\lambda),\,\underline{x}^{(j)}(\lambda)}{\bf E}_{\epsilon,\underline{y}}\left[\mathbf 1_{\{\tau_{i,j}>\frac \tau2\}}{\bf 1}_{\underline{x}(\tau)=I_+^{n-1}}\right]\Big].
  \end{eqnarray*}
             From Corollary \ref{corollary: difference of rhos}  we can bound the previous display by a constant times 
             \begin{eqnarray*}
\eps^{-2-2\zeta}\int_{0}^{t-\lambda}\,d\tau
\mathbf{E}_\epsilon\left[ 
 \mathbf{1}_{\{|  x_i(\lambda)- x_j(\lambda)|=1\}} \sum_{\und y=\underline{x}^{(i)}(\lambda),\,\underline{x}^{(j)}(\lambda)}{\bf E}_{\epsilon,\underline{y}}\left[\mathbf 1_{\{\tau_{i,j}>\frac \tau2\}}{\bf 1}_{\underline{x}(\tau)=I_+^{n-1}}\right]\right]
 \end{eqnarray*}
 and now we control the time integral as follows:
  \begin{equation*}\label{sum_diff_v2}  
  \begin{split}
&{\int_{0}^{t-\lambda}\,d\tau
\mathbf{E}_\epsilon\left[ \mathbf{1}_{\{|x_i(\lambda)-x_j(\lambda)|=1\}}{\bf E}_{\epsilon,\underline{x}^{(i)}(\lambda)}\left[\mathbf 1_{\tau_{i,j}>\frac{\tau}{2}}{\bf 1}_{\underline{x}(\tau)=\{N-(n-2),\dots, N\}}\right]\right]}\nonumber\\
&{\lesssim \int_{0}^{t-\lambda}\,d\tau
\mathbf{E}_\epsilon\left[ \mathbf{1}_{\{|x_i(\lambda)-x_j(\lambda)|=1\}}\mathbf{P}_{\epsilon}(\underline{x}^{(i)}(\lambda)\overset{\tau}\longrightarrow (N-(n-2),\dots, N)\right]}\nonumber\\
&= {\int_{0}^{t-\lambda}\,d\tau \sum_{\und w} \mathbf{P}_{\epsilon}(\underline{x}\overset{\lambda}\rightarrow \underline{w})\mathbf{1}_{\{|w_i-w_j|=1\}}\mathbf{P}_{\epsilon}(\underline{w}^{(i)}\overset{\tau}\longrightarrow (N-(n-2),\dots, N)}\nonumber\\
&= \int_{0}^{t-\lambda}\,d\tau \sum_{w_1,\dots,w_n} \mathbf{P}_{\epsilon}(\underline{x}\overset{\lambda}\rightarrow w_1,\dots, w_j,w_j\pm1,\dots,w_n)\mathbf{P}_{\epsilon}(\underline{w}^{(i)}\overset{\tau}\longrightarrow N-(n-2),\dots, N)\nonumber\\
&\leq\int_{0}^{t-\lambda}\,d\tau  \sum_{\pi\in S_n}\prod_{k\neq i,j}P^{(\eps)}_{\lambda+\tau}(x_{\pi(k)}, N-(n-k))\nonumber\\
&\hspace{1cm}\times \sum_{w_j}P^{(\eps)}_{\lambda}(x_{\pi(j)}, w_j)P^{(\eps)}_{\tau}(w_j, N-(n-j))P^{(\eps)}_{\lambda}(x_{\pi(i)},w_j+1)\nonumber\\
&\lesssim \frac{1}{[\eps^{-2}\lambda]^{1/2}}\int_{0}^{t-\lambda}\,d\tau \sum_{\pi\in S_n}\prod_{k\neq i}P^{(\eps)}_{\lambda+\tau}(x_{\pi(k)}, N-(n-k))\nonumber\\
&\lesssim \frac{1}{[\eps^{-2}\lambda]^{1/2}+1}\int_{0}^{t-\lambda}\,d\tau \frac{1}{[\eps^{-2}(\tau+\lambda)]^{(n-1)/2}+1}\nonumber\\
&\lesssim \frac{1}{[\eps^{-2}\lambda]+1}\int_{0}^{t-\lambda}\,d\tau \frac{1}{[\eps^{-2}(\tau+\lambda)]^{(n-2)/2}+1}\nonumber\\
&\lesssim \frac{1}{[\eps^{-2}\lambda]+1}\int_{0}^{t-\lambda}\,d\tau \frac{1}{[\eps^{-2}(\tau+\lambda)]+1}\nonumber\\
&\lesssim \frac{1}{[\eps^{-2}\lambda]+1}\int_{0}^{t-\lambda}\,d\tau \frac{1}{[\eps^{-2}\tau]^{1-\zeta}}\lesssim  \frac{1}{[\eps^{-2}\lambda]+1}\epsilon^{2-2\zeta}(t-\lambda)^\zeta.
\end{split}
\end{equation*}
In the above estimation,  we bound $\mathbf 1_{\tau_{i,j}>\frac{\tau}{2}}$ by 1,  we obtain the fourth line of the right-hand side by Liggett's inequality given in \eqref{Liggett},  the fifth line  by Kolmogorov's identity and the sixth-eighth line by also using that  $\lambda+\tau>\max\{\tau, \lambda\}$. Thus, 
$$  \frac{1}{[\eps^{-2}(\tau+\lambda)]^{(n-1)/2}+1}
\lesssim \frac{1}{[\eps^{-2}(\tau+\lambda)]^{(n-2)/2}+1}\frac{1}{[\eps^{-2}\lambda]^{1/2}+1}$$ and
$$\frac{1}{[\eps^{-2}(\tau+\lambda)]^{(n-2)/2}+1}\leq \frac{1}{[\eps^{-2}\tau]^{1-\zeta}}.$$
\end{proof}

\begin{remark}\label{rem:comment_integration}
For $p\geq2$, the factors {$\psi_p^\eps(\lambda)$} in \eqref{13.1.111_p=1} are not integrable over the time interval $[0,t]$, i.e., when $t^*=0$. In general, throughout this work, whenever the time exponent in the denominator is equal or exceeds 1-- e.g. $\frac{\epsilon^{-1}}{[\epsilon^{-2}\lambda]^{k}+1}$ or $\frac{\epsilon^{-1}}{[\epsilon^{-2}(t-\lambda)]^{k}+1}$ with $k\geq 1$-- we  apply the same technique used in the proof of Lemma \ref{lem: bound T1} to reduce the exponent to a value strictly less than 1 and make these factors integrable over the time interval $[0,t]$. We will refer to the next estimation at several points later in the paper when $k\geq 2$
\begin{eqnarray}\label{important*m=3}
\int_{0}^{t}d\lambda\frac{\epsilon^{-1}}{[\epsilon^{-2}\lambda]^{k}+1}\leq\epsilon+\int_{\epsilon^2}^{t}d\lambda\frac{\epsilon^{-1}}{[\epsilon^{-2}\lambda]^{1+\zeta}}
&=&\epsilon+\epsilon^{1+2\zeta}t^{-\zeta}+ \epsilon^{1+2\zeta} \epsilon^{-2\zeta}\nonumber\\
&\leq& c (\epsilon+\epsilon^{1+2\zeta}t^{-\zeta})\leq c \epsilon.
\end{eqnarray}
\end{remark}

\begin{lemma}\label{lem:widetilde_gammas}
Let  $i,j\in \{1,\dots,n\}$, $\zeta>0$ and $\eps<1$. Under the assumption given in \eqref{Ass1}, we have: 
\begin{equation}
\mathbf{E}_\epsilon [\widetilde\gamma_{0}^{\eps}(i,j,\lambda, t-\lambda)]\lesssim \frac{\epsilon^{1-3\zeta}}{\lambda^{1-\zeta}}
\end{equation}
and
\begin{equation}
\mathbf{E}_\epsilon [\widetilde\gamma^\eps_{2}(i,j,\lambda, t-\lambda)]\lesssim \frac{\epsilon^{3/2-2\zeta}}{\lambda^{1-\zeta}}.
\end{equation}
For any $J\subset\{1,\dots,n\}$ and $1\leq p\leq n$,
               \begin{equation}
  \mathbf{E}_\epsilon [ \widetilde\gamma^\eps_{p}(\lambda)]\lesssim \psi_{p+1}^\eps(\lambda).
  \end{equation}
\end{lemma}
\begin{proof}
 The proof follows the same steps as the proof of  Lemma \ref{lem:gammas}. More precisely, for $m\in\{0,2\}$, let us denote $\tilde\gamma^\eps_m(i,j,\lambda, t-\lambda)=\mathbf 1_{\{\tau_{i,j}>\frac{\lambda}{2}\}}\gamma_m^\eps(i,j,\lambda, t-\lambda)$ and 
  $\tilde\gamma^\eps_p(i,j,\lambda)=\mathbf 1_{\{\tau_{i,j}>\frac{\lambda}{2}\}}\psi_p^\eps(\lambda)$.
Then the estimate goes as follows: let $\mathcal{F}(\frac{\lambda}{2})$ be the $\sigma$-algebra generated by the active/passive marks in the interval $[0,\frac{\lambda}{2})$. Then for $m\in\{0,2\}$
  \begin{eqnarray*}
  \mathbf{E}_\epsilon \left[\widetilde{\gamma}_m^\eps(\und x(\lambda)) \mathbf 1_{\{\tau_{i,j}>\frac{\lambda}{2}\}}\right]&=& \mathbf{E}_\epsilon \left[ \mathbf 1_{\tau_{i,j}>\frac{\lambda}{2}}\mathbf{E}_\epsilon \left[ \gamma(\und x(\lambda)) |\mathcal{F}\left(\tfrac{\lambda}{2}\right)\right]\right]\nonumber\\
  &\leq&\phi_m(\lambda)\mathbf{E}_\epsilon \left[ \mathbf 1_{\tau_{i,j}>\frac{\lambda}{2}}\right]\lesssim\phi_m(\lambda)\frac{1}{\sqrt{\epsilon^{-2}\frac{\lambda}{2}}+1},
 \end{eqnarray*}
 where we use \eqref{19.e1}. Finally, we adjust the power of $\lambda$ such that the time factors are integrable, e.g., for $\tilde\gamma_0$, starting from \eqref{est:gamma_0}, we write the estimate
 \begin{eqnarray*}
\left( \frac{1}{\lambda^{1/2}}+\frac{1}{\lambda^{1/2+\zeta}}\right)\frac{1}{[\epsilon^{-2}\frac{\lambda}{2}]^{1/2}+1}
&\lesssim&\frac{\eps^{-1}}{[\epsilon^{-2}\lambda]+1}+\frac{\eps^{-1-\zeta}}{[\epsilon^{-2}\lambda]^{1+\zeta}+1}
\lesssim\epsilon^{1-3\zeta}\frac{1}{\lambda^{1-\zeta}}.
 \end{eqnarray*} The same estimate holds for $\tilde\gamma^\eps_p(i,j,\lambda)$.
 \end{proof}
\begin{remark}
The expectations estimated in Lemma~\ref{lem:gammas} will later be integrated in time. The condition $0<\zeta<\frac12$ guarantees the integrability of the corresponding time singularities.
\end{remark}

\begin{lemma}\label{lem: difference of rhos}
Let {$x\in \{-N,\dots, N-1\}$} and $0\leq t\leq T$. Under assumption \eqref{Ass1}, {for any $\eps<1$, } there exists a constant $C$ such that 
\begin{eqnarray}
|\rho_\epsilon(x,t)-\rho_\epsilon(x+1,t)|&\leq&{c_T\epsilon}+\epsilon^{-1}\frac{j}{2}\int_{0}^t d\lambda \frac{1}{\sqrt{\epsilon^{-2}\lambda}+1}\sum_{z\in I_\pm}G_{\epsilon^{-2}\lambda}(x,z).
\end{eqnarray}
where  $G_{\epsilon^{-2}\lambda}(x,z)$ denotes the heat kernel defined as:     \begin{equation}
    \label{a3.6???}
G_t (x,y)=   \frac{e^{- (x-y)^2/2t}}{\sqrt{2\pi t}}.
    \end{equation}
\end{lemma}
\begin{proof}
The proof of this lemma shares similar  ideas as those of the proof of  Lemma 6.2 in \cite{gonccalves2020non}, but given the nature of the equation, some technical details are different.
 
Fix $x\in \{-N,\dots, N-1\}$ and $0\leq t\leq T$. We denote $h_{\epsilon}(x,t)=\rho_{\epsilon}(x,t)-\rho_t(\epsilon x)$, where $\rho_t$ is the solution of the hydrodynamic equation \eqref{integralformofequation}. Then by summing and subtracting appropriate terms, we get 
\begin{equation*}
\begin{split}
|\rho_{\epsilon}(x,t)-\rho_{\epsilon}(x+1,t)|&\leq |h_{\epsilon}(x,t)-h_{\epsilon}(x+1,t)|+|\rho_t(\epsilon x)-\rho_t(\epsilon (x+1))|\\
&\leq |h_{\epsilon}(x,t)-h_{\epsilon}(x+1,t)|+\epsilon.
\end{split}
\end{equation*}
Last inequality is obtained from the Mean value theorem and the fact that $\rho_t(u)$ is a smooth function, see Remark \ref{rem:smoothness}.
It remains now to bound the discrete gradient of $h_{\epsilon}(x,t)$. A simple computation shows that 
  for $x\in \{-N,\dots, N-1\}$,   $h_{\epsilon}(x,t)$ satisfies
  
\begin{equation}\label{eq:linearized_h}
	\begin{cases}
	&\frac{d}{dt}h_\epsilon(x,t)= \frac 12 \Delta_\epsilon h_\epsilon(x,t)+ F_{\epsilon}(x,t)+
\epsilon^{-1} \frac j2\Big(\mathbf 1_{\{x\in I_+\}} D_+\rho_\epsilon(x,t)
-\mathbf 1_{\{x\in I_-\}}  D_-\rho_\epsilon(x,t)\Big)\\
	&{h_\epsilon(x,0)=0}.
	\end{cases}
	\end{equation}
    \smallskip
    
Note that the initial condition is zero since the initial conditions match, i.e. $\rho_0\equiv u_0$. 
Above $F_{\epsilon}(x,t)=\frac{1}{2}\left[\Delta_\epsilon\rho_t(\epsilon x)-\Delta\rho_t(\epsilon x)\right]$. {Note that since $\rho_t(x)$ is smooth then $\|F\|_\infty\leq C \epsilon^{2}$.} Now, from  the Duhamel's formula,   we get
\begin{eqnarray*}
h_\epsilon(x,t)&=\int_{0}^t d\lambda \sum_{z}P^{(\epsilon)}_{\lambda}(x,z) F_{\epsilon}(z,t-\lambda) +\epsilon^{-1}{\frac j2}\int_{0}^t d\lambda \sum_{z\in I_+}P^{(\epsilon)}_{\lambda}(x,z) D_+\rho_\epsilon(z,t-\lambda) \nonumber \\&{-\epsilon^{-1}\frac j2\int_{0}^t d\lambda \sum_{z\in I_-}P^{(\epsilon)}_{\lambda}(x,z) D_-\rho_\epsilon(z,t-\lambda)},
\end{eqnarray*}
so that 
\begin{eqnarray*}
h_\epsilon(x,t)-h_\epsilon(x+1,t)&=&\int_{0}^t d\lambda \sum_{z}\left(P^{(\epsilon)}_{\lambda}(x,z)-P^{(\epsilon)}_{\lambda}(x+1,z)\right) F_{\epsilon}(z,t-\lambda) \nonumber\\
&&+\epsilon^{-1}{\frac j2}\int_{0}^t d\lambda \sum_{z\in I_+}\left(P^{(\epsilon)}_{\lambda}(x,z)-P^{(\epsilon)}_{\lambda}(x+1,z)\right) D_+\rho_\epsilon(z,t-\lambda) \nonumber\\
&&-{\epsilon^{-1}{\frac j2}\int_{0}^t d\lambda \sum_{z\in I_-}\left(P^{(\epsilon)}_{\lambda}(x,z)-P^{(\epsilon)}_{\lambda}(x+1,z)\right) D_-\rho_\epsilon(z,t-\lambda)}. \nonumber
\end{eqnarray*}
Since $P^{(\epsilon)}_{\lambda}$ is a probability and since $\|F\|_\infty\leq C \epsilon^{2}$, we can bound the first line on the right hand side of last display by $C\epsilon^2 t$. Then we obtain
\begin{eqnarray*}
|h_\epsilon(x,t)-h_\epsilon(x+1,t)|&\leq C\epsilon^2 t+\epsilon^{-1}{\frac j2}\int_{0}^t d\lambda \sum_{z\in I_+}\left|P^{(\epsilon)}_{\lambda}(x,z)-P^{(\epsilon)}_{\lambda}(x+1,z)\right| D_+\rho_\epsilon(z,t-\lambda) \nonumber\\
&+\epsilon^{-1}{\frac j2}\int_{0}^t d\lambda \sum_{z\in I_-}\left|P^{(\epsilon)}_{\lambda}(x,z)-P^{(\epsilon)}_{\lambda}(x+1,z)\right| D_-\rho_\epsilon(z,t-\lambda). \nonumber
\end{eqnarray*}
From this it follows that $|\rho_\epsilon(x,t)-\rho_\epsilon(x+1,t)|$ is bounded by 
\begin{equation*}\begin{split}
&C(\epsilon^2 t+\epsilon)+\epsilon^{-1}\frac j2\int_{0}^t d\lambda \sum_{z\in I_\pm}\left|P^{(\epsilon)}_{\lambda}(x,z)-P^{(\epsilon)}_{\lambda}(x+1,z)\right| D_\pm\rho_\epsilon(t-\lambda,z) \nonumber\\
&\leq C(\epsilon^2 t+\epsilon)+\epsilon^{-1}\frac j2 \int_{0}^t d\lambda \frac{1}{\sqrt{\epsilon^{-2}\lambda}+1}\sum_{z\in I_\pm}G_{\epsilon^{-2}\lambda}(x,z)
\end{split}\end{equation*}
where in the last inequality we have used \eqref{N4.6 additional}.  Since $\eps^2t+\eps\leq\eps ( T+1)$, we set $c_T=\eps T+1$ and 
this finishes the proof. 
\end{proof}

\begin{corollary}\label{cor: difference of rhos}
Let {$x\in \{-N,\dots, N-1\}$} and $0\leq t\leq T$. Under assumption \eqref{Ass1}, {for any $\eps<1$, } there exists a constant $C$ such that 
\begin{equation}
|\rho_\epsilon(x,t)-\rho_\epsilon(x+1,t)|\lesssim\eps^{1-2\zeta}.
\end{equation}
\end{corollary}
The above estimate is obtained after bounding the heat kernel $G$ and integrate reducing the overall time-exponent in the denominator to $1-\zeta$.

\subsection{Integral inequalities for space  $v$-functions (Step 4.)}

Equations \eqref{13.1.111_p=1} and \eqref{17.3.3} become more tractable when taking the supremum over all configurations with the same number of particles. Therefore, for every $0\leq t\leq T$, we define
\begin{equation}\label{def:a}
a_{\epsilon}(n,t):=\sup_{\underline{x}\in\Lambda_N^{n,\neq}}|v_n^{\epsilon}(\underline{x},t)|\quad \textrm{with}\quad a_{\epsilon}(0,t)=1
\end{equation} 
and for $\und e_1$ being the unit vector in the positive 1st-direction, we  define
\begin{equation}\label{def:d}
d_{\epsilon}(n,t):=\sup_{\underline{x}\in\Lambda_N^{n,\neq}}|d_n^\eps(\und x,\und x+\und e_1,t)|\;\textrm{and\;} d_{\epsilon}(0,t)=0.
\end{equation} 
such that $\underline{x}+\und e_1\in \Lambda_N^{n,\neq}$. Similarly, we define
\begin{equation}\label{def:beta}
    \beta_\epsilon(n,t, {t^*}):=\sup_{\underline{x}\in\Lambda_N^{n,\neq}}\beta^{\eps}_n(\underline{x},\underline{x}+\und e_1,t, {t^*}).
\end{equation}
\begin{remark}
We remark that in the definition of $d$ given above, it is irrelevant whether the increment is taken in the first coordinate or not: since we start from the set representation, e.g.\ $X=\{x,y,z\}$ and $\tilde X=\{x,y,z+1\}$, we may always choose an ordering such as $\underline{x}=(x_1,x_2,x_3)=(z,x,y)$ so that $\underline{x}+\underline{e}_1=(x_1+1,x_2,x_3)=(z+1,x,y)$. 
Of course, one could equally consider $\underline{x}+\underline{e}_2$ or $\underline{x}+\underline{e}_3$, but the argument works in exactly the same way for any coordinate. 
\end{remark}

For simplicity of presentation, we restrict the calculations that follow to the right boundary $I_+$. The corresponding calculations for the left boundary $I_-$ are analogous and are therefore omitted. The first main result of this subsection is the next lemma.  
\begin{lemma}
\label{lem: IIfor a for space_v_estimate} Let  $K\geq 2$ be a fixed  integer number, $0\leq t^*<t\leq T$ and $\zeta>0$. Under Assumption 1 given in \eqref{Ass1}  the following integral inequality for $a_{\epsilon}(n,t)$ holds: for $n\geq 1$ 
\begin{equation}\label{split_time_integral for main thm}
\begin{split}
a_{\epsilon}(n,t)&\lesssim a_{\epsilon}(n,t^*)\\
&+\int_{0}^{t-t^*} d\lambda \;\frac{\eps^{-1}}{[\eps^{-2}\lambda]^{1/2}+1}a_{\eps}(n,t-\lambda)\\
 &+\int_{0}^{t-t^*} d\lambda \;\frac{\eps^{-1}}{[\eps^{-2}\lambda]^{1/2}+1}\sum_{q=2}^K
 a_\eps(n-1+q,t-\lambda)\\
 &+\mathbf 1_{n\geq 2}\int_{0}^{t-t^*} d\lambda\;\frac{\eps^{-1}}{[\eps^{-2}\lambda]^{1-\zeta}}
\sum_{p=2}^{\min\{n-1,K\}}
a_{\eps}(n-p,t-\lambda)\\
&+\mathbf 1_{n\geq 2}\int_{0}^{t-t^*} d\lambda\;\frac{\eps^{-1}}{[\eps^{-2}\lambda]^{1-\zeta}}\sum_{q=1}^{K}\;\;\;
\sum_{p=2}^{\min\{n,K\}}\,
a_{\eps}(n-p+q,t-\lambda)\\
 &+\mathbf{1}_{\{n=2\}}\epsilon^{1-2\zeta}(t-t^*)^{\zeta}+\mathbf{1}_{\{3\leq n\leq K\}} \epsilon\\
&+\mathbf 1_{n= 3}\epsilon\left((t-t^*)^{1/2}+(t-t^*)^{1/2-\zeta}\right) +\mathbf 1_{n\geq 4}\epsilon^{2-6\zeta}(t-t^*)^{\zeta}\\
&+{\bf 1}_{n\geq 2}\int_{0}^{t-t^*}d\lambda\phi_0(\lambda)\beta_\epsilon(n-1,t-\lambda,{t^*})\\
&+  \mathbf{1}_{\{n\geq2\}}\int_{0}^{t-t^*} d\lambda\,
\phi_2^\eps(\lambda)a_\epsilon(n-2,t-\lambda)\,.
\end{split}
\end{equation}
\end{lemma}
\begin{proof}
The derivation of \eqref{split_time_integral for main thm} is mainly based on  Lemma \ref{lem:gammas}.  We first bound each $v_n^\eps$ appearing on the right-hand side of \eqref{13.1.111 with gammas1} with its  supremum with respect to all configurations with the same number of  particles. We bound all $\psi_p(\lambda)$ with $2\leq p\leq \min\{n,K\}$ in \eqref{13.1.111_p=1} by $\frac{1}{[\eps^{-2}\lambda]^{1-\zeta}}$ using Remark \ref{rem:comment_integration} and estimate $\mathbf{E}_\epsilon \left[ {\bf 1}_{\und{x}(\lambda)=I^n_+}\right]$ as $\frac{1}{[\eps^{-2}\lambda]^{n/2}+1}\leq \frac{1}{[\eps^{-2}\lambda]^{1-\zeta}}$, $n\leq K$.
\end{proof}
Note that Lemma \ref{lem: IIfor a for space_v_estimate} is used in the proof of Theorem~\ref{thm:vestimate2}; see Section~\ref{proof of main theorem}. However, this lemma is not enough for our purposes, hence we derive another integral inequality, which is stated in the next lemma. The idea behind this lemma is as follows: the right-hand side of \eqref{13.1.111 with gammas1} contains terms that cannot be further improved using the configuration dependence of $v_j^{\epsilon}$. Consequently, $v_j^{\epsilon}$ is bounded by the supremum over all configurations with the same number of particles. In contrast, for the remaining terms, the configuration dependence plays a role and leads to sharper estimates; hence, these terms are not bounded by the supremum over all configurations with the same number of particles.

\begin{lemma}\label{lem: IIfor a} Let  $K\geq 2$ be a fixed  integer number, $0\leq t^*<t\leq T$ and $\zeta>0$. Under {\bf{Assumption 1}} given in \eqref{Ass1},  the following integral inequality for $v^{\epsilon}_n(\und x,t)$ holds: for $n\geq 1$ and any labeled configuration $\und x\in\Lambda_N^{\neq, n}$
\begin{eqnarray}\label{split_time_integral}
|v_n^{\epsilon}(\underline{x},t)|&\lesssim&a_{\epsilon}(n,t^*)\nonumber\\
&&+\int_{0}^{t-t^*} d\lambda \;\eps^{-1}\sum_{j=1}^{n}\sum_{z'\in I_{+}}
 \mathbf{E}_\epsilon \left[\mathbf{1}_{\{x_j(\lambda)\in I_{+}\}} |v_{n}^\epsilon(\und x^{(j)}{(\lambda)}\cup z',t-\lambda)|\Big]\right]\nonumber\\
 &&+\int_{0}^{t-t^*} d\lambda \;\frac{\eps^{-1}}{[\eps^{-2}\lambda]^{1/2}+1}\sum_{j=1}^{n}\sum_{\substack{\underline{z}'\subset I_{+},\\q\geq2}}
 \mathbf{E}_\epsilon \left[ |v_{n-1+q}^\epsilon\big(\underline{x}^{(j)}(\lambda)\cup \underline{z}',t-\lambda\big)|\Big]\right]\nonumber\\
 &&+\mathbf 1_{n\geq 2}\int_{0}^{t-t^*} d\lambda\;\frac{\eps^{-1}}{[\eps^{-2}\lambda]^{1-\zeta}}
\sum_{p=2}^{\min\{n,K\}}
a_{\eps}(n-p,t-\lambda)\nonumber\\
&&+\mathbf{1}_{\{n\geq2\}}\int_{0}^{t-t^*} d\lambda\;\frac{\eps^{-1}}{[\eps^{-2}\lambda]^{1-\zeta}}\sum_{q=1}^{K}\;\;\;
\sum_{p=2}^{ \min\{n,K\}}\,
a_{\eps}(n-p+q,t-\lambda)\nonumber\\
&&+\mathbf{1}_{\{2\leq n\leq K\}}\int_{0}^{t-t^*} d\lambda \;
\eps^{-1}\mathbf{E}_\epsilon \left[ {\bf 1}_{\und{x}(\lambda)=I^n_+}\right]\nonumber\\
&&+\epsilon\left((t-t^*)^{1/2}+(t-t^*)^{1/2+\zeta}\right) \mathbf 1_{n= 3}+(t-t^*)^{\zeta}\epsilon^{2-4\zeta}\mathbf 1_{n\geq 4}\nonumber\\
&&+{\bf 1}_{\{n\geq 2\}}\int_{0}^{t-t^*}d\lambda{\phi_0(\lambda)}\beta_\epsilon(n-1,t-\lambda,{t^*})\nonumber\\
&&+  \mathbf{1}_{\{n\geq2\}}\int_{0}^{t-t^*} d\lambda\,
{\phi_2^\epsilon(\lambda)} a_\epsilon(n-2,t-\lambda)\,.
\end{eqnarray}
\end{lemma}

The proof is omitted, as it is similar to the proof of Lemma \ref{lem: IIfor a for space_v_estimate}. 
\begin{lemma}\label{lem: IIfor d} Let  $K\geq 2$ be a fixed integer, $0\leq t^*<t\leq T$ and $\zeta>0$.
For $n\geq 1$, it holds: 
\begin{eqnarray}\label{split_time_integral for d}
d_{\epsilon}(n,t)&\lesssim &\mathbf{1}_{\{2\leq n\leq K\}}\eps+ \beta_{\epsilon}(n,t, {t^*}),
\end{eqnarray}
where 
\begin{equation}\label{split_time_integral for beta}
\begin{split}
\beta_{\epsilon}(n,t, t^*)&\lesssim \epsilon(t-t^*)^{-1/2} a_{\epsilon}(n,t^*)\\
&+\int_{0}^{t-t^*} d\lambda \;\frac{\eps^{-1}}{[\eps^{-2}\lambda]^{1-\zeta}}a_{\eps}(n,t-\lambda)\\
&+\int_{0}^{t-t^*} d\lambda \;\frac{\eps^{-1}}{[\eps^{-2}\lambda]^{1-\zeta}}\sum_{q=2}^K
 a_\eps(n-1+q,t-\lambda)\\
&+\int_{0}^{t-t^*} d\lambda\;\frac{\eps^{-1}}{[\eps^{-2}\lambda]^{1-\zeta}}\sum_{q=1}^{K}\;\;\;
\sum_{p=2}^{\min\{n,K\}}\,
a_{\eps}(n-p+q,t-\lambda)\\
&+\mathbf 1_{n\geq 2}\int_{0}^{t-t^*} d\lambda\;\frac{\eps^{-1}}{[\eps^{-2}\lambda]^{1-\zeta}}
\sum_{p=2}^{\min\{n-1,K\}}\,
a_{\eps}(n-p,t-\lambda)\\
&+\mathbf{1}_{\{n\geq2\}}\int_{0}^{t-t^*}d\lambda
\frac{\eps^{3/2-2\zeta}}{\lambda^{1-\zeta}}a_{\epsilon}(n-2,t-\lambda)\\
&+\mathbf{1}_{\{n\geq2\}}\int_{0}^{t-t^*}d\lambda
\frac{\epsilon^{1-3\zeta}}{\lambda^{1-\zeta}} d_{\epsilon}(n-1,t-\lambda).
\end{split}
\end{equation}
\end{lemma}
\begin{proof}
The integral inequality \eqref{split_time_integral for d}--\eqref{split_time_integral for beta} is obtained by \eqref{17.3.3} and \eqref{17.3.3,n=1}. The derivation is analogous to that of Lemma \ref{lem: IIfor a}, except that the $\gamma$-factors are replaced by modified factors $\widetilde\gamma$ defined in \eqref{19.1**}--\eqref{19.3*}. This is necessary since \eqref{17.3.3} includes the additional indicator  $\mathbf 1_{\{\tau_{i,j}>\frac{\lambda}{2}\}}$ which can be estimated using Theorem \ref{3.01}-- a result originally proved in \cite{de2012truncated}. The estimates of the new factors $\widetilde\gamma$ are given in Lemma \ref{lem:widetilde_gammas} and are bounded by applying the same arguments used to control the expectations in Lemma \ref{lem:gammas} along with Remark \ref{rem:comment_integration}. 
\end{proof}
A direct consequence of the above Lemma is that  if $\beta_{\eps, t^*, t}(n)=\sup_{t^*\leq r\leq t}(r-t^*)^{1/2}\beta_{\epsilon}(n,r, t^*)$ for $t^*>0$  and $\beta_{\eps, t}(n)=\sup_{0\leq r\leq t}(t-r)^{1/2}\beta_{\epsilon}(n,r)$ when $t^*=0$, then 
\begin{equation}\label{beta_sup1}
\begin{split}
\beta_{\eps, t^*, t}(n)\lesssim\eps+\eps^{1-3\zeta}(t-t^*)^{1/2+\zeta}
\end{split}
\end{equation}
where we bound all $a_\eps$'s by 1, and 
\begin{equation}\label{beta_sup2}
\begin{split}
\beta_{\eps, t}(n)\lesssim\eps^{1-3\zeta}t^{\zeta}
\end{split}
\end{equation}

We consider a simplified version of the integral inequality for $a_\eps(n,t)$ obtained by setting $t^*=0$ and avoiding $\beta_\eps$ but instead $d_\eps$ as $A^1_\eps$ was initially defined:
\begin{equation}\label{eq_a_special}
\begin{split}
a_{\epsilon}(n,t)&\lesssim\int_{0}^{t} d\lambda \;\frac{\eps^{-1}}{[\eps^{-2}\lambda]^{1/2}}a_{\eps}(n,t-\lambda)\nonumber\\
&+\int_{0}^{t} d\lambda \;\frac{\eps^{-1}}{[\eps^{-2}\lambda]^{1/2}}\sum_{q=2}^K
 a_\eps(n-1+q,t-\lambda)\\
&+\mathbf 1_{n\geq 2}\int_{0}^{t} d\lambda\;\frac{\eps^{-1}}{[\eps^{-2}\lambda]^{1-\zeta}}
\sum_{p=2}^{\min\{n-1,K\}}
a_{\eps}(n-p,t-\lambda)\\
&+\mathbf 1_{n\geq 2}\int_{0}^{t} d\lambda\;\frac{\eps^{-1}}{[\eps^{-2}\lambda]^{1-\zeta}}\sum_{q=1}^{K}\;\;\;
\sum_{p=2}^{\min\{n,K\}}\,
a_{\eps}(n-p+q,t-\lambda)\\
&+\mathbf{1}_{\{n=2\}}\epsilon^{1-2\zeta}t^{\zeta}+\mathbf{1}_{\{3\leq n\leq K\}} \epsilon\\
&+{\bf 1}_{n\geq 2}\int_{0}^{t}d\lambda\phi_0(\lambda) \,\beta_\epsilon(n-1,t-\lambda)\\
&+  \mathbf{1}_{\{n\geq2\}}\int_{0}^{t} d\lambda\,
\phi_2^\eps(\lambda)\,a_\epsilon(n-2,t-\lambda)\,.
\end{split}
\end{equation}
This integral inequality is used for the proof of Theorem~\ref{thm:vestimate2}, for short times.

\subsection{Proof of Theorem \ref{thm:vestimate2} (Step 5)}\label{proof of main theorem}
The strategy of the proof splits into short times, i.e., $t\leq \eps^{\beta^*}$, and long times $t>\eps^{\beta^*}$. We start with the short times. Specifically, let $a_\epsilon(n)=\sup_{0\leq t\leq \eps^{\beta^*}}a_\eps(n,t)$. Based on \eqref{eq_a_special}, we bound the correlation functions on the right-hand side by their suprema over the short-time interval and integrate the remaining time-dependent factors. This yields the inequality
\begin{equation}
\begin{aligned}
&a_\epsilon(n) \lesssim
\epsilon^{\beta^*/2}
\left(
a_\epsilon(n)
+\sum_{q=2}^{K} a_\epsilon(n-1+q)
\right)
\\
&+\epsilon^{1-2\zeta+\beta^*\zeta}\left[\sum_{p=2}^{\min\{n,K\}} a_\epsilon(n-p)
+
\sum_{q=1}^{K}
\sum_{p=2}^{\min\{n,K\}}
a_\epsilon(n-p+q)
\right]
\\
&+\mathbf{1}_{\{n=2\}}
\epsilon^{1-2\zeta+\beta^*\zeta}
+\mathbf{1}_{\{3\leq n\leq K\}}\epsilon+\mathbf{1}_{\{n\geq 2\}}
\epsilon^{1-2\zeta+\beta^*(1/2+\zeta)}
+\mathbf{1}_{\{n\geq 2\}}
\epsilon^{1+\beta^*/4} a_\epsilon(n-2).
\end{aligned}
\end{equation}
We observe that the coefficients appearing on the right-hand side are of orders $\epsilon^{\beta^*/2}$, $\epsilon^{1-2\zeta+\beta^*\zeta}$, $\epsilon$, and $\epsilon^{1+\beta^*/4}$. Furthermore, the correlation functions on the right-hand side have degrees ranging from $n-2$ to $n-1+K$. We first move the terms involving $a_\epsilon(n)$ from the right-hand side to the left-hand side. For sufficiently small $\epsilon$, the resulting coefficient can be bounded from below by a positive constant independent of $\epsilon$, which can therefore be absorbed into the implicit constant. Next, we apply the same integral inequality to each of the remaining correlation functions on the right-hand side. The terms that do not involve correlation functions, namely for some $\delta>0$ implied by 
$$
\mathbf{1}_{\{n=2\}}
\epsilon^{1-2\zeta+\beta^*\zeta}
+\mathbf{1}_{\{3\leq n\leq K\}}\epsilon,
$$
are retained separately. After this second iteration, we bound all remaining correlation functions by $1$. It remains to estimate the products of the coefficients arising from the two iterations and show that their contributions are bounded by the required powers of $\epsilon$. In particular, if these contributions can be bounded by $\epsilon^{1-2\zeta+\beta^*}$, then, since
$$
\epsilon^{1-2\zeta+\beta^*}
\leq
\epsilon^{1-2\zeta+\beta^*\zeta},
$$
for $0<\epsilon<1$ and $0<\zeta<1$, the desired short-time estimate follows, provided the resulting terms also satisfy the required comparison with $\epsilon$.

\bigskip

The core idea for long times was originally developed in \cite{de1989weakly}. It boils down to showing that  
\begin{align}\label{def: a*}
a^{\star}_{\epsilon}(m,t)&=\tilde\Gamma_a^\eps(m)\sup_{0\leq\sigma\leq t}a_{\epsilon}(m,\sigma)
\end{align} for $t=T$ is bounded.  Above
\begin{eqnarray}
\widetilde \Gamma^\eps_{a}(m)&=\mathbf{1}_{\{m=1,2\}}\epsilon^{-1+2\zeta}+\mathbf{1}_{\{3\leq m\leq K\}}\epsilon^{-1}+\mathbf{1}_{\{m\geq K+1\}}\epsilon^{-1-\zeta}\label{tildegamma function:a}.
\end{eqnarray}
To do so here,  we  divide the time interval $[\eps^{\beta^*},T]$ into $l\in\mathbb{Z}$, to be chosen later, equal sub-intervals $[t_{k-1}, t_k]$  with $t_0=0$ and $t_{l}=T$. The length of each interval is given by $|t_k-t_{k-1}|=\Delta$ and $\frac{T}{\Delta}=l$.  For $m\geq 1$ and $k=1,\dots,l$, we  denote  
 \begin{align}\label{def: d*}
d^{\star}_{\epsilon}(m,k)&=\tilde\Gamma_d^\eps(m)\sup_{0\leq\sigma\leq t_k}d_{\epsilon}(m,\sigma),
\end{align}
where 
\begin{eqnarray}
 \widetilde\Gamma^\eps_{d}(m)&=\mathbf{1}_{\{1\leq m\leq K\}}\epsilon^{-1}+\mathbf{1}_{\{m\geq K+1\}}\epsilon^{-1-\zeta}\label{tildegamma function:d}.
\end{eqnarray}
We note that 
\begin{equation}
a^{\star}_{\epsilon}(0,k)=1 \qquad\textrm{and} \qquad d^{\star}_{\epsilon}(0,k)=0.
\end{equation}
Finally, recalling \eqref{def:beta}, for $m\geq 1$ we define 
\begin{equation}
\beta^{\star}_{\epsilon}(m,k)=\tilde{\Gamma}^\eps_\beta(m)\sup_{0\leq\sigma\leq t_k}\beta_{\epsilon}(m,\sigma),
\end{equation}
where $\beta_{\epsilon}(m,\sigma)=\beta_{\epsilon}(m,\sigma;0)$ as defined in \eqref{17.3.3} and  
where \begin{equation}
\tilde{\Gamma}^\eps_\beta(m)=\eps^{-1-\zeta}.
 \end{equation}
We shall consider $m\leq N^*$ where $N^*$ is defined in \eqref{N*}.
 By Corollary \ref{cor:N*}, we have, for any positive integer $m\geq  N^*+1$ and $0<t\leq T$, that
\begin{equation}\label{v for a,d}
a_{\epsilon}(m,t)\lesssim \epsilon^{1+\zeta}\qquad\textrm{and}\qquad d_{\epsilon}(m,t)\lesssim\epsilon^{1+\zeta}.
\end{equation}
For $1\leq m\leq N^*$, we also define 
\begin{eqnarray}
&\Gamma_a^\eps(m)&=\mathbf{1}_{\{m=1,2\}}\epsilon^{1-2\zeta}+\mathbf{1}_{\{3\leq m\leq K\}}\epsilon+\mathbf{1}_{\{m\geq K+1\}}\epsilon^{1+\zeta};\label{gamma function:a}\\
&
  \Gamma^\eps_{d}(m)&=\mathbf{1}_{\{1\leq m\leq K\}}\epsilon+\mathbf{1}_{\{m\geq K+1\}}\epsilon^{1+\zeta}\label{gamma function:d}.
\end{eqnarray}
Finally,  for any $m\geq 1$
\begin{equation}
 \Gamma^\eps_{\beta}(m)=\eps^{1+\zeta}.
 \end{equation}
 For $t^*=(k-1)\Delta$, $k=1,\dots,l$, and $1\leq m\leq N^*$, by Lemma \ref{lem: IIfor a for space_v_estimate}, the integral inequalities for $a^*_{\epsilon}(m,k)$ take the form:
\begin{equation}\label{split_time_integrala*}
\begin{split}
a^*_{\epsilon}(m,k)&\lesssim a^*_{\epsilon}(m,k-1)+\Delta^{1/2} a^*_{\eps}(m,k)\\
&+\Delta^{1/2}\sum_{q=2}^{K}
\,
\underbrace{\Gamma^\eps_a(m-1+q)\tilde\Gamma^\eps_a(m)}_{\leq 1}a^*_{\eps}(m-1+q,k)\mathbf 1_{m-1+q\leq N^*}\\
&+\Delta^\zeta\sum_{q=1}^{K}\;\;\;
\sum_{p=2}^{\min\{m,K\}}\,
\underbrace{\eps^{1-2\zeta}\Gamma_a(m-p+q)\tilde\Gamma^\eps_a(m)}_{\leq\eps^{1-5\zeta}}a^*_{\eps}(m-p+q,k)\mathbf 1_{\substack{m-p+q\leq N^*,\\m\geq 2}}\\
&+\Delta^\zeta\sum_{q=1}^{K}\;\;\;
\sum_{p=2}^{\min\{m,K\}}\,
\underbrace{\eps^{1-2\zeta}\tilde\Gamma^\eps_a(m)\eps^{1+\zeta}}_{\leq \eps^{1-2\zeta}}\mathbf 1_{\substack{m-p+q\geq N^*+1,\\m\geq 2}}\\
&+{\mathbf 1_{m\geq 2}\Delta^{\zeta}
\sum_{p=2}^{\min\{m-1,K\}}\underbrace{\eps^{1-2\zeta}\Gamma^\eps_a(m-p)\tilde\Gamma^\eps_a(m)}_{\leq\eps^{1-5\zeta}}\,
a^*_{\eps}(m-p,k)}\\
&+\Delta^{1/2}\sum_{q=2}^{K}\,
\underbrace{\tilde\Gamma^\eps_a(m)\eps^{1+\zeta}}_{\leq 1}\mathbf 1_{m-1+q\geq N^*+1}\\
&+\mathbf{1}_{\{m=2\}}\underbrace{\epsilon^{1-2\zeta}\tilde\Gamma^\eps_a(m)}_{=1}\Delta^{\zeta}+\mathbf{1}_{\{3\leq m\leq K\}} \underbrace{\epsilon\tilde\Gamma^\eps_a(m)}_{=1}\\
&+\mathbf 1_{m= 3}\underbrace{\epsilon\tilde\Gamma^\eps_a(m)}_{=1}\left(\Delta^{1/2}+\Delta^{1/2+\zeta}\right) +{\mathbf 1_{m\geq 4}\underbrace{\epsilon^{2-4\zeta}\tilde\Gamma^\eps_a(m)}_{\eps^{1-7\zeta}}\Delta^{\zeta}}\\
&+{\bf 1}_{m\geq 2}\left(\Delta^{1/2}+\Delta^{1/2-\zeta}\right)\underbrace{\tilde\Gamma^\eps_a(m)\Gamma^\eps_\beta(m-1)}_{\leq 1}\beta^*_\epsilon(m-1,k)\\
&+{\bf 1}_{m=2}\left(\Delta^{1/2}+\Delta^{1/4}\right)\underbrace{\eps\tilde\Gamma^\eps_a(m)}_{=\eps^{2\zeta}}\\
&+{{\bf 1}_{m\geq 3}\left(\Delta^{1/2}+\Delta^{1/4}\right)\underbrace{\eps\tilde\Gamma^\eps_a(m)\Gamma^\eps_a(m-2)}_{\leq \eps^{1-3\zeta}}a^*_\epsilon(m-2,k)},
\end{split}
\end{equation}
where we have also included the estimates of the underbraced terms, whose derivation is explained later in the proof.
The above inequality is constructed as follows. We first change variables $t-\lambda\rightarrow\lambda$, take the supremum with respect to time of $a_{\epsilon}(m,\lambda)$ and of  $\beta_{\epsilon}(m-1, k\Delta - \lambda)$, allowing these terms to be pulled out of the time integrals. Next, we multiply and divide by the corresponding powers of $\epsilon$. The factors of $\epsilon$ with negative exponents are absorbed into the terms $a^*_{\epsilon}$ and $\beta^*_{\epsilon}$ on the right-hand side, while the factors with positive exponents are collected into $\Gamma^\eps_{a}$ and $\Gamma^\eps_{\beta}$, respectively. Since $a^*_{\epsilon}$ also appears on the left-hand side, all terms on the right-hand side are multiplied by $\tilde{\Gamma}^\eps_{a}(m)$. E.g. to obtain 
$${\bf 1}_{m\geq 2}\left(\Delta^{1/2}+\Delta^{1/2-\zeta}\right)\tilde\Gamma^\eps_a(m)\Gamma^\eps_\beta(m-1)\beta^*_\epsilon(m-1,k),$$ we worked as follows:
\begin{eqnarray*}
{\bf 1}_{m\geq 2}\tilde\Gamma^\eps_{a}(m)\int_{(k-1)\Delta}^{k\Delta}d\lambda\left(\frac{1}{\lambda^{1/2}}+\frac{1}{\lambda^{1/2+\zeta}}\right)\underbrace{\eps^{1+\zeta}}_{=\Gamma^\eps_{\beta}(m-1)}\underbrace{\eps^{-1-\zeta}\sup_{0\leq\sigma\leq k\Delta}\beta_\epsilon(m-1,\sigma)}_{=\beta^*_\epsilon(m-1,k)}.
\end{eqnarray*}
All underbraced terms in \eqref{split_time_integrala*} that multiply $a^*_{\eps}$, are bounded uniformly in $\eps$ as follows: we use the worst–case bounds, namely 
$\Gamma^\eps_a(\cdot)\leq \epsilon^{\,1-2\zeta}$ and 
$\tilde\Gamma^\eps_a(\cdot)\leq \epsilon^{-1-\zeta}$, whenever the arguments of 
$\Gamma^\eps_a$ and $\tilde\Gamma^\eps_a$ permit such estimates, and by  \eqref{v for a,d}, for $m\geq 2$
\begin{equation}
\eps^{1-2\zeta}\tilde\Gamma^\eps_a(m)\eps^{1+\zeta}\mathbf 1_{m-p+q\geq N^*+1}\leq \eps^{1-2\zeta} \;\;\textrm{ and  }\; \;\tilde\Gamma^\eps_a(m)\eps^{1+\zeta}\mathbf 1_{\{m-1+q\geq N^*+1\}}\leq 1
\end{equation}
and thus these terms become smaller than 1 since $\eps<1$. Moreover, 
\begin{align*}
\eps^{1-2\zeta}\Gamma^\eps_a(m-p+q)\tilde\Gamma^\eps_a(m)&\leq \eps^{1-2\zeta}\eps^{1-2\zeta}\eps^{-1-\zeta}=\eps^{1-5\zeta}.
\end{align*}
Since $m-1+q\geq m$ for $q\geq 2$, $\Gamma^\eps_a(m-1+q)\tilde\Gamma^\eps_a(m)<1$  when $\eps<1$. Finally, $\Gamma^\eps_\beta(m-1)\tilde\Gamma^\eps_a(m)\leq\eps^{1+\zeta}\eps^{-1-\zeta}=1$ and $\eps\Gamma^\eps_a(m-2)\tilde\Gamma^\eps_a(m)\leq\eps\eps^{1-2\zeta}\eps^{-1-\zeta}=\eps^{1-3\zeta}$. We omit the estimation of the remaining terms, as it is straightforward to obtain them and is already indicated in the underbraced expressions. Also $\Delta$'s in each term of \eqref{split_time_integrala*} have a positive exponent.
We also observe that the integral inequality for  $d^{\star}_{\epsilon}(m,k)$,  thanks to  Lemma \ref{lem: IIfor d}, can be written as  
\begin{eqnarray}\label{split_time_integrald*}
&&d^*_{\epsilon}(m,k)\lesssim \mathbf 1_{\{2\leq m\leq K\}}\eps \tilde\Gamma^\eps_d(m)+\tilde\Gamma^\eps_d(m)\Gamma^\eps_{\beta}(m)\beta^*_{\epsilon}(m,k).
\end{eqnarray}
Finally, from \eqref{split_time_integral for beta}, the integral inequality for $\beta_{\epsilon}(n,t)$ can be estimated as 
\begin{eqnarray}\label{split_time_integral for beta*}
\beta^*_{\epsilon}(m,k)&\lesssim& \Delta^{-1/2}\underbrace{\epsilon \Gamma^\eps_{a}(m)\tilde\Gamma^\eps_{\beta}(m)}_{\leq \eps^{1-3\zeta}}a_{\epsilon}^*(m,k-1)\nonumber\\
&&+\Delta^\zeta\sum_{q=1}^{K}\;\;\;
\sum_{p=2}^{\min\{m,K\}}\,\underbrace{\eps^{1-2\zeta}
\Gamma^\eps_a(m-p+q)\tilde\Gamma^\eps_{\beta}(m)}_{\leq \eps^{1-5\zeta}}a^*_{\eps}(m-p+q,k)\mathbf 1_{\substack{\{m-p+q\leq N^*,\\m\geq 2\}}}\nonumber\\
&&+\Delta^\zeta\sum_{q=1}^{K}\;\;\;
\sum_{p=2}^{\min\{m,K\}}\,
\underbrace{\eps^{1-2\zeta}\tilde\Gamma^\eps_{\beta}(m)\eps^{1+\zeta}}_{\leq \eps^{1-2\zeta}}\mathbf 1_{\substack{m-p+q\geq N^*+1,\\m\geq 2}}\nonumber\\
&&+\mathbf 1_{\{m\geq 2\}}\Delta^{\zeta}
\sum_{\substack{p=2,\\p\neq m}}^{\min\{m-1,K\}}\underbrace{\eps^{1-2\zeta}\Gamma^\eps_a(m-p)\tilde\Gamma^\eps_a(m)}_{\leq\eps^{1-5\zeta}}\,
a^{{*}}_{\eps}(m-p,k)\nonumber\\
&&+\Delta^{\zeta}\underbrace{\eps^{1-2\zeta}\Gamma^\eps_{a}(m) \tilde\Gamma^\eps_{\beta}(m)}_{\leq \eps^{1-5\zeta}}a^*_{\eps}(m,k)\nonumber\\
&&+\Delta^{\zeta}\sum_{q=2}^{K}
\underbrace{\epsilon^{1-2\zeta}
\Gamma^\eps_a(m-1+q)\tilde\Gamma^\eps_{\beta}(m)}_{\leq \eps^{1-5\zeta}}a^*_{\eps}(m-1+q,k)\mathbf 1_{\{m-1+q\leq N^*\}}\nonumber\\
&&+\Delta^{\zeta}\sum_{q=2}^{K}\,
\underbrace{\epsilon^{1-2\zeta}\tilde\Gamma^\eps_{\beta}(m)\eps^{1+\zeta}}_{\leq \eps^{1-2\zeta}}\mathbf 1_{\{m-1+q\geq N^*+1\}}\nonumber\\
 &&+\mathbf{1}_{\{m\geq2\}}\Delta^{\zeta}
\underbrace{\eps^{3/2-2\zeta}\tilde\Gamma^\eps_{\beta}(m)\Gamma^\eps_a(m-2)}_{\leq \eps^{3/2-5\zeta}}a^*_{\epsilon}(m-2,k)\nonumber\\
 &&+\mathbf{1}_{\{m\geq2\}}\Delta^{\zeta}
\underbrace{\epsilon^{1-3\zeta}\tilde\Gamma^\eps_\beta(m)\Gamma^\eps_d(m-1)}_{\leq \eps^{1-4\zeta}}d^*_{\epsilon}(m-1,k).
\end{eqnarray}
Because the product of $\tilde{\Gamma}^\eps_{\beta}(m)$ with any of $\Gamma^\eps_{a}(\cdot)$, $\Gamma^\eps_{\beta}(\cdot)$, or $\Gamma^\eps_{d}(\cdot)$ may grow as fast as $\epsilon^{-3\zeta}$, while all terms are multiplied by $\epsilon^{1 - x}$ with $x \in \{0, -2\zeta, -\frac{1}{2} +2\zeta, -3\zeta\}$, it follows that each coefficient remains bounded by a positive power of $\epsilon$, and hence is uniformly bounded by 1 when $\epsilon < 1$. We build two vectors
\begin{align*}
    \underline{V}(k) &= [
           a^{*}_{\epsilon}(1,k)\,\cdots 
           a^{*}_{\epsilon}(N^*,k)\, \,
           d^{*}_{\epsilon}(1,k)\,
           \cdots \,
           d^{*}_{\epsilon}(N^*,k)\,
           \beta^{*}_{\epsilon}(1,k)\,
           \cdots \,
           \beta^{*}_{\epsilon}(N^*,k)
       ]^T
  \end{align*}
  and  \begin{align*}
   u_{\epsilon,\Delta}(k) = [
          u_{\epsilon, \Delta}(1,k) \,
           \cdots \,
           u_{\epsilon, \Delta}(3N^*,k)]^T
        \end{align*}
for $k=1,\dots,l$, the vector version of \eqref{split_time_integrala*},\eqref{split_time_integrald*},\eqref{split_time_integral for beta*} can be written as
\begin{equation}\label{matrix inequality}
\underline{V}(k)\leq A_{\epsilon,\Delta}\,\underline{V}(k-1)+\Lambda_{\epsilon,\Delta}(k)\underline{V}(k)+u_{\epsilon,\Delta}(k),\;\;\;\;k=1,\dots,l.
\end{equation}
The elements of $A_{\epsilon,\Delta}$are given by
\[
A_{\varepsilon,\Delta}(i,j)=
\begin{cases}
1, & 1 \le i = j \le N^*, \\[2mm]
\Delta^{-1/2}\,\varepsilon\,\Gamma_a^\varepsilon(m)\,\tilde\Gamma_\beta^\varepsilon(m), & i = 2N^* + j,\; 1 \le j \le N^*, \\[1mm]
0, & \text{otherwise.}
\end{cases}
\]
The elements of the matrix $\Lambda_{\epsilon,\Delta}(k)$, which are all bounded by a constant (because are bounded either by a positive power of $\Delta$, or by a positive power of $\epsilon$ multiplied by a positive power of $\Delta$),  can be figured out in a similar manner, and we leave those details to the reader. 
Finally, the matrix $u_{\epsilon,\Delta}(k)$ consists of all terms in 
\eqref{split_time_integrala*}, \eqref{split_time_integrald*} and \eqref{split_time_integral for beta*}
that do not involve $a^*_\epsilon$, $d^*_\epsilon$, or $\beta^*_\epsilon$, 
and are uniformly bounded in $\eps$ as well.
We choose $\Delta$ such that 
$
\|\Lambda_{\epsilon,\Delta}(k)\|_{\infty} < \frac{1}{2},
$
where $\|\cdot\|_{\infty}$ denotes the matrix norm defined as the maximum absolute row sum. 
More generally, we may also choose $\Delta$ such that 
$\|\Lambda_{\epsilon,\Delta}(k)\|_{\max} \leq C$ for some constant $C < 1$.  For every $k = 1, \dots, l+1$, the matrices $\Lambda_{\epsilon,\Delta}(k)$ and $u_{\epsilon,\Delta}(k)$ 
are uniformly bounded by $1/2$ and some constant $c$, respectively, independently of $\epsilon$. 
Consequently, all components of $\underline{V}(k)$ are uniformly bounded in $\epsilon$ for all $k = 1, \dots, l+1$, 
which proves the result. 
\begin{corollary}\label{cor: thm for space v}
Let  $K\geq 2$ be a fixed  integer number  and $\zeta>0$. For any initial product measure $\mu^{\epsilon}$ satisfying \eqref{Ass1},  and time  $0<t\leq T$
\begin{equation}\label{d}
d_{\epsilon}(n,t)\lesssim
\begin{cases}
\epsilon,&\;1\leq n\leq K\\
 \epsilon^{1+\zeta}, &n\geq K+1,
\end{cases}
\end{equation}
and 
$
\beta_{\epsilon}(n,t)\lesssim \epsilon^{1+\zeta}.
$
 \end{corollary}
From the previous result, we immediately obtain the following bounds, which we state as a corollary for later convenience. We omit the proof, as it follows directly from the argument used to prove the preceding result. Nevertheless, for the reader's convenience, we briefly sketch the main steps in the corresponding section where the corollary is applied.

\begin{corollary}\label{cor: thm for space v_initial}
Let  $K\geq 2$ be a fixed  integer number  and $\zeta>0$. For any initial product measure $\mu^{\epsilon}$ satisfying \eqref{Ass1},  and time  $0<t\leq T$
\begin{equation}\label{d_initial}
\sup_{\underline{x}\in\Lambda_{N}^{n,\neq}}|v^{\epsilon}_{n}( \underline{x},t)-v^{\epsilon}_{n}( \underline{x}+\und e_1,t)|\lesssim
\begin{cases}
\epsilon,&\;1\leq n\leq K\\
 \epsilon^{1+\zeta}, &n\geq K+1,
\end{cases}
\end{equation}
where $\und e_1$ is the unit vector in the positive 1st-direction.
 \end{corollary}
This corollary is the same as Corollary \ref{cor: thm for space v}.  In fact, Corollary \ref{cor: thm for space v} is expressed in terms of $d_\eps$ including also the estimate for $\beta_\eps$. We do not present the proof as it follows easily from the proof of the previous result.

 \section{Further results on space $v$-functions}
In this section, we revisit the estimates for the space $v$-functions and derive sharper bounds that will be useful in the subsequent analysis of space-time correlations. We first improve the estimates for correlations of sufficiently high degree by exploiting the results established in Theorem~\ref{thm:vestimate2}. We then turn our attention to the remaining lower-degree correlations, for which a more refined analysis is required.

\subsection{Space $v$-functions for sufficiently high degree} Having established Theorem~\ref{thm:vestimate2}, we return to
\eqref{split_time_integral for main thm} and
\eqref{split_time_integral for d}--\eqref{split_time_integral for beta}
to derive sharper estimates for the space correlation functions
when $n$ is sufficiently large. Recall that, in the proof of Theorem~\ref{thm:vestimate2}
in Section~\ref{proof of main theorem}, we established
\begin{equation}
\sup_{0<t\leq \eps^{\beta^*}}
\sup_{\underline{x}\in\Lambda_N^{n,\neq}}
|v^{\epsilon}_{n}(\underline{x},t)|
\leq c\epsilon^{1+\delta},
\qquad n\geq K+1,
\end{equation}
for some $\delta>0$. Combining this estimate with
Lemma~\ref{lem: IIfor d}, we obtain
\begin{equation}\label{split_time_integral for beta_reestimated}
\beta_{\epsilon}(n,t,t^*)
\lesssim
(t-t^*)^{-1/2}\eps^{2+\delta}
{\bf 1}_{\{t^*>0\}}
+(t-t^*)^\zeta\eps^{2-4\zeta},
\qquad n\geq K+3.
\end{equation}
Substituting this bound into
\eqref{split_time_integral for main thm}
then yields
\begin{equation}\label{a_reestimated}
\sup_{0<r\leq t}a_{\epsilon}(n,r)
\lesssim \eps^{2-6\zeta}t^\zeta,
\qquad n\geq K+3.
\end{equation}
Since $\beta^*<2$, this implies
\begin{equation}\label{a_reestimated*}
\sup_{0<t\leq \eps^{\beta^*}}a_{\epsilon}(n,t)
\lesssim \eps^{2-6\zeta+\beta^*\zeta},
\qquad n\geq K+3.
\end{equation}

\subsection{Space $v$-functions of lower degree}
Before we move on to estimating the space-time $v$-function, we anticipate that the results we have previously obtained for the space $v$-function have to be refined in such a way that the bounds in $\epsilon$ improve when estimating the space time function. Motivated by the integral inequality for the space-time $v$-function, given in Definition~\ref{def:space time correlations}, namely
\begin{eqnarray*}
v^{\epsilon,m,n}_{s,r}(\underline{y}, \underline{x})
&=&\sum_{\underline{w}}\mathbf{P}_{\epsilon}(\underline{x}\overset{r-s}{\longrightarrow} \underline{w})v^{\epsilon,m,n}_{s,s}( \underline{y},\underline{w})
+\int_{s}^{r}d\lambda\,\sum_{\underline{w}}\mathbf{P}_{\epsilon}(\underline{x}\overset{r-\lambda}{\longrightarrow} \underline{w})(C_{\epsilon} v^{\epsilon,m,n}_{s,\lambda})(\underline{y}, \underline{w}),
\end{eqnarray*}
which is discussed in detail in Section \ref{sec:space_time_corr}, we aim to improve {the obtained estimates with respect to  $\epsilon$ for the space $v$-function. In particular, a better estimate of the initial condition, which depends on the space $v$-function but also on the transition probability of the stirring process, 
\[
\sum_{\underline{w}}\mathbf{P}_{\epsilon}(\underline{x}\overset{r-s}{\longrightarrow} \underline{w})v^{\epsilon,m,n}_{s,s}(\underline{w}, \underline{y})
\]
 is required. Our strategy is then to split } the above sum into the cases $\underline{w} = \underline{y}$ and $\underline{w} \neq \underline{y}$, the latter includes the case where $|\und w\cap \und y|=0$ (i.e., none of the labeled particles in $\underline{w}$ occupy the same site as any of the particles in $\underline{y}$) and, from Liggett's inequality given in  {\eqref{eq:psi_n}}, it can be  bounded from above by 
\[
\sum_{\pi\in S_n}\sum_{{\substack{\underline{w} = (w_1,\dots,w_n):\\|\underline{w} \cap \underline{y}| = 0}}} \prod_{i=1}^n P^{(\epsilon)}_{\lambda}(x_{\pi(i)},w_i)\, |v_{n+m}^{\epsilon}(\underline{w}\cup  \underline{y}, s)|.
\]
In Lemma~\ref{lem: improved v2}, we show that, for general times, the sum of products of transition probabilities multiplying the $v$'s, i.e., {for any $\pi\in S_n$},
\[
\sum_{{\substack{\underline{w} = (w_1,\dots,w_n):\\|\underline{w} \cap \underline{y}| = 0}}} \prod_{i=1}^n P_{\tau}^{(\epsilon)}(x_{\pi(i)}, w_i)\, |v_{n+m}^{\epsilon}(\underline{w}\cup  \underline{y}, t)|,\quad \tau\neq t
\]
improves the order in $\epsilon$ by a factor of $\epsilon^{2\zeta}$ for $n+m=2$ and by $\epsilon^{\zeta}$ for $4 \leq n+m \leq K$. For the remaining cases of $n+m$, the order in $\epsilon$ remains unchanged, meaning that we obtain a bound of order $\epsilon$ when $n + m = 2, 3$, and of order $\epsilon^{1 + \zeta}$ when $n + m \geq K+1$. 

We begin with Corollary~\ref{cor:final v inequality}, in which we summarize the integral inequalities for $v_n^{\epsilon}$ in a form that preserves those terms estimated by $\epsilon$ to a positive power less than or equal to~1 in their original form, while by using Theorem \ref{thm:vestimate2}, the remaining terms are already bounded by $\epsilon^{1+\zeta}$ and thus require no further treatment.

\begin{corollary}\label{cor:final v inequality}
Let  $2\leq K\leq N$ be a fixed  integer number  and   $\zeta>0$. Let also  
 $\mu^{\epsilon}$  be any initial measure satisfying \eqref{Ass1},   $0<t\leq T$ and  $2\leq n\leq K$. Then, the following bound holds:
\begin{equation}\label{split_time_integral of v}
\begin{split}
&|v_n^{\epsilon}(\underline{x},t)|\lesssim \mathbf 1_{\{n=2,3\}}\epsilon T^{1/2}+\mathbf 1_{\{4\leq n\leq K\}}\epsilon^{1+\zeta}T^{1/2}\\ 
&\hspace{1cm}+\int_{0}^t d\lambda \;\eps^{-1}\sum_{j=1}^{n}\sum_{z'\in I_{+}} \mathbf 1_{\{n\leq K\}}\mathbf{E}_\epsilon \left[\mathbf{1}_{\{x_j(\lambda)\in I_+\}} |v_{n}^\epsilon\left(\underline{x}^{(j)}(\lambda)\cup z',t-\lambda\right)|\Big]\right]\\
&\hspace{1cm}+\int_{0}^t d\lambda \;\frac{\eps^{-1}}{[\eps^{-2}\lambda]^{1/2}+1}\sum_{j=1}^{n}\sum_{\underline{z}'\subset I_{+},q\geq 2} \mathbf 1_{\{n+q\leq N^*+1\}}\mathbf{E}_\epsilon \left[ |v_{n-1+q}^\epsilon\big(\underline{x}^{(j)}(\lambda)\cup\underline{z}',t-\lambda\big)|\Big]\right]\\
&\hspace{1cm}+\int_{0}^t d\lambda \;
\eps^{-1}\mathbf{E}_\epsilon \left[ {\bf 1}_{\und{x}(\lambda)=I^n_+}\right] {\bf 1}_{\{n\leq K\}}.
\end{split}
\end{equation}
\end{corollary}
 \begin{proof}
 The result follows from \eqref{split_time_integral} {with $t^*=0$}, Corollary \ref{cor: thm for space v} and Theorem \ref{thm:vestimate2}. More precisely, taking  \eqref{split_time_integral} {with $t^*=0$}, and noting that $a_{\epsilon}(n,0)=0$,  we get the bound:
\begin{equation*}
\begin{split}
|v_n^{\epsilon}(\underline{x},t)|&\lesssim\int_{0}^t d\lambda\;\frac{\eps^{-1}}{[\eps^{-2}\lambda]^{1-\zeta}}\eps^{1-2\zeta}\nonumber\\
&+\int_{0}^t d\lambda \;\eps^{-1}\sum_{j=1}^{n}\sum_{z'\in I_{+}}
 \mathbf{E}_\epsilon \left[\mathbf{1}_{\{x_j(\lambda)\in I_{+}\}} |v_{n}^\epsilon(\und x^{(j)}{(\lambda)}\cup z',t-\lambda)|\Big]\right]\nonumber\\
&+\int_{0}^t d\lambda \;\frac{\eps^{-1}}{[\eps^{-2}\lambda]^{1/2}+1}\sum_{j=1}^{n}\sum_{\substack{\underline{z}'\subset I_{+},\\q\geq2}}
 \mathbf{E}_\epsilon \left[ |v_{n-1+q}^\epsilon\big(\underline{x}^{(j)}(\lambda)\cup \underline{z}',t-\lambda\big)|\Big]\right]\nonumber\\
&+\int_{0}^t d\lambda\;\phi_0(\lambda)\eps^{1+\zeta}+  \int_{0}^t d\lambda\;\phi_2^\epsilon(\lambda)\eps^{1-2\zeta}\\
&+\int_{0}^t d\lambda \;
\eps^{-1}\mathbf{E}_\epsilon \left[ {\bf 1}_{\und{x}(\lambda)=I^n_+}\right] {\bf 1}_{\{n\leq K\}}.
\end{split}
\end{equation*} And now a simple computation allows to conclude the proof.
 \end{proof}

 \begin{lemma}\label{lem: improved v2}
We fix  an  integer number $K$ such that $2\leq K\leq N$. Let $\und x\in \Lambda_N^{\neq, n}$ and  $\und y\in \Lambda_N^{\neq, m}$ such that  $2\leq n+m\leq K$. For any $\,0<\tau, t\leq T$, with $t\neq \tau$, {and for any $\pi\in S_n$,} the following bound holds:
\begin{equation}\label{eq: improved v2!}
\sum_{{\substack{\und w=(w_1,\dots,w_n):\\| \und w\,\cap \
\und y|=0}}}\;\;\prod_{i=1}^n P_{\tau}^{(\epsilon)}(x_{\pi(i)},w_i)|v_{n+m}^{\epsilon}(\underline{w},\und y,t)|\lesssim\begin{cases}
\eps+\frac{1}{[\eps^{-2}\tau]^{1/2}},&\;\;n+m=2,3\\
\epsilon^{1+\zeta}+\frac{\epsilon}{[\eps^{-2}\tau]^{1/2}},
&\;\; 4\leq n+m\leq K
\end{cases}
\end{equation}
where the constant in $\lesssim$ depends on $K,T$ and it is bounded by $(K+1)^2 e^{(K+1)^3K^4\pi T}K^{K+1}(1+T)^{(K+1)/2}$. 
\end{lemma}

\begin{remark}
 Note that the bound in \eqref{eq: improved v2!} is independent of $t$. This is because $t$ appears only through a positive exponent, which we may uniformly control by replacing $t$ with its maximal value $T$. Consequently, the resulting estimate depends on $T$ but not on  $t\in[0,T]$. For our convenience in the calculations later, we call 
\begin{equation}\label{def:D}
D^{\eps}_{n+m}(\tau):=
\begin{cases}
\eps+\frac{1}{[\eps^{-2}\tau]^{1/2}},&\;\; n+m=2,3\\
\epsilon^{1+\zeta}+\frac{\epsilon}{[\eps^{-2}\tau]^{1/2}},
&\;\; 4\leq n+m\leq K.
\end{cases}
\end{equation}
\end{remark}

\begin{proof}
We note that we first prove the lemma for the identity permutation $\pi(i)=i$ for all $i\in\{1,\dots,n\}$ in \eqref{eq: improved v2!}.  
The proof then extends directly to any $\pi\in S_n$, as the same arguments apply. Let $(n,m)$ such that $n,m\geq 1$ and  $n+m=K$. We start by working only with $v^\eps_{K}$, without considering the sum of product of probabilities given in \eqref{eq: improved v2!}. To simplify the next formulas we denote $\und{u}=\und{w}\cup\und{y}$.  By Corollary \ref{cor:final v inequality} 
\begin{eqnarray}
              \label{13.1.111 with gammas 2_2}
&& |v_{K}^{\epsilon}(\underline{u},t)|\lesssim 
\epsilon^{1+\zeta}T^{1/2}+\int_{0}^{t} d\lambda_1 \;
\eps^{-1}\mathbf{E}_{\epsilon, \und{u}} \left[ {\bf 1}_{\und{x}(\lambda_1)=I^{K}_+}\right]\nonumber\\
    &&\hspace{1.6cm}+\sum_{j_1=1}^{K}\sum_{z_1'\in I_+}\int_{0}^{t} d\lambda_1   
\mathbf{E}_{\epsilon, \und{u}} \Big[  \mathbf{1}_{\{x_{j_1}(\lambda_1)\in I_{+}\}} \underbrace{|v_{K}^\epsilon\big(\underline{x}^{(j_1)}(\lambda_1)\cup z_1',t-\lambda_1\big)|}_{\textrm{we bound it using  \eqref{13.1.111 with gammas 2_2} with $t-\lambda_1$} }\Big],
  \end{eqnarray} where the third line in \eqref{split_time_integral of v} is estimated as $\eps^{1+\zeta}T^{\zeta}$ because the degree of all $v$'s included in this term satisfy $n-1+q\geq K+1$ for $q\geq 2$, and thus \eqref{est: v for large n} applies. This argument cannot be applied to $|v_{K}^\epsilon\big(\underline{x}^{(j)}(\lambda_1)\cup z',t-\lambda_1\big)|$ appearing in the second line of \eqref{split_time_integral of v} since its degree satisfies $K<K+1$.  We handle this term as follows: we eliminate $|v_{K}^\epsilon\big(\underline{x}^{(j)}(\lambda_1)\cup z',t-\lambda_1\big)|$ from \eqref{13.1.111 with gammas 2_2} by bounding it using the bound \eqref{13.1.111 with gammas 2_2}.   This introduces  the sum of the following additional terms to  \eqref{13.1.111 with gammas 2_2}:
 \medskip
 
\noindent $\bullet$ Additional term 1:
\begin{eqnarray*}\label{Add_Term1}
K \epsilon^{1+\zeta}T^{1/2}\sum_{j_1=1}^{K}\int_{0}^{t} d\lambda_1   \eps^{-1}\mathbf{E}_{\epsilon, \und{u}} \left[ \mathbf{1}_{\{x_{j_1}(\lambda_1)\in I_{+}\}} \right].
\end{eqnarray*}
Note that the factor  $K$ is coming from  the sum $\sum_{z_1'\in I_+}$ appearing in \eqref{13.1.111 with gammas 2_2}. \\
$\bullet$ Additional term 2:
\begin{eqnarray*}\label{Add_Term2}\epsilon^{-2}\sum_{j_1=1}^{K}\sum_{z_1'\in I_+}\int_{0}^{t} d\lambda_1 \int_{\lambda_1}^{t} d\lambda_2  
\mathbf{E}_{\epsilon,\und{u}} \left[  \mathbf{1}_{\{x_{j_1}(\lambda_1)\in I_{+}\}} \mathbf{E}_{\epsilon, \underline{x}^{(j_1)}(\lambda_1)}\left[{\bf 1}_{\und{x}(\lambda_2)=I^{K}_+}\right]\right].
\end{eqnarray*}
$\bullet$ Additional term 3:
\begin{eqnarray*}\label{Add_Term3}
\epsilon^{-2}\!\!\sum_{j_1=1}^{K}\sum_{z_1'\in I_+}\sum_{j_2=1}^{K}\sum_{z_2'\in I_+}\int_{0}^{t} \!\!\!\!d\lambda_1\!\! \int_{\lambda_1}^{t} \!d\lambda_2\,\mathbf{E}_{\epsilon,\und{u}} [ \mathbf{1}_{\{x_{j_1}(\lambda_1)\in I_{+}\}}  \mathbf{E}_{\epsilon, \underline{x}^{(j_1)}(\lambda_1),\,z_1'}[\underbrace{|v_{K}^\epsilon\big(\underline{x}^{(j_2)}(\lambda_2)\cup z'_2,t-\lambda_2\big)|}_{\textrm{we bound it using \eqref{13.1.111 with gammas 2_2} for $t-\lambda_2$} }]].
\end{eqnarray*}
By iterating this procedure (i.e., every time bounding $|v_{K}^\epsilon\big(\underline{x}^{(j_k)}(\lambda_k)\cup z'_k,t-\lambda_k\big)|$ by itself in the additional term 3), we obtain an infinite series whose $k$-th term is the sum of the following terms:
\begin{equation}\label{k-th term in the exponential_positive exponent 1}
\begin{split}
&\hskip1cm\epsilon^{1+\zeta}T^{1/2}\sum_{j_1=1}^{K}\sum_{z_1'\in I_+}\cdots \sum_{j_k=1}^{K}\sum_{z_k'\in I_+} \int_{0}^{t} d\lambda_1\cdots\int_{\lambda_{k-1}}^{t} d\lambda_k \\
&\hskip2cm\times\eps^{-k}\mathbf{E}_{\epsilon, \und{u}} \left[ \mathbf{1}_{\{x_{j_1}(\lambda_1)\in I_{+}\}}  \left[\mathbf{E}_{\epsilon, \underline{x}^{(j_1)}(\lambda_1),\,z_1'}\left[\cdots\mathbf{E}_{\epsilon, \underline{x}^{(j_{k-1})}(\lambda_{k-1}),\,z_{k-1}'}\left[\mathbf{1}_{\{x_{j_{k}}(\lambda_k)\in I_{+}\}}\right]\cdots\right]\right]\right]
\end{split}
\end{equation}
 and 
\begin{equation}\label{k-th term in the exponential_positive exponent2}
\begin{split}
&\eps^{-k}\sum_{j_1=1}^{K}\sum_{z_1'\in I_+}\cdots \sum_{j_k=1}^{K}\sum_{z_k'\in I_+}\int_{0}^{t} d\lambda_1\cdots\int_{\lambda_{k}}^{t} d\lambda_{k+1}  \eps^{-1}\\
&\hskip1.5cm\times
\mathbf{E}_{\epsilon, \und{u}} \left[ \mathbf{1}_{\{x_{j_1}(\lambda_1)\in I_{+}\}}  \mathbf{E}_{\epsilon, \underline{x}^{(j_1)}(\lambda_1),\,z_1'}\left[\cdots\left[\mathbf{1}_{\{x_{j_k}(\lambda_k)\in I_{+}\}}\mathbf{E}_{\epsilon, \underline{x}^{(j_k)}(\lambda_k),\,z_k'}\left[{\bf 1}_{\und{x}(\lambda_{k+1})=I^{K}_+}\right]\right]\cdots\right]\right].
\end{split}
\end{equation}
Note that in the configurations appearing in the index in the expectations above, and in those used later in the proof, we omit the symbol ``$\cup$'' and replace by comma, e.g. $\mathbf{E}_{\epsilon, \underline{x}^{(j_k)}(\lambda_k),\,z_k'}[\cdot]=\mathbf{E}_{\epsilon, \underline{x}^{(j_k)}(\lambda_k)\,\cup \,z_k'}[\cdot]$. The first term  \eqref{k-th term in the exponential_positive exponent 1} is estimated as follows: we start with estimating the last expectation
$$\mathbf{E}_{\epsilon, \underline{x}^{(j_{k-1})}(\lambda_{k-1}),\,z_{k-1}'}\left[\mathbf{1}_{\{x_{j_{k}}(\lambda_k)\in I_{+}\}}\right]\lesssim \frac{1}{[\eps^{-2}(\lambda_k-\lambda_{k-1})]^{1/2}+1}.$$
Similarly,
$$\mathbf{E}_{\epsilon, \underline{x}^{(j_{k-2})}(\lambda_{k-2}),\,z_{k-2}'}\left[\mathbf{1}_{\{x_{j_{k-1}}(\lambda_{k-1})\in I_{+}\}}\right]\lesssim \frac{1}{[\eps^{-2}(\lambda_{k-1}-\lambda_{k-2})]^{1/2}+1}.$$
By keep estimating in the same manner, the first term is eventually bounded by 
\begin{equation}\label{k-th term in the exponential_positive exponent, term 1}
K^{2k}\epsilon^{1+\zeta}T^{1/2}\int_0^{t}
\frac 1{\sqrt{\lambda_1}}\;d\lambda_1\int_{\lambda_1}^{t}\frac 1{\sqrt{\lambda_2-\lambda_1}}d\lambda_2\;\dots \int_{\lambda_{k-1}}^{t}
    \frac{1}{\sqrt{\lambda_k-\lambda_{k-1}}}d\lambda_k.
\end{equation}
The remaining part is devoted to estimating the second term  \eqref{k-th term in the exponential_positive exponent2},  making use of 
\begin{equation}\label{sum_prod_transition_prob}\sum_{{\substack{\und w=(w_1,\dots,w_n):\\| \und w\,\cap \
\und y|=0}}}\prod_{i=1}^n P^{(\epsilon)}_{\tau}(x_i,w_i),
\end{equation}
(which was not used for the estimation of the first term \eqref{k-th term in the exponential_positive exponent 1}, as we simply bounded \eqref{sum_prod_transition_prob} by 1) and showing that it is bounded by 
\begin{equation}\label{k-th term in the exponential_positive exponent, term 2}
K^{3k}\frac{\epsilon}{[\eps^{-2}\tau]^{1/2}+1} \int_0^{t}
\frac 1{\sqrt{\lambda_1}}\;d\lambda_1\int_{\lambda_1}^{t}\frac 1{\sqrt{\lambda_2-\lambda_1}}d\lambda_2\;\dots \int_{\lambda_{k-1}}^{t}
    \frac{1}{\sqrt{\lambda_k-\lambda_{k-1}}}d\lambda_k.
\end{equation}
We call $a_0(t):=1$ and for $k\geq 1$ with $0=\lambda_0<\lambda_1<\lambda_2<\dots<\lambda_{k}<t$
\begin{equation}
        \label{a10.8}
        a_k(t):=\int_0^{t}
\frac 1{\sqrt{\lambda_1}}\;d\lambda_1\int_{\lambda_1}^{t}\frac 1{\sqrt{\lambda_2-\lambda_1}}d\lambda_2\;\dots \int_{\lambda_{k-1}}^{t}
    \frac{1}{\sqrt{\lambda_k-\lambda_{k-1}}}d\lambda_k.
\end{equation}
Then,  by \eqref{k-th term in the exponential_positive exponent, term 1} and \eqref{k-th term in the exponential_positive exponent, term 2}, it is immediate that when $n+m=K$ and $\tau\neq t$
\begin{eqnarray}\label{a10.8*}
&&\sum_{{\substack{\und w=(w_1,\dots,w_n):\\| \und w\,\cap \
\und y|=0}}}\prod_{i=1}^n P^{(\epsilon)}_{\tau}(x_i,w_i)|v_{K}^{\epsilon}(\und{u},t)|\\&&\hspace{2cm}\leq \left(\epsilon^{1+\zeta}T^{1/2}+\frac{\epsilon }{[\eps^{-2}\tau]^{1/2}}\right)\sum_{k=0}^\infty  K^{3k} a_k(t).
\end{eqnarray}
 By  Lemma 10.4 in \cite{de2012truncated}, the series in last display is bounded by $2e^{K^6\pi t}$, and thus the right-hand side of \eqref{a10.8*} is bounded by\begin{equation}\label{a10.8**}\underbrace{2e ^{(K+1)^2K^4\pi T}(1+T^{1/2})}_{\textrm{constant for $n+m=K$}}\left(\epsilon^{1+\zeta}+\frac{\epsilon }{[\eps^{-2}\tau]^{1/2}+1}\right),
 \end{equation}
where we used that $e ^{K^6\pi T}\leq e ^{(K+1)^2K^4\pi T}$.
\medskip

Now we prove \eqref{k-th term in the exponential_positive exponent2}. All calculations are performed by keeping track of the history of each particle and the events happened at $\lambda_1,\dots,\lambda_{k+1}$, i.e.
$$\mathbf{E}_{\epsilon, \und{u}} \left[ \mathbf{1}_{\{x_{j_1}(\lambda_1)\in I_{+}\}}  \left[\mathbf{E}_{\epsilon, \underline{x}^{(j_1)}(\lambda_1),\,z_1'}\left[\cdots\mathbf{E}_{\epsilon, \underline{x}^{(j_{k-1})}(\lambda_{k-1}),\,z_{k-1}'}\left[\mathbf{1}_{\{x_{j_{k}}(\lambda_k)\in I_{+}\}}\right]\cdots\right]\right]\right],$$
where starting from the configuration $\und u$, at time $\lambda_1$, the particle labeled by $j_1$ reaches the boundary, dies, and a new particle is born at the site $z_1'\in I_{+}$. The particles in the resulting configuration $(\underline{x}^{(j_1)}(\lambda_1), z_1')$ then move within the system until, at time $\lambda_2$, the particle labeled by $j_2$ reaches the boundary, dies, and a new particle is born at the site $z_2'\in I_{+}$. The particles continue evolving in this manner until, at time $\lambda_{k+1}$, all particles in the configuration $(\underline{x}^{(j_k)}(\lambda_k), z_k')$ occupy the last $K$ sites at the boundary. All said events are then estimated using Gaussian kernels, meaning \eqref{N4.5} and \eqref{N4.5 additional}.  
\medskip

We present the calculations for  $k=2$ (i.e., 2-iterations) and easily can be adopted for general $k$. By developing the expectation,  using Liggett's inequality given in \eqref{Liggett_new} and Kolmogorov`s identity, we obtain: for simplicity, let us call the initial configuration $\und u=(\und w, \und y)$, then we first compute 
\begin{eqnarray}
&&\hspace{-.6cm}
\mathbf{E}_{\epsilon,\und u} \left[ \mathbf{1}_{\{x_{j_1}(\lambda_1)\in I_{+}\}}  \mathbf{E}_{\epsilon, \underline{x}^{(j_1)}(\lambda_1),\,z_1'}\left[\mathbf{1}_{\{x_{j_2}(\lambda_2)\in I_{+}\}}\mathbf{E}_{\epsilon, \underline{x}^{(j_2)}(\lambda_2),\,z_2'}\left[{\bf 1}_{\und{x}(\lambda_{3})=I^{K}_+}\right]\right]\right]\nonumber\\
&&\lesssim\sum_{z_1\in I_+}\sum_{z_2\in I_+}\sum_{\pi\in S_{n+m}}\prod_{i\neq j_1, j_2}P^{(\eps)}_{\lambda_3}(u_{\pi(i)},N-K+i)\nonumber\\
&&\hspace{1cm}\times P^{(\eps)}_{\lambda_3-\lambda_1}(z_1',N-K+j_1) P^{(\eps)}_{\lambda_1}(u_{\pi(j_1)},z_1)\nonumber\\
&&\hspace{1cm}\times P^{(\eps)}_{\lambda_3-\lambda_2}(z_2',N-K+j_2) P^{(\eps)}_{\lambda_2}(u_{\pi(j_2)},z_2),
\end{eqnarray}
where $S_{n+m}$ is the set of permutations of $\{1,\dots,n+m\}$, and  $z_1,z_2$ come from the fact that we write
\[
\mathbf{1}_{\{x_{j_1}(\lambda_1)\in I_+\}}=\sum_{z_1\in I_+}\mathbf{1}_{\{x_{j_1}(\lambda_1)=z_1\}} \quad\textrm{ and } \quad\mathbf{1}_{\{x_{j_2}(\lambda_2)\in I_+\}}=\sum_{z_2\in I_+}\mathbf{1}_{\{x_{j_2}(\lambda_2)=z_2\}}.
\]
Recalling that  $\und u=\und w\,\cup\,\und y$, the $n$ first particles are $\und w$. For any $\pi\in S_{n+m}$, we focus on $\{u_{\pi(i)}:i\}=\{w_1,\dots,w_n\}$. For all these $u_{\pi(i)}$, we use that 
$$\sum_{{\substack{\und w=(w_1,\dots,w_n):\\| \und w\,\cap \
\und y|=0}}}
\prod_{i=1}^n \left(P_{\tau}^{(\epsilon)}(x_i,w_i)P^{(\eps)}_{s}(w_i,\cdot)\right)\leq\prod_{i=1}^nP^{(\eps)}_{s+\tau}(x_i,\cdot),$$
where $s$ can be any of $\lambda_1,\lambda_2$ and $\lambda_3$ and $"\cdot"=N-K+i, z_1,z_2$.
Overall, e.g., $u_{\pi(j_1)}\in\und x$ and $u_{\pi(j_2)}\notin\und x$
\begin{equation}\label{2-th term in the exponential_positive exponent}
\begin{split}
&\sum_{{\substack{\und w=(w_1,\dots,w_n):\\| \und w\,\cap \
\und y|=0}}}
\prod_{i=1}^n P_{\tau}^{(\epsilon)}(x_i,w_i)
\mathbf{E}_{\epsilon,\und u} \left[ \mathbf{1}_{\{x_{j_1}(\lambda_1)\in I_{+}\}}  \mathbf{E}_{\epsilon, \underline{x}^{(j_1)}(\lambda_1),\,z_1'}\left[\mathbf{1}_{\{x_{j_2}(\lambda_2)\in I_{+}\}}\mathbf{E}_{\epsilon, \underline{x}^{(j_2)}(\lambda_2),\,z_2'}\left[{\bf 1}_{\und{x}(\lambda_{3})=I^{K}_+}\right]\right]\right]\\
&\lesssim \frac{1}{[\eps^{-2}(\tau+\lambda_3)]^{(n-1)/2}+1}\frac{1}{[\eps^{-2}\lambda_3]^{(m-1)/2}+1}\\
&\hskip1cm\times\frac{1}{[\eps^{-2}(\lambda_3-\lambda_1)]^{1/2}+1}\frac{1}{[\eps^{-2}(\lambda_3-\lambda_2)]^{1/2}+1}\frac{1}{[\eps^{-2}(
\tau+\lambda_1)]^{1/2}+1}\frac{1}{[\eps^{-2}\lambda_2]^{1/2}+1}.
\end{split}
\end{equation}
Since  $0<\lambda_1<\lambda_2<t$, we use  the following inequalities: 
\begin{align}\label{inequalities_lambda}
\lambda_2>\lambda_2-\lambda_1; \quad 
\tau+\lambda_1>\lambda_1;\quad 
\lambda_3-\lambda_1>\lambda_3-\lambda_2; \quad 
\tau+\lambda_3>\tau; \quad 
\tau+\lambda_3>\lambda_3>\lambda_3-\lambda_2.
\end{align}
By using \eqref{inequalities_lambda}, we bound the first line of the right-hand side in \eqref{2-th term in the exponential_positive exponent} by$$\frac{1}{[\eps^{-2}(\tau+\lambda_3)]^{(n-1)/2}+1}\frac{1}{[\eps^{-2}\lambda_3]^{(m-1)/2}+1}\leq \frac{1}{[\eps^{-2}\tau]^{1/2}+1}\frac{1}{[\eps^{-2}(\lambda_3-\lambda_2)]^{(n+m-3)/2}+1},$$
while the second line of \eqref{2-th term in the exponential_positive exponent} is  bounded by 
$$\frac{1}{[\eps^{-2}(\lambda_3-\lambda_2)]+1}\frac{1}{[\eps^{-2}\lambda_1)]^{1/2}}\frac{1}{[\eps^{-2}(\lambda_2-\lambda_1)]^{1/2}}.$$
Putting them together, we obtain 
the next bound for the right-hand side of \eqref{2-th term in the exponential_positive exponent}:
\[
\left(\frac{1}{[\eps^{-2}(\lambda_2-\lambda_1)]^{1/2}}\frac{1}{[\eps^{-2}\lambda_1]^{1/2}}\right) \frac{1}{[\eps^{-2}\tau]^{1/2}+1}\frac{1}{[\eps^{-2}(\lambda_3-\lambda_2)]^{(n+m-1)/2}+1}.
\]
For any $j_1,j_2\in\{1,\dots, K\}$, we can always obtain a similar bound with products of transition probabilities as in \eqref{2-th term in the exponential_positive exponent} and then adjust $\lambda$'s similar to \eqref{inequalities_lambda}, leading to the following bound
\begin{eqnarray}\label{use_pat}
&&\sum_{j_1,j_2=1}^{K}\sum_{z_1,z_2\in I_+}\sum_{z_1',z_2'\in I_+}\underbrace{\int_{0}^{t} d\lambda_1\int_{\lambda_{1}}^{t} d\lambda_{2}\prod_{l=1}^{2}\frac{1}{(\lambda_{l}-\lambda_{l-1})^{1/2}}}_{= a_2(t)} \nonumber\\
&&\hspace{2cm}\times\frac{1}{[\eps^{-2}\tau]^{1/2}+1}\int_{\lambda_{2}}^{t} d\lambda_{3}\frac{\eps^{-1}}{[\eps^{-2}(\lambda_{3}-\lambda_2)]^{(K-1)/2}+1}\nonumber\\
&&\lesssim K^6 a_{2}(t)\frac{\eps^{-1}}{[\eps^{-2}\tau]^{1/2}+1} \epsilon^2,
\end{eqnarray}
where we use Remark \ref{rem:comment_integration} to estimate
$$\int_{\lambda_{2}}^{t} d\lambda_{3}\frac{1}{[\eps^{-2}(\lambda_{3}-\lambda_2)]^{(K-1)/2}+1}\lesssim \eps^2, $$
because $K-1\geq 3$ (as $K\geq4$). Also, the constant  $K^6$ appearing in \eqref{use_pat}, comes from the double sum with respect to $j_1,j_2$ and the two double sums with respect $z_1,z_2$ and $z_1',z_2'$ which gives us $K^2$ and $K^4$  in total, respectively.
For any $j_1,\dots,j_k$ and $k\in \mathbb{N}$, we can always obtain a similar bound with products of transition probabilities as in \eqref{2-th term in the exponential_positive exponent} and then adjust $\lambda$'s similar to \eqref{inequalities_lambda}, leading to the following bound
\begin{eqnarray}\label{k-th term in the exponential_positive exponent_1}
   K^{3k} a_{k}(t)\frac{1}{[\eps^{-2}\tau]^{1/2}+1} \epsilon,
\end{eqnarray}
 for the second term  \eqref{k-th term in the exponential_positive exponent2}.
\bigskip

For $K\geq5$, let $(n,m)$ such that $n,m\geq 1$ and $n+m=K-1$. The integral inequality for $v_{K-1}^{\eps}$ is given by Corollary \ref{cor:final v inequality}. The term in the third line of \eqref{split_time_integrala*} includes  
\begin{equation}\label{v_N*-1_j=1,z=2}
\eps^{-1}\int_{0}^{t} d\lambda  \frac{1}{[\eps^{-2}\lambda]^{1/2}+1}\sum_{j=1}^{K-1}\;\;\sum_{\underline{z}'\subset I_{+}:|\underline{z}'|=2}\mathbf{E}_{\epsilon, \und u }\Big[   |v_{K}^\epsilon\big(\underline{x}^{(j)}(\lambda)\cup\underline{z}',t-\lambda\big)|\Big]
\end{equation}
i.e., the terms with  $|J|=1$ and $|\und z'|=2$. We  obtain a bound proportional to $\eps^{1+\zeta}$ as follows:  We observe that 
for $j=1,\dots K-1$ and by \eqref{Liggett_new},  $\mathbf{E}_{\epsilon, \und u} \Big[|v_{K}^\epsilon\big(\underline{x}^{(j)}(\lambda)\cup \underline{z}',t-\lambda\big)|\Big]$ is bounded by 
\begin{equation}
\begin{split}
&\sum_{\substack{\und v=(v_1,\dots, v_{K-1}),\\|\und v\,\cap\und z'|=0}}\mathbf P_{\eps}(\und u \overset{\lambda}\rightarrow z)v_{K}^\epsilon\big(\und v^{(j)}\cup\underline{z}',t-\lambda\big)|\nonumber\\
&\hspace{2cm}\lesssim\sum_{\pi\in S_{K-1}}\underbrace{\sum_{\substack{\und v^{(j)},\\|\und v^{(j)}\,\cap\und z'|=0}}\prod_{\substack{i=1,\\i\neq j}}^{K-1}P^{(\eps)}_\lambda(u_{\pi(i)},v_i)|v_{K}^\epsilon\big(\und v^{(j)}\cup\underline{z}',t-\lambda\big)|}_{\textrm{estimated by \eqref{a10.8**}}}\nonumber\\
&\hspace{4cm}\lesssim \underbrace{2e ^{(K+1)^2K^4\pi T}(1+T^{1/2})}_{\textrm{constant denoted by $C^{(K)}_{K,T}$}}\left(\eps^{1+\zeta}+\frac{\epsilon}{[\eps^{-2}\lambda]^{1/2}+1} \right).
\end{split}
\end{equation}
The last inequality follows from  \eqref{a10.8**} with $n=K-2$, $m=|\und z'|=2$, $\tau=\lambda$ and $t-\lambda$ in place of $t$. Thus \eqref{v_N*-1_j=1,z=2} is bounded by
$$
\eps^{-1}\int_{0}^{t} d\lambda  \frac{1}{[\eps^{-2}\lambda]^{1/2}+1}\sum_{j=1}^{K-1}\;\;\sum_{\underline{z}'\subset I_{+}:|\underline{z}'|=2} 2e ^{(K+1)^2K^4\pi T}(1+T^{1/2})\left(\eps^{1+\zeta}+\frac{\epsilon}{[\eps^{-2}\lambda]^{1/2}+1} \right).
$$
After integrating, we obtain the bound for the last display as
\begin{equation}\label{est: v N^*-2 }
2e ^{(K+1)^2K^4\pi T}(1+T^{1/2})(K-1){{K}\choose{2}} \left(\eps^{1+\zeta}T^{1/2}+\eps^{2-2\zeta}T^{\zeta}\right)\lesssim 2e ^{(K+1)^2K^4\pi T}(K+1)K^2 (1+T)^{3/2}\, \eps^{1+\zeta}.
\end{equation}
For the terms with  $|J|=1$ and $|\und z'|=1$, i.e.,
    \begin{equation*}
    \eps^{-1}\int_{0}^{t} d\lambda  \sum_{j=1}^{K-1}\sum_{z'\in I_+}\mathbf{E}_\epsilon \Big[\mathbf{1}_{\{x_j(\lambda)\in I_{+}\}} |v_{K-1}^\epsilon\big(\underline{x}^{(j)}(\lambda)\cup z',t-\lambda\big)|\Big],
    \end{equation*}
    we build the infinite series described earlier in the proof. Each $k$-th term of $$\int_{0}^{t}d\lambda \;
\eps^{-1}\mathbf{E}_{\epsilon, \und{u}} \left[ {\bf 1}_{\und{x}(\lambda)=I^{K-1}_+}\right]$$  and the remaining terms, including \eqref{v_N*-1_j=1,z=2} (all estimated as $\epsilon^{1+\zeta}$), are treated exactly as in the case $K$. Reminding that $\und u=\und w\,\cup\,\und y$, overall, the left-hand side of \eqref{eq: improved v2!} is bounded by 
$$  e^{(K-1)^2K^4\pi T}\left(2e ^{(K+1)^2K^4\pi T}(K+1)K^2(1+T)^{3/2}\,\eps^{1+\zeta} +\frac{\epsilon}{[\eps^{-2}\tau]^{1/2}+1} \right)$$ 
and since the above is bounded by $2e ^{2(K+1)^2K^4\pi T}(K+1)K^2 (1+T)^{3/2}\left(\eps^{1+\zeta}+\frac{\epsilon}{[\eps^{-2}\tau]^{1/2}+1} \right)$, we obtain that 
\begin{equation}\label{est:v(N^*-2)}
\sum_{{\substack{\und w=(w_1,\dots,w_n):\\| \und w\,\cap \
\und y|=0}}}\prod_{i=1}^n P_{\tau}^{(\epsilon)}(x_i,w_i)|v_{K-1}^\epsilon(\und{u},t)|\leq \underbrace{2e ^{2(K+1)^2K^4\pi T}(K+1)K^2(1+T)^{3/2}}_{\textrm{constant denoted by $C^{(K-1)}_{K,T}$ }}\left(\eps^{1+\zeta}+\frac{\epsilon}{[\eps^{-2}\tau]^{1/2}+1} \right).
\end{equation}
We proceed in this manner for all pairs $(n,m)$ with $n,m \geq 1$ and  
$
n+m = K-2,\, K-3,\,\dots \,,4$  
with the corresponding constant being  equal to $$C^{(n+m)}_{K,T}=(K+1-(n+m))e ^{(K+1-(n+m))(K+1)^2K^4\pi T}(K+1)K^{(K+1-(n+m))} (1+T)^{(K+2-(n+m))/2},$$
for $4\leq n+m\leq K-1$. For $K\geq6$, for example, when $n+m = K-2$, the method is the same as described above, but we must additionally estimate  
\[
\mathbf{E}_{\epsilon,\und{u}} \Big[\, |v_{K}^\epsilon(\underline{x}^{(j)}(\lambda)\cup \underline{z}',\, t-\lambda)| \,\Big],
\quad
\mathbf{E}_{\epsilon, \und{u}} \Big[\, |v_{K-1}^\epsilon(\underline{x}^{(j)}(\lambda)\cup\underline{z}',\, t-\lambda)| \,\Big],
\]  
where $|\underline{x}^{(j)}(\lambda)| = n+m-1=K-3$,  $|\underline{z}'| = 3$ and $|\underline{z}'| = 2$, respectively. Both are estimated by using 
\eqref{a10.8*} and \eqref{est:v(N^*-2)}.

For $K\geq3$, the proof for $n,m$ with $n+m=3$ follows the same steps described above. Due to $\eps T^{1/2}$ in \eqref{split_time_integrala*},  we obtain
\begin{eqnarray*}
&&\sum_{{\substack{\und w=(w_1,\dots,w_n):\\| \und w\,\cap \
\und y|=0}}}\prod_{i=1}^n P_{\tau}^{(\epsilon)}(x_i,w_i)|v_{n+m}^\epsilon(\und{u},t)|\lesssim C^{(n+m)}_{K,T}\left(\eps+\frac{1}{[\eps^{-2}\tau]^{1/2}}\right).
\end{eqnarray*}
The last case for $n=m=1$ is handled similarly. For $K=2,3$, the low-degree estimates are obtained using the same argument, with $\eps T^{1/2}$ instead of $\eps^{1+\zeta}T^{1/2}$. 
For $K=3$, the degree-three contribution can also be bounded directly by the spatial $v$-estimate of order $\eps$. The remaining degree-two terms, as well as the degrees two and three when $K\geq4$, are treated by the same iteration with the corresponding low-degree source term $\eps T^{1/2}$. The resulting estimate is $\eps+1/[\eps^{-2}\tau]^{1/2}$.
We finish the proof by observing that  
$$C^{(n+m)}_{K,T}\lesssim (K+1)^2e ^{(K+1)^3K^4\pi T}K^{K+1}(1+T)^{(K+1)/2}.$$
\end{proof}

\section{Estimates on space-time  $v$-functions}\label{sec:space_time_corr}

 Let $\underline{x}=(x_1,\dots,x_n)\in \Lambda_{N}^{\neq,n}$ and $\underline{y}=(y_1,\dots,y_m)\in \Lambda_{N}^{\neq,m}$, with $n,m\geq 1$. For any initial product measure $\mu^{\epsilon}$ and times  $0<s<r$, the space-time $v$-correlation function is given by \eqref{def:space time corr fun}.
At  the earliest time $s$, the positions of the $m$ particles are  $\underline{y}=(y_1,\dots,y_m)$ while at time $r$,  the positions of the $n$ particles are $\underline{x}=(x_1,\dots,x_n)$.  The subscript $s,r$ indicates the times involved, ordered from smallest to largest. The superscript includes the parameter $\epsilon$, as well as the number of particles (i.e., the cardinality of the configurations $\und y$ and $\und x$) at times $s$ and $r$. The input of the space-time $v$-correlation function consists of the actual labeled configurations, where $\und y$ is considered fixed and $\und x$ changes over time through the discrete equation for $v^{\epsilon,m,n}_{s,r}(\underline{y},\underline{x})$ given in \eqref{eq:Discrete_equation}. For simplicity---as in \eqref{formula: equation for v}---we express the equation in terms of $v^{\epsilon,m,n}_{s,r}(Y,X)$, where $X, Y \subset \Lambda_N$. This is justified by the fact that $v^{\epsilon,m,n}_{s,r}(\underline{y}, \underline{x})$ is symmetric under any permutation of the components of $\underline{x}$ and $\underline{y}$. When we later write its Duhamel form, we return to the notation involving labeled configurations.
\begin{equation}\label{eq:Discrete_equation}
	\hspace{-0.1cm}\begin{cases}
	&\hspace{-0.4cm}\partial_{r}v^{\epsilon,m,n}_{s,r}(Y,X)= \epsilon^{-2}(L_0v^{\epsilon,m,n}_{s,r})(Y,X)+(C_{\epsilon}v^{\epsilon,m,n}_{s,r})(Y,X), \quad r>s,\\
	&\hspace{-0.4cm}\displaystyle{v^{\epsilon,n,n}_{s,s}(Y,X)= \mathbb{E}^{\epsilon}_{\mu^\epsilon}\left[\left(\prod_{y\in Y}\bar\eta_s(y)\right)^2\right]},\quad Y=X,\\
	&\hspace{-0.4cm}\displaystyle{v^{\epsilon,m,n}_{s,s}(Y,X)= \sum_{L\subset X\cap Y}\psi_s(L,\rho_\eps)v^{\epsilon}_{|[X\triangle Y]\cup L|}([X\triangle Y]\cup L,s)},\quad Y\neq X
	\end{cases}
	\end{equation}
where 
\begin{equation}\label{eq:psi}
\psi_s(L,\rho_\eps)=\prod_{z\in L}(1-2\rho_\epsilon(z,s))\prod_{\tilde z\in X\cap Y\setminus L}\chi(\rho_\epsilon(\tilde z,s))
\end{equation}
(we recall $\chi(u)=u(1-u)$). The generator $L_0$ is given in \eqref{eq:gen_stir} and $C_{\epsilon}$ is defined in \eqref{12aa}. The notation $\sum_{L\subset X\cap Y}$ means the sum over all possible subsets of  $X\cap Y$ and $v^{\epsilon}_{|[X\triangle Y]\cup L|}([X\triangle Y]\cup L,s)$ is the space correlation $v$-function with $|[X\triangle Y]\cup L|$ number of particles of $[X\triangle Y]\cup L$ at the smallest time $s$. The discrete equation \eqref{eq:Discrete_equation} is derived from \eqref{formula: equation for v} adapted to the function $v^{\epsilon,m,n}_{s,r}$ by fixing the time $s$ and $Y$.  An important property of the space-time $v$-correlation function is that 
\begin{align}\label{property:emptyset}
v^{\epsilon,m,0}_{s,r}(Y,\emptyset)=v^{\epsilon}_{m}(Y,s).
\end{align}
Although \eqref{eq:Discrete_equation} is identical to the discrete equation of the space correlation $v$-estimate \eqref{10_discussion}, the initial condition when $r=s$, i.e., 
$$\sum_{ \und w}\mathbf{P}_{\epsilon}(\und x\overset{t-\lambda}\rightarrow \und w)v^{\epsilon,m,n}_{s,s}( \und w, \und y)$$
differs, resulting in a distinct integral solution:
\begin{eqnarray}\label{eq:space-time later time}
v^{\epsilon,m,n}_{s,r}(\und y,\und x)
&=&\sum_{ \und w}\mathbf{P}_{\epsilon}(\und x\overset{r-s}\rightarrow \und w)v^{\epsilon,m,n}_{s,s}( \und y,  \und w)\nonumber\\
&& +\int_{s}^{r}d\lambda\,\sum_{ \und w}\mathbf{P}_{\epsilon}( \und x\overset{r-\lambda}\rightarrow \und w)(C_{\epsilon} v^{\epsilon,m,n}_{s,\lambda})(\und y, \und w).
 \end{eqnarray}
We further analyze the initial condition by splitting into $\und w\neq \und y$ which takes the form:
\begin{equation}\label{eq:intial_con1}
\sum_{\und w:\,| \und w\cap \und y|=0}\mathbf{P}_{\epsilon}(\und x\overset{r-s}\rightarrow \und w)v_{n+m}^{\epsilon}( \und w\cup\und y,s)
\end{equation}
and $\und w=\und y$ (when $n=m$) which equals to $$\mathbf{P}_{\epsilon}(\und x\overset{r-s}\rightarrow \und y)v^{\epsilon,n,n}_{s,s}( \und y, \und y).$$ The integral solution takes the form:
 \begin{equation}\label{eq:Duhamel formula second,n=m}
 \begin{split}
v^{\epsilon,n,n}_{s,r}(\und y,\und x)&=\sum_{\und w:\,| \und w\cap \und y|=0}\mathbf{P}_{\epsilon}(\und x\overset{r-s}\rightarrow \und w)v_{2n}^{\epsilon}( \und w\cup \und y,s)+\mathbf{P}_{\epsilon}(\und x\overset{r-s}\rightarrow \und y)v^{\epsilon,n,n}_{s,s}( \und y, \und y)\\
&\hskip1cm+\sum_{l=1}^{n-1}\sum_{ \und w:\,| \und w\cap \und y|=l}\mathbf{P}_{\epsilon}(\und x\overset{r-s}\rightarrow \und w)\sum_{\und l\subset  \und w\cap \und y}\psi_s(l,\rho_\eps)v^{\epsilon}_{| \und w\triangle \und y\cup \und l|}( \und w\triangle \und y\cup \und l,s)\\
&\hskip1cm+\int_{s}^{r}d\lambda\,\sum_{ \und w}\mathbf{P}_{\epsilon}( \und x\overset{r-\lambda}\rightarrow \und w)(C_{\epsilon} v^{\epsilon,m,n}_{s,\lambda})(\und y, \und w),
\end{split}
\end{equation}
where $v^{\epsilon,n,n}_{s,s}( \und y, \und y)$ is given by the middle equation in \eqref{eq:Discrete_equation}, while when $n\neq m$, we write
\begin{equation}\label{eq:Duhamel formula second,n dif than m}
\begin{split}
v^{\epsilon,m,n}_{s,r}(\und y,\und x)&=\sum_{\und w:\,| \und w\cap \und y|=0}\mathbf{P}_{\epsilon}(\und x\overset{r-s}\rightarrow \und w)v_{n+m}^{\epsilon}( \und w\cup\und y,s)\\
&+\sum_{l=1}^{
\min\{n,m\}}\sum_{ \und w:\,| \und w\cap \und y|=l}\mathbf{P}_{\epsilon}(\und x\overset{r-s}\rightarrow \und w)\sum_{\und l\subset  \und w\cap \und y}\psi_s(l,\rho_\eps)v^{\epsilon}_{| \und w\triangle \und y\cup \und l|}\left( \left(\und w\triangle \und y\cup \und l\right),s\right)\\
&+\int_{s}^{r}d\lambda\,\sum_{ \und w}\mathbf{P}_{\epsilon}( \und x\overset{r-\lambda}\rightarrow \und w)(C_{\epsilon} v^{\epsilon,m,n}_{s,\lambda})(\und y, \und w).
\end{split}
\end{equation}


\subsection{Estimation of the initial condition \eqref{eq:intial_con1}}\label{subsec:IC} The construction of the integral inequality for space-time correlation $v$-function requires the estimation of the initial condition: 
\medskip

\noindent $\bullet$ First line in \eqref{eq:Duhamel formula second,n=m} and \eqref{eq:Duhamel formula second,n dif than m}:
$$\sum_{\und w:\,| \und w\cap \und y|=0}\mathbf{P}_{\epsilon}(\und x\overset{r-s}\rightarrow \und w)v_{n+m}^{\epsilon}( \und w\cup \und y,s)+\mathbf{P}_{\epsilon}(\und x\overset{r-s}\rightarrow \und y)v^{\epsilon,m,n}_{s,s}( \und y, \und y)\mathbf 1_{\{n=m\}}.$$
The first term is estimated only for $r-s\leq\eps^{\beta^*}$, $\beta^*<2$  because this is what we need for the proving Theorem~\ref{thm:vestimate3}. In particular, when $n\leq K$, the first term is bounded by using Lemma \ref{lem: improved v2},  for $K+1\leq n\leq  K+3-1$ is bounded as $\eps^{1+\zeta}$ and $n\geq  K+3$ is bounded by \eqref{a_reestimated*}. For the second term,
\begin{equation}
\mathbf{P}_{\epsilon}(\und x\overset{r-s}\rightarrow \und y)|v^{\epsilon,n,n}_{s,s}( \und y, \und y)|\lesssim
\begin{cases}
\frac{1}{{[\epsilon^{-2}(r-s)]}^{1/2}+1},&n=1\\
\frac{1}{{[\epsilon^{-2}(r-s)]}^{n/2}+1},&n\geq 2
\end{cases}
\end{equation}
where we use  Lemma \ref{lem: bound T1} and the fact that  $|v^{\epsilon,n,n}_{s,s}( \und y, \und y)|\leq 1$. When $n=m=1$ we define:
\begin{equation}\label{C11}
C_{1,1}^\epsilon(r-s):=\frac{\eps}{{(r-s)}^{1/2}}.
\end{equation}
$\bullet$ When $n=m\geq 2$, the second line in \eqref{eq:Duhamel formula second,n=m} can be bounded as 
\begin{equation}\label{Cnn}
 \frac{1}{{[\epsilon^{-2}(r-s)]}^{n/2}+1} + \sum_{l=1}^{n-1} \frac{\epsilon^{1-2\zeta}}{{[\epsilon^{-2}(r-s)]}^{l/2}+1}\lesssim \frac{\epsilon^{2-2\zeta}}{{(r-s)}^{1-\zeta}}+\frac{\epsilon^{2-2\zeta}}{{(r-s)}^{1/2}}=:C_{n,n}^\epsilon(r-s).
\end{equation}
$\bullet$ When $n\neq m$, the second line in \eqref{eq:Duhamel formula second,n dif than m} can be treated as
\begin{equation}\label{Cnm}
\begin{split}
&\sum_{k=1}^{
\min\{n,m\}}\sum_{ \und w:\,| \und w\cap \und y|=k}\mathbf{P}_{\epsilon}(\und x\overset{r-s}\rightarrow \und w)\sum_{\und l\subset  \und w\cap \und y}\psi_{s}(\und l,\rho_\epsilon)v^{\epsilon}_{| \und w\triangle \und y\cup \und l|}( \und w\triangle \und y\cup \und l,s)\\
&\hspace{1cm}\lesssim\sum_{k=1}^{\min\{n,m\}} \frac{\epsilon^{1-2\zeta}}{{[\epsilon^{-2}(r-s)]}^{k/2}+1}\lesssim   \frac{\epsilon^{2-2\zeta}}{{(r-s)}^{1/2}}=:C_{n,m}^\epsilon(r-s).
\end{split}
\end{equation}
We summarize all the initial conditions mentioned above, for $2\leq n+m\leq K$ and $r-s<1$:
\begin{itemize}
\item[$n+m=2,3$:] $C_{1,1}^\eps(r-s)+D_2\lesssim\frac{\eps}{{(r-s)}^{1/2}}$ where $D_{n+m}$ is given in \eqref{eq: improved v2!}.
\item[$4\leq 2n\leq K$:] $C_{n,n}^\eps(r-s)+D_{2n}\lesssim\frac{\epsilon^{2-2\zeta}}{{(r-s)}^{1-\zeta}}+\eps^{1+\zeta}\lesssim \frac{\eps^{1+\zeta}}{{(r-s)}^{1-\zeta}}$.
\item[$4\leq n+m\leq K$:] $C_{n,m}^\eps(r-s)+D_{n+m}=\frac{\epsilon^{2-2\zeta}}{{(r-s)}^{1/2}}+\eps^{1+\zeta}\lesssim \frac{\epsilon^{1+\zeta}}{{(r-s)}^{1/2}}$.
\end{itemize}
For $ n+m=K=1, K+2$ and $r-s<1$
\begin{itemize}
\item[$n=m$:] $C_{n,n}^\eps(r-s)+\eps^{1+\zeta}\lesssim\frac{\epsilon^{2-2\zeta}}{{(r-s)}^{1-\zeta}}+\eps^{1+\zeta}\lesssim \frac{\eps^{1+\zeta}}{{(r-s)}^{1-\zeta}}$.
\item[$n\neq m$:] $C_{n,m}^\eps(r-s)+\eps^{1+\zeta}=\frac{\epsilon^{2-2\zeta}}{{(r-s)}^{1/2}}+\eps^{1+\zeta}\lesssim\frac{\eps^{1+\zeta}}{{(r-s)}^{1/2}}$.
\end{itemize}
Finally, $n+m\geq  K+3$ and $r-s<1$
\begin{itemize}
\item[$n=m$:] $C_{n,n}^\eps(r-s)+\eps^{2-6\zeta}s^\zeta=\frac{\epsilon^{2-2\zeta}}{{(r-s)}^{1-\zeta}}+\eps^{2-6\zeta}s^\zeta\lesssim \frac{\epsilon^{2-6\zeta}}{{(r-s)}^{1-\zeta}}$.
\item[$n\neq m$:] $C_{n,m}^\eps(r-s)+\eps^{2-6\zeta}s^\zeta=\frac{\epsilon^{2-2\zeta}}{{(r-s)}^{1/2}}+\eps^{2-6\zeta}s^\zeta\lesssim \frac{\epsilon^{2-6\zeta}}{{(r-s)}^{1/2}}$.
\end{itemize}
where we use \eqref{a_reestimated*}.

\noindent $\bullet$ Finally, the method for bounding the integral term in \eqref{eq:Duhamel formula second,n=m}--\eqref{eq:Duhamel formula second,n dif than m}, i.e., 
$$\int_{s}^{r}d\lambda\,\sum_{ \und w}\mathbf{P}_{\epsilon}( \und x\overset{r-\lambda}\rightarrow \und w)(C_{\epsilon} v^{\epsilon,m,n}_{s,\lambda})(\und y, \und w)$$
parallels the proof of Lemma \ref{lem: IIfor a}, as both  $(C_{\epsilon} v^{\epsilon,m,n}_{s,\lambda})$ and $(C_{\epsilon} v^{\epsilon}_{n})$ involve the same operator.  The key distinction arises in the treatment of scenarios where all particles die and no particles are left in the system. In such cases, $(C_{\epsilon} v^{\epsilon}_{n})$  includes $v_0(\emptyset, t-\lambda)=1$ (see Definition \ref{definitionvestimate}) while $(C_{\epsilon} v^{\epsilon,m,n}_{s,r-s-\lambda})$ includes $v^{\epsilon,m,n}_{s,r-s-\lambda}(\und y,\emptyset)=v^{\epsilon}_{m}(\und y,s)$, see \eqref{property:emptyset}. This allows to improve the estimation of these terms by using Theorem \ref{thm:vestimate2} for $v^{\epsilon}_{m}(\und y,s)$. The precise estimation of the integral term is given in the proof of Lemma \ref{lem: IIfor space-time a}.

Finally, for $n,m\geq 1$, let us we denote $$a_{\eps,s}(n,m,r)=\sup_{\und x\in\Lambda_N^{n,\neq}}\sup_{\und y\in\Lambda_N^{m,\neq}}|v^{\epsilon,m,n}_{s,r}(\und y,\und x)|,\;\;\;\;a_{\eps,s}(0,m,r)=\sup_{\und y\in\Lambda_N^{m,\neq}}|v_m^{\eps}(\und y,s)|$$  
\begin{equation}\label{def:d_space-time}
d_{\eps,s}(n,m,r)=\sup_{\underline{x}\in\Lambda_N^{n,\neq}}\sup_{\und y\in\Lambda_N^{m,\neq}}|v_{s,r}^{\epsilon,m,n}(\underline{y},\underline{x})-v_{s,r}^{\epsilon,m,n}                        (\underline{y},\underline{x}+\und e_1)|,\;\;\;\;d_{\eps,s}(0,m,r)=0.\end{equation}
\subsection{Analysis of difference of space-time $v$-functions}\label{subsec:dif_of_space_time_v}
In this subsection, we derive the integral inequality for the difference of space-time $v$-functions, i.e., $d_{\eps,s}(n,m,r)$.
We now estimate the difference of the space-time
$v$-functions by following the argument used for
the spatial $v$-functions (see \eqref{DFD}).
Let $\widetilde{\underline{x}}
=(x_1,x_1+1,x_2,\ldots,x_n)$, and denote by
$\widetilde{\underline{x}}^{(i)}$ the configuration
obtained by removing the $i$-th particle. Applying Duhamel's formula to the two
configurations and using the antisymmetry of
their difference under the exchange of the
first two labels, Lemma~\ref{lem:DPTV_ejp 4.3} allows us to
restrict the expectations to the event that
the corresponding stopping time $\tau_{1,2}$
has not occurred. This yields the following initial condition 
\begin{align*}
&\left(\frac{\eps^{2-2\zeta}}{(r-s)^{1/2}}+\frac{\eps^{2-2\zeta}}{(r-s)^{1-\zeta}}\right){\bf 1}_{n=m=1}+\left(\frac{\epsilon^{2-2\zeta}}{(r-s)^{1/2}}+ C^{\eps}_{n,m}(r-s)\right){\bf 1}_{n+m\geq 3}
\end{align*}
where the first two terms are obtained by the following estimations. The term 
$$\mathbf{E}_{\epsilon, x} \left[v_{s,r}^{\epsilon,1,1}(y,x(r-s)\right]-\mathbf{E}_{\epsilon, x+1}\left[v_{s,r}^{\epsilon,1,1}(y,x(r-s))\right]$$ is bounded by 
\begin{equation*}
\begin{split}
&\sum_{w:\,w\neq y}\left(\mathbf{P}_{\epsilon}(x\overset{r-s}\rightarrow w)-\mathbf{P}_{\epsilon}(x+1\overset{r-s}\rightarrow w)\right)v_{2}^{\epsilon}( w, y,s)+\left(\mathbf{P}_{\epsilon}(x\overset{r-s}\rightarrow y)-\mathbf{P}_{\epsilon}(x+1\overset{r-s}\rightarrow y)\right)v^{\epsilon,1,1}_{s,s}( y, y)\\
&\lesssim \frac{1}{[\eps^{-2}(r-s)]^{1/2}+1}\eps^{1-2\zeta}+\frac{1}{[\eps^{-2}(r-s)]+1}\\
&\lesssim \frac{\eps^{2-2\zeta}}{(r-s)^{1/2}}+\frac{\eps^{2-2\zeta}}{(r-s)^{1-\zeta}},
\end{split}
\end{equation*}
The first inequality above is obtained as follows
\begin{equation}
\begin{split}
&\sum_{w:\,w\neq y}\left|P^{(\epsilon)}_{r-s}(x, w)-P^{(\epsilon)}_{r-s}(x+1, w)\right||v_{2}^{\epsilon}( w, y,s)|\\
&\hskip1cm\lesssim \sum_{w:\,w\neq y}\frac{1}{\sqrt{\epsilon^{-2}(r-s)}+1} G_{\epsilon^{-2}(r-s)}(x,w)|v_{2}^{\epsilon}( w, y,s)|\\
&\hskip1cm\lesssim \frac{1}{\sqrt{\epsilon^{-2}(r-s)}+1} \eps^{1-2\zeta}\underbrace{\sum_{w:\,w\neq y}G_{\epsilon^{-2}(r-s)}(x,w)}_{\leq 1}\lesssim \frac{1}{\sqrt{\epsilon^{-2}(r-s)}+1} \eps^{1-2\zeta}
\end{split}
\end{equation}
For the other term 
\begin{equation}
\begin{split}
&|P^{(\epsilon)}_{r-s}(x, y)-P^{(\epsilon)}_{r-s}(x+1, y)||v^{\epsilon,1,1}_{s,s}( y, y)|\\
&\lesssim \frac{1}{\sqrt{\epsilon^{-2}(r-s)}+1} \underbrace{G_{\epsilon^{-2}(r-s)}(x,w) }_{\lesssim\frac{1}{\sqrt{\epsilon^{-2}(r-s)}+1} }\underbrace{|v^{\epsilon,1,1}_{s,s}( y, y)|}_{\leq 1}\\
&\lesssim\frac{1}{\epsilon^{-2}(r-s)+1}
\end{split}
\end{equation} When  $n+m\geq 3$, we estimate the difference
$$\mathbf{E}_{\epsilon, \underline{x}} \left[|v_{s,r}^{\epsilon,m,n}(\underline{y},\underline{x}(r-s))\right]-\mathbf{E}_{\epsilon, \underline{x}+\und e_{1}} \left[v_{s,r}^{\epsilon,m,n}(\underline{y},\underline{x}(r-s))|\right],$$ by enlarging the space, resulting in  $\und{ \tilde x}=\left(x_1,x_1+1,x_2,x_3,\dots,x_n\right)$ (and thus $|\und{ \tilde x}|=n+1$) and applying the same argument  used to obtain the initial condition in \eqref{17.3.3}:
\begin{equation*} 
\begin{split}
&\mathbf{E}_{\epsilon, \und{ \tilde x}} \left[\left(v_{s,r}^{\epsilon,n,m}(\underline{y},\und{ \tilde x}^{(2)}(r-s))-v_{s,r}^{\epsilon,n,m}(\underline{y},\und{ \tilde x}^{(1)}(r-s))\right)\left(\mathbf 1_{\{|\und{ \tilde x}(r-s) \cap\und y|=0\}}+\mathbf 1_{\{|\und{ \tilde x}(r-s) \cap\und y|\neq0\}}\right)\right]\\
&\lesssim \PP_\eps\left(\tau_{1,2}>\frac{r-s}{2}\right)a_{\eps,s}(n,m,s)+C^{\eps}_{n,m}(r-s)\\
&\lesssim \frac{\epsilon^{2-2\zeta}}{(r-s)^{1/2}}+ C^{\eps}_{n,m}(r-s),
\end{split}
\end{equation*}
Since the integral terms appearing in the estimates for both $a_{\epsilon,s}(n,m,r)$ and $d_{\epsilon,s}(n,m,r)$ can be treated in the same way with only difference that $d_{\epsilon,s}$'s integral term has the indicator function of the stopping time, we make use of \eqref{v:emptyset}, \eqref{lem: difference of rhos} and Corollary~\ref{cor: difference of rhos}, and estimate them simultaneously as follows: 

$\bullet$ Integral term for $a_{\epsilon,s}(n,m,r)$:
\begin{equation}\label{integral_term_a_s}
\begin{split}
&+\int_s^r d\lambda\;
\frac{\eps^{-1}}{[\eps^{-2}(r-\lambda)]^{1/2}+1}
a_{\epsilon,s}(n,m,\lambda)\\
&\int_{s}^{r} d\lambda \;\frac{\eps^{-1}}{[\eps^{-2}(r-\lambda)]^{1/2}+1}\sum_{q=2}^K
 a_{\epsilon,s}(n-1+q,m,\lambda)\\
 &+\mathbf 1_{n\geq 2}\int_{s}^{r} d\lambda\;\frac{\eps^{-1}}{[\eps^{-2}(r-\lambda)]+1}
\sum_{p=2}^{\min\{n-1,K\}}
a_{\epsilon,s}(n-p,m,\lambda)\\
&+\mathbf 1_{n\geq 2}\int_{s}^{r} d\lambda\;\frac{\eps^{-1}}{[\eps^{-2}(r-\lambda)]+1}\sum_{q=1}^{K}\;\;\;
\sum_{p=2}^{\min\{n,K\}}\,
a_{\epsilon,s}(n-p+q,m,\lambda)\\
&+\mathbf{1}_{\{n=2\}}\epsilon^{2-4\zeta}(r-s)^{\zeta}+\epsilon^{2-2\zeta}\mathbf{1}_{\{3\leq n\leq K\}} \\
&+{\bf 1}_{n\geq 2}\int_{s}^{r}d\lambda\frac{\eps^{-1-2\zeta} }{[\eps^{-2}(r-\lambda)]^{1/2}+1}d_{\eps,s}(n-1,m,\lambda)\\
&+  \mathbf{1}_{\{n\geq2\}}\int_{s}^{r} d\lambda\,
\frac{\eps^{-4\zeta}}{[\eps^{-2}(r-\lambda)]^{1/2}+1}a_{\epsilon,s}(n-2,m,\lambda)\,.  
\end{split}
\end{equation}
$\bullet$ Integral term for $d_{\epsilon,s}(n,m,r)$:
\begin{equation}\label{integral_term_d_s}
\begin{split}
&\int_s^r d\lambda\;
\frac{\eps^{-1}}{[\eps^{-2}(r-\lambda)]+1}
a_{\epsilon,s}(n,m,\lambda)\\
&\int_{s}^{r} d\lambda \;\frac{\eps^{-1}}{[\eps^{-2}(r-\lambda)]+1}\sum_{q=2}^K
 a_{\epsilon,s}(n-1+q,m,\lambda)\\
 &+\mathbf 1_{n\geq 2}\int_{s}^{r} d\lambda\;\frac{\eps^{-1}}{[\eps^{-2}(r-\lambda)]+1}
\sum_{p=2}^{\min\{n-1,K\}}
a_{\epsilon,s}(n-p,m,\lambda)\\
&+\mathbf 1_{n\geq 2}\int_{s}^{r} d\lambda\;\frac{\eps^{-1}}{[\eps^{-2}(r-\lambda)]+1}\sum_{q=1}^{K}\;\;\;
\sum_{p=2}^{\min\{n,K\}}\,
a_{\epsilon,s}(n-p+q,m,\lambda)\\
&+\mathbf{1}_{\{n=2\}}\epsilon^{2-4\zeta}(r-s)^{\zeta}+\epsilon^{2-2\zeta}\mathbf{1}_{\{3\leq n\leq K\}} \\ 
&+{\bf 1}_{n\geq 2}\int_{s}^{r}d\lambda\frac{\eps^{-1-2\zeta} }{[\eps^{-2}(r-\lambda)]+1}d_{\eps,s}(n-1,m,\lambda)\\
&+  \mathbf{1}_{\{n\geq2\}}\int_{s}^{r} d\lambda\,
\frac{\eps^{-4\zeta}}{[\eps^{-2}(r-\lambda)]+1}a_{\epsilon,s}(n-2,m,\lambda)\,. 
\end{split}
\end{equation}

\subsection{A priori estimate for the space-time $v$-function of sufficiently high degree}
In the next theorem, for $s<\lambda$,  we establish an a priori estimate for the space-time correlation $v$-function, $v^{\epsilon,m,n}_{s,\lambda}$ for $n+m\geq  K+3$ with $n\neq m$. 
We estimate first the space-time correlation $v$-function first $r-s\leq \eps^{\beta^*}$ by writing the integral inequality.
\begin{lemma}\label{lem: IIfor space-time a_high} Let $K\geq2$ and $r-s\leq \eps^{\beta^*}$.
For for $n\geq  K+3,m\geq 1$ with $n\neq m$, the following integral inequalities for $a_{\eps,s}(n,m,r)$ hold: 
\begin{equation}
\label{inequality: IIfor space-time a_high}
\begin{split}
a_{\epsilon,s}(n,m,r)
&\lesssim
\frac{\epsilon^{2-2\zeta}}{{(r-s)}^{1/2}}+\eps^{2-6\zeta}s^\zeta\\
&+\int_s^r d\lambda\;
\frac{\eps^{-1}}{[\eps^{-2}(r-\lambda)]^{1/2}+1}
a_{\epsilon,s}(n,m,\lambda)\\
&+\int_s^r d\lambda\;
\frac{\eps^{-1}}{[\eps^{-2}(r-\lambda)]^{1/2}+1}
\sum_{q=2}^K
a_{\epsilon,s}(n-1+q,m,\lambda)\\
&+\mathbf 1_{n\geq 2}
\int_s^r d\lambda\;
\frac{\eps^{-1}}{[\eps^{-2}(r-\lambda)]^{1-\zeta}}
\sum_{p=2}^{\min\{n-1,K\}}
a_{\epsilon,s}(n-p,m,\lambda)\\
&+\mathbf 1_{n\geq 2}
\int_s^r d\lambda\;
\frac{\eps^{-1}}{[\eps^{-2}(r-\lambda)]^{1-\zeta}}
\sum_{q=1}^{K}
\sum_{p=2}^{\min\{n,K\}}
a_{\epsilon,s}(n-p+q,m,\lambda)\\
&+(r-s)^{\zeta}\epsilon^{2-4\zeta}\mathbf 1_{n\geq 4}\nonumber\\
&+{\bf 1}_{n\geq 2}\int_{s}^{r}d\lambda\frac{\eps^{-1-2\zeta} }{[\eps^{-2}(r-\lambda)]^{1/2}+1}d_{\eps,s}(n-1,m,\lambda)\\
&+  \mathbf{1}_{\{n\geq2\}}\int_{s}^{r} d\lambda\,
\frac{\eps^{-4\zeta}}{[\eps^{-2}(r-\lambda)]^{1/2}+1}a_{\epsilon,s}(n-2,m,\lambda)\,.  
\end{split}
\end{equation}
\end{lemma}
The proof is analogous to Lemma~\ref{lem: IIfor a for space_v_estimate} along with \eqref{property:emptyset} and thus omitted. Now, we define 
\begin{equation}
a_{\epsilon,s}(n,m)=
\sup_{s\leq \lambda\leq s+\eps^{\beta^*}}(\lambda-s)^{1/2}a_{\epsilon,s}(n,m,\lambda),\qquad n\geq K+3
\end{equation}
We now multiply the integral inequality by the corresponding time weights and take the supremum over $s<r\leq s+\epsilon^{\beta^*}$. By $a_{\epsilon,s}(n,m,\lambda)\leq 1$, direct integration of the factors in \eqref{inequality: IIfor space-time a_high} yields
\begin{equation}
\begin{aligned}
a_{\epsilon,s}(n,m)
\lesssim
\max\Big\{
&\epsilon^{2-2\zeta},
\epsilon^{2-6\zeta+\beta^*/2},
\epsilon^{\beta^*},\,\epsilon^{1-2\zeta+\beta^*(1/2+\zeta)},
\epsilon^{2-4\zeta},\,\epsilon^{\beta^*-2\zeta},
\epsilon^{1-4\zeta+\beta^*}
\Big\}.
\end{aligned}
\end{equation}

By choosing $2-2\zeta\leq\beta^*<2$ and
$\zeta>0$ sufficiently small, all terms on the
right-hand side are bounded by
$\epsilon^{2-4\zeta}$. Consequently,
\begin{equation}
a_{\epsilon,s}(n,m,r)
\lesssim
\frac{\epsilon^{2-4\zeta}}{(r-s)^{1/2}},
\qquad r-s\leq\epsilon^{\beta^*}.
\end{equation}

\begin{theorem}\label{thm:conditional-estimate}
Let $K\geq 2$ be fixed, and let $\beta^*<2$ and $c^*$ be as
in Theorem~\ref{vestimate1}. Let $\zeta>0$, and
$\gamma<\frac{(2-\beta^*)c^*}{3}$. Then, for any positive integers $m,n$, $\lambda-s> \eps^{\beta^*}$ and $0<\epsilon<1$,
\[
\sup_{\underline w\in\Lambda_N^{n,\neq}}\sup_{\underline y\in\Lambda_N^{m,\neq}}
\left|
v^{\epsilon,m,n}_{s,\lambda}
(\underline y,\underline w)
\right|
\lesssim
\displaystyle \epsilon^{\gamma n}
\]
\end{theorem}

\begin{proof}
Let  $\mathcal F_s= \mathcal{F}(\{\eta_{s'}:s'\leq s\})$. A first calculation shows that 
\begin{equation}\label{important property**}
\begin{split}
|v^{\epsilon,m,n}_{s,\lambda}(\und y, \und w)|&=\left|\mathbb{E}^{\epsilon}_{\mu^\epsilon}\left[\prod_{j=1}^{m}\bar\eta_s(y_j)\prod_{i=1}^{n}\bar\eta_{\lambda}(w_i)\right]\right|\\
&\leq\mathbb{E}^{\epsilon}_{\mu^\epsilon}\left[\left|\prod_{j=1}^{m}\bar\eta_s(y_j)\right|\;\left|\mathbb{E}^{\epsilon}_{\mu^\epsilon}\left[\prod_{i=1}^{n}\bar\eta_{\lambda}(w_i)|\mathcal F_s\right]\right|\right].
\end{split}
\end{equation}
Let $W={w_1,\dots,w_n}$. We first observe that
\begin{equation}\label{explanation_1}
\begin{split}
\prod_{i=1}^{n}\bar\eta_{\lambda}(w_i)=\sum_{Y\subset W}\prod_{y\in Y}\left(\eta_{\lambda}(y)-\rho_\eps(y,\lambda|\eta_s,s)\right)\prod_{x\in W\setminus Y}\left(\rho_\eps(x,\lambda|\eta_s,s)-\rho_\eps(x,\lambda)\right)
\end{split}
\end{equation}
so by considering the tower property, we have 
\begin{equation}\label{explanation_2}
\begin{split}
\mathbb E^\epsilon _{\eta}\left[\prod_{i=1}^{n}\bar\eta_{\lambda}(w_i)|\mathcal F_s\right]&=\sum_{\emptyset\neq Y\subset W}\mathbb E^\epsilon _{\eta}\big[\prod_{y\in Y}\left(\eta_{\lambda}(y)-\rho_\eps(y,\lambda|\eta_s,s)\right)\\
&\hskip1cm\times \prod_{x\in W\setminus Y}\left(\rho_\eps(x,\lambda|\eta_s,s)-\rho_\eps(x,\lambda)\right)|\mathcal F_s\big]\\
&\hskip2cm+\mathbb E^\epsilon _{\eta}\big[\prod_{x\in W}\left(\rho_\eps(x,\lambda|\eta_s,s)-\rho_\eps(x,\lambda)\right)|\mathcal F_s\big]\\
&\lesssim 
\sum_{\emptyset\neq Y\subset W}
\epsilon^{(2-\beta^*)c^*|Y|}\mathbb E^\epsilon _{\eta}\big[\prod_{x\in W\setminus Y}\left|\rho_\eps(x,\lambda|\eta_s,s)-\rho_\eps(x,\lambda)\right|\big]\\
&\hskip2cm+\mathbb E^\epsilon _{\eta}\big[\prod_{x\in W}\left|\rho_\eps(x,\lambda|\eta_s,s)-\rho_\eps(x,\lambda)\right|\big]
\end{split}
\end{equation}
where $\rho_\eps(x,\lambda|\eta_s,s)$ denotes the solution of \eqref{eq:linearized} with initial datum $\eta_s$ at time $s$. In the last line, we applied  Theorem~\ref{vestimate1},   because the bound is uniform in the initial configuration (here $\eta_s$). It remains to estimate the difference $\mathbb E^\epsilon _{\eta}\big[\left|\rho_\eps(x,\lambda|\eta_s,s)-\rho_\eps(x,\lambda)\right|\big]$. We set
\[
h_s(x,\lambda)=\left|\rho_\eps(x,\lambda|\eta_s,s)-\rho_\eps(x,\lambda)\right|.
\]
Then, we obtain
\begin{equation}
\begin{split}
h_s(x,\lambda)&\leq |\sum_{y}P^{(\eps)}_{\lambda-s}(x,y')\left(\eta_s(y')-\rho_\eps(y',s)\right)|+c\eps^{-1}\int_s^\lambda ds'\sum_{y\in I_{\pm}}P^{(\eps)}_{\lambda-s'}(x,y)h_s(y,s')
\end{split}
\end{equation}
If $\left|\sum_{y'}P^{(\eps)}_{\lambda-s}(x,y')\left(\eta_s(y')-\rho_\eps(y',s)\right)\right|\leq \eps^\gamma$, for $\gamma>0$, then similar to  Lemma 10.4 in \cite{de2012truncated}, we obtain $h_s(x,\lambda)\leq \eps^\gamma e^{\pi T}$. On the other hand, by  the Chebyshev inequality with power $2$
\begin{equation}\label{prob_norm}
    \begin{split}
        \mathbb P^\eps_\eta\left(|\sum_{y'}P^{(\eps)}_{\lambda-s}(y,y')\left(\eta_s(y')-\rho_\eps(y',s)\right)|\geq \eps^\gamma\right)&\leq  \eps^{-2\gamma }
 \sum_{x_1,x_2}\prod_{i=1}^{2} P^{(\eps)}_{\lambda-s}(x,x_i)\\
 &\hskip1cm\times\mathbb E^\eps_\eta\big(
 \prod_{i=1}^{2}[\eta(x_i,s)-\rho(x_i,s)]\big).
  \end{split}
\end{equation}
 $|\mathbb E_{\eta}^\eps\big(
 \prod_{y\in Y}[\eta(y,s)-\rho(y,s)]\big)|= |v^\eps(Y,s)|$ and by Theorem~\ref{vestimate1}, it is bounded by $\epsilon^{(2-\beta^*)c^*|Y|}
$ where $Y\subset\{x_1,x_{2}\}$ with all elements of $Y$ different.  Furthermore, $P^{(\eps)}_{\lambda-s}(x,x_i) \lesssim \frac{1}{[\eps^{-2}(\lambda-s)+1]^{1/2}}$. Therefore, \eqref{prob_norm} is bounded by 
$\epsilon^{-2\gamma}
\left[
\epsilon^{2(2-\beta^*)c^*}
+
\frac{\epsilon^{(2-\beta^*)c^*}}
{[\epsilon^{-2}(\lambda-s)+1]^{1/2}}
+
\frac{1}
{[\epsilon^{-2}(\lambda-s)+1]^{1/2}}
\right]$ which is further bounded by 
\begin{equation}\label{pepsilon}
p_\eps(\lambda-s):=\epsilon^{(2-\beta^*)c^*-2\gamma}
+
\frac{\epsilon^{-2\gamma}}
{[\epsilon^{-2}(\lambda-s)+1]^{1/2}}
\end{equation}
and thus
\[
\begin{aligned}
&
\left|
\mathbb E^\epsilon_\eta
\left[
\prod_{i=1}^{n}\bar\eta_\lambda(w_i)
\mid\mathcal F_s
\right]
\right|\lesssim
\sum_{\emptyset\neq Y\subset W}
\epsilon^{(2-\beta^*)c^*|Y|}
\left(
\epsilon^\gamma
+
p_\epsilon(\lambda-s)
\right)^{n-|Y|}+
\left(
\epsilon^\gamma
+
p_\epsilon(\lambda-s)
\right)^n .
\end{aligned}
\]
By the binomial theorem and $\gamma<\frac{(2-\beta^*)c^*}{3}$, the right-hand side equals, we obtain
\[
\begin{aligned}
|v_{s,\lambda}^{\epsilon,m,n}(\underline y,\underline w)|
&\lesssim
\left(
\epsilon^\gamma+\epsilon^{(2-\beta^*)c^*-2\gamma}
+\frac{\epsilon^{-2\gamma}}{[\epsilon^{-2}(\lambda-s)]^{1/2}+1}
\right)^n\\
&\lesssim
\epsilon^{\gamma n }
+
\frac{\epsilon^{-2\gamma n}}{([\epsilon^{-2}(\lambda-s)]^{1/2}+1)^n}.
\end{aligned}
\]

Finally, since $([\epsilon^{-2}(\lambda-s)]^{1/2}+1)^2\ge [\epsilon^{-2}(\lambda-s)]+1$,
\[
|v_{s,\lambda}^{\epsilon,m,n}(\underline y,\underline w)|
\lesssim
\eps^{\gamma n}
+
\frac{\epsilon^{-2\gamma n}}{([\epsilon^{-2}(\lambda-s)]+1)^{n/2}}.
\]

Taking the supremum over $\underline y$ and
$\underline w$ completes the argument. 
\end{proof}
We choose $N_0> K+3$ such that  $\sup_{\underline w\in\Lambda_N^{n,\neq}}\sup_{\underline y\in\Lambda_N^{m,\neq}}
\left|
v^{\epsilon,m,N_0}_{s,\lambda}
(\underline y,\underline w)
\right|
\lesssim
\displaystyle \epsilon^{\gamma N_0}\lesssim \eps^{1+\zeta}$ for some $\zeta>0$. Overall, for $n\geq N_0$ 
\begin{equation}\label{N*_final}
\sup_{\underline w\in\Lambda_N^{n,\neq}}\sup_{\underline y\in\Lambda_N^{m,\neq}}
\left|
v^{\epsilon,m,n}_{s,\lambda}
(\underline y,\underline w)
\right|\lesssim \frac{\epsilon^{2-4\zeta}}{(\lambda-s)^{1/2}}{\bf 1}_{\{\lambda-s\leq \eps^{\beta^*}\}}+ \eps^{1+\zeta}{\bf 1}_{\{\lambda-s> \eps^{\beta^*}\}}
\end{equation}

\subsection{Integral inequalities for space-time  $v$-functions}\label{sec: Integral inequalities for space-time correlation}
In this section, we present the integral inequalities for the space-time  $v$-function for $r-s>0$ and $n,m\geq 1$. Based on the analysis of the initial condition in Section~\ref{subsec:IC}, for $t<1$, we denote 
\begin{equation}\label{def:IC}
{\rm{IC}}_{n,m}(t):=
\begin{cases}
\frac{\eps}{t^{1/2}}, &n+m=2,3\\[6mm]
\frac{\eps^{1+\zeta}}{t^{1/2}}, &4\leq n+m\leq N_0, \,n\neq m\\[6mm]
\frac{\eps^{1+\zeta}}{t^{1-\zeta}}, &4\leq n+m\leq N_0, \, n= m.
\end{cases}
\end{equation} 
\begin{lemma}\label{lem: IIfor space-time a} Let $K\geq2$, $0\leq s<r\leq T$.
For $1\leq n,m\leq N_0-1$, the following integral inequality for $a_{\eps,s}(n,m,r)$ holds: 
\begin{equation}
\label{inequality: IIfor space-time a}
\begin{split}
a_{\epsilon,s}(n,m,r)
&\lesssim
{\rm IC}_{n,m}(r-s)\\
&+\int_s^r d\lambda\;
\frac{\eps^{-1}}{[\eps^{-2}(r-\lambda)]^{1/2}+1}
a_{\epsilon,s}(n,m,\lambda)\\
&+\int_s^r d\lambda\;
\frac{\eps^{-1}}{[\eps^{-2}(r-\lambda)]^{1/2}+1}
\sum_{q=2}^K
a_{\epsilon,s}(n-1+q,m,\lambda){\bf 1}_{n-1+q\leq N_0-1}\\
&+\sum_{q=2}^K\left( \eps^{2-4\zeta}{\bf 1}_{\{r-s\leq \eps^{\beta^*}\}}+[\eps^{2-4\zeta}+(r-s)^{1/2}\eps^{1+\zeta}]{\bf 1}_{\{r-s> \eps^{\beta^*}\}}\right){\bf 1}_{n-1+q\geq N_0}\\
&+\mathbf 1_{n\geq 2}
\int_s^r d\lambda\;
\frac{\eps^{-1}}{[\eps^{-2}(r-\lambda)]+1}
\sum_{p=2}^{\min\{n-1,K\}}
a_{\epsilon,s}(n-p,m,\lambda)\\
&+\mathbf 1_{n\geq 2}
\int_s^r d\lambda\;
\frac{\eps^{-1}}{[\eps^{-2}(r-\lambda)]+1}
\sum_{q=1}^{K}
\sum_{p=2}^{\min\{n,K\}}
a_{\epsilon,s}(n-p+q,m,\lambda){\bf 1}_{n-p+q\leq N_0-1}\\
&+\sum_{q=1}^{K}
\sum_{p=2}^{\min\{n,K\}}\big( \eps^{3-6\zeta}(r-s)^{-1/2+\zeta}{\bf 1}_{\{r-s\leq \eps^{\beta^*}\}}\\
&\hskip1cm+[\eps^{3-6\zeta-\beta^*(1/2-\zeta)}+\eps^{1+\zeta}(r-s)^{\zeta}]{\bf 1}_{\{r-s>\eps^{\beta^*}\}}\big){\bf 1}_{n-p+q\geq N_0}\\
&+\mathbf{1}_{\{n=2\}}
\epsilon^{2-4\zeta}(r-s)^{\zeta}
+\mathbf{1}_{\{3\leq n\leq K\}}\epsilon^{2-2\zeta}
\\
&+{\bf 1}_{n\geq 2}\int_{s}^{r}d\lambda\frac{\eps^{-1-2\zeta} }{[\eps^{-2}(r-\lambda)]^{1/2}+1}d_{\eps,s}(n-1,m,\lambda)\\
&+  \mathbf{1}_{\{n\geq2\}}\int_{s}^{r} d\lambda\,
\frac{\eps^{-4\zeta}}{[\eps^{-2}(r-\lambda)]^{1/2}+1}a_{\epsilon,s}(n-2,m,\lambda)\,.  
\end{split}
\end{equation}
where ${\rm IC}_{n,m}(r-s)$ in \eqref{def:IC}.
\end{lemma}
\noindent The proof uses \eqref{integral_term_a_s}. The sums in which the degree of $a_{\eps,s}$'s may go beyond $N_0$ are split into $n$ going up to $N_0-1$ and beyond that where we use \eqref{N*_final}. Similarly, we derive the integral inequality for $d_{\eps, s}$ using the estimation of its initial condition discussed in Section~\ref{subsec:dif_of_space_time_v} and the integral term \eqref{integral_term_d_s}.
\begin{lemma}\label{lem:integral_inequality_dspace_time}
Let $K\geq2$, $0\leq s<r\leq T$.
For $1\leq n,m\leq N_0-1$, the following integral inequality for $d_{\eps,s}(n,m,r)$ holds: 
\begin{equation}
\begin{split}
    d_{\eps,s}(n,m,r)&\lesssim \left(\frac{\eps^{2-2\zeta}}{(r-s)^{1/2}}+\frac{\eps^{2-2\zeta}}{(r-s)^{1-\zeta}}\right){\bf 1}_{n=m=1}+\left(\frac{\epsilon^{2-2\zeta}}{(r-s)^{1/2}}+ C^{\eps}_{n,m}(r-s)\right){\bf 1}_{n+m\geq 3}\\
&+\int_s^r d\lambda\;
\frac{\eps^{-1}}{\eps^{-2}(r-\lambda)+1}
a_{\epsilon,s}(n,m,\lambda)
\\
&+\int_s^r d\lambda\;
\frac{\eps^{-1}}{\eps^{-2}(r-\lambda)+1}
\sum_{q=2}^K
a_{\epsilon,s}(n-1+q,m,\lambda)
\mathbf1_{\{n-1+q\leq N_0-1\}}
\\
&+\sum_{q=2}^K
\Big(
\eps^{3-6\zeta}(r-s)^{-1/2+\zeta}
\mathbf1_{\{r-s\leq\eps^{\beta^*}\}}
\\
&\hspace{1.5cm}
+\big[
\eps^{3-6\zeta-\beta^*(1/2-\zeta)}
+\eps^{2-\zeta}(r-s)^\zeta
\big]
\mathbf1_{\{r-s>\eps^{\beta^*}\}}
\Big)
\mathbf1_{\{n-1+q\geq N_0\}}
\\[2mm]
&+\mathbf1_{\{n\geq2\}}
\int_s^r d\lambda\;
\frac{\eps^{-1}}{\eps^{-2}(r-\lambda)+1}
\sum_{p=2}^{\min\{n-1,K\}}
a_{\epsilon,s}(n-p,m,\lambda)
\\[2mm]
&+\mathbf1_{\{n\geq2\}}
\int_s^r d\lambda\;
\frac{\eps^{-1}}{\eps^{-2}(r-\lambda)+1}
\sum_{q=1}^K
\sum_{p=2}^{\min\{n,K\}}
a_{\epsilon,s}(n-p+q,m,\lambda)
\mathbf1_{\{n-p+q\leq N_0-1\}}
\\
&+\sum_{q=1}^K
\sum_{p=2}^{\min\{n,K\}}
\Big(
\eps^{3-6\zeta}(r-s)^{-1/2+\zeta}
\mathbf1_{\{r-s\leq\eps^{\beta^*}\}}
\\
&\hspace{1.5cm}
+\big[
\eps^{3-6\zeta-\beta^*(1/2-\zeta)}
+\eps^{2-\zeta}(r-s)^\zeta
\big]
\mathbf1_{\{r-s>\eps^{\beta^*}\}}
\Big)
\mathbf1_{\{n-p+q\geq N_0\}}
\\[2mm]
&+\mathbf1_{\{n=2\}}
\eps^{2-4\zeta}(r-s)^\zeta
+\mathbf1_{\{3\leq n\leq K\}}
\eps^{2-2\zeta}
\\
&+\mathbf1_{\{n\geq2\}}
\int_s^r d\lambda\;
\frac{\eps^{-1-2\zeta}}
{\eps^{-2}(r-\lambda)+1}
d_{\epsilon,s}(n-1,m,\lambda)
\\
&+\mathbf1_{\{n\geq2\}}
\int_s^r d\lambda\;
\frac{\eps^{-4\zeta}}
{\eps^{-2}(r-\lambda)+1}
a_{\epsilon,s}(n-2,m,\lambda).
\end{split}
\end{equation}
\end{lemma}

\subsection{Proof of Theorem~\ref{thm:vestimate3}}
\label{sec:proof of space time}

The proof follows the same strategy as that of
Theorem~\ref{thm:vestimate2}. The main difference
is that we must additionally control the initial
conditions ${\rm IC}_{n,m}(r-s)$ given in
\eqref{def:IC}. We first establish the estimates for
$0<r-s\leq\tilde\Delta$, where
$0<\tilde\Delta<1$ will be chosen sufficiently
small. More precisely, for
$1\leq n,m\leq N_0-1$, we aim to prove
\[
\sup_{s<\lambda\leq s+\tilde\Delta}
(\lambda-s)^{1/2}
a_{\epsilon,s}(n,m,\lambda)
\lesssim\epsilon,
\qquad n+m=2,3,
\]
and, for $n+m\geq4$,
\[
\sup_{s<\lambda\leq s+\tilde\Delta}
(\lambda-s)^{1-\zeta}
a_{\epsilon,s}(n,m,\lambda)
\lesssim\epsilon^{1+\zeta},
\qquad n=m,
\]
while
\[
\sup_{s<\lambda\leq s+\tilde\Delta}
(\lambda-s)^{1/2}
a_{\epsilon,s}(n,m,\lambda)
\lesssim\epsilon^{1+\zeta},
\qquad n\neq m.
\]

To this end, we introduce the weighted quantities
\begin{align*}
a_{\epsilon,s}^*(n,m)
&=
\sup_{s<\lambda\leq s+\tilde\Delta}
\tilde\Gamma_a(n,m,\lambda-s)
a_{\epsilon,s}(n,m,\lambda),\\
d_{\epsilon,s}^*(n,m)
&=
\sup_{s<\lambda\leq s+\tilde\Delta}
\tilde\Gamma_d(n,m,\lambda-s)
d_{\epsilon,s}(n,m,\lambda),
\end{align*}
where
\[
\tilde\Gamma_a(n,m,t)=
\begin{cases}
\displaystyle
\epsilon^{-1}t^{1/2},
& n+m=2,3,\\[2mm]
\displaystyle
\epsilon^{-1-\zeta}t^{1-\zeta},
& n+m\geq4,\quad n=m,\\[2mm]
\displaystyle
\epsilon^{-1-\zeta}t^{1/2},
& n+m\geq4,\quad n\neq m,
\end{cases}
\]
and
\[
\tilde\Gamma_d(n,m,t)
=
\epsilon^{-1-3\zeta}t^{1-\zeta}.
\]

We also define the corresponding inverse weights
\[
\Gamma_a(n,m,t)=
\begin{cases}
\displaystyle
\epsilon t^{-1/2},
& n+m=2,3,\\[2mm]
\displaystyle
\epsilon^{1+\zeta}t^{-1+\zeta},
& n+m\geq4,\quad n=m,\\[2mm]
\displaystyle
\epsilon^{1+\zeta}t^{-1/2},
& n+m\geq4,\quad n\neq m,
\end{cases}
\]
and
\[
\Gamma_d(n,m,t)
=
\epsilon^{1+3\zeta}t^{-1+\zeta}.
\]
All these weights are considered for
$1\leq n,m\leq N_0-1$ and $t>0$. Starting from
Lemmas~\ref{lem: IIfor space-time a}
and~\ref{lem:integral_inequality_dspace_time},
we express the correlation functions appearing
on the right-hand sides in terms of their
corresponding weighted suprema. In particular,
\[
a_{\epsilon,s}(j,m,\lambda)
\leq
\Gamma_a(j,m,\lambda-s)
a_{\epsilon,s}^*(j,m),
\]
and
\[
d_{\epsilon,s}(j,m,\lambda)
\leq
\Gamma_d(j,m,\lambda-s)
d_{\epsilon,s}^*(j,m).
\]
We then multiply the integral inequalities
by $\tilde\Gamma_a(n,m,r-s)$ and
$\tilde\Gamma_d(n,m,r-s)$, respectively,
and take the supremum over
$s<r\leq s+\tilde\Delta$. The estimates are based on the kernel bounds
\[
\frac{\epsilon^{-1}}
{[\epsilon^{-2}t]^{1/2}+1}
\lesssim t^{-1/2},
\qquad
\frac{\epsilon^{-1}}
{\epsilon^{-2}t+1}
\lesssim
\epsilon^{1-2\zeta}t^{-1+\zeta},
\]
together with the convolution identity
\[
\int_0^t
(t-u)^{-\alpha}u^{-\beta}\,du
=
B(1-\alpha,1-\beta)
t^{1-\alpha-\beta},
\qquad \alpha,\beta<1.
\]
We estimate the integral terms according to
whether the indices of the correlation functions
appearing on the right-hand sides are equal
or different. First, consider the terms involving the
difference functions. For example, when
$n\neq m$ and $n+m\geq4$, we have
\[
\begin{aligned}
&\tilde\Gamma_a(n,m,r-s)
\int_s^r
\frac{\epsilon^{-1-2\zeta}}
{[\epsilon^{-2}(r-\lambda)]^{1/2}+1}
d_{\epsilon,s}(n-1,m,\lambda)\,d\lambda\\
&\qquad\lesssim
(r-s)^\zeta
d_{\epsilon,s}^*(n-1,m).
\end{aligned}
\]
Similarly, using the additional decay
obtained from the stopping-time estimate,
we obtain
\[
\begin{aligned}
&\tilde\Gamma_d(n,m,r-s)
\int_s^r
\frac{\epsilon^{-1-2\zeta}}
{\epsilon^{-2}(r-\lambda)+1}
d_{\epsilon,s}(n-1,m,\lambda)\,d\lambda\\
&\qquad\lesssim
\epsilon^{1-4\zeta}(r-s)^\zeta
d_{\epsilon,s}^*(n-1,m).
\end{aligned}
\]
Next, we consider the diagonal contributions.
These occur whenever the later-time degree
of a correlation function appearing on the
right-hand side equals $m$, namely when
$n=m$, $n-1+q=m$, $n-p=m$,
$n-p+q=m$, or $n-2=m$, whenever the
corresponding terms are present.

For these contributions, we use the diagonal
weight $\Gamma_a(m,m,\lambda-s)$, while
for the remaining terms we use the
corresponding off-diagonal weights.

In particular, when $m\geq2$,
$n\neq m$ and $n+m\geq4$, a diagonal
contribution under the kernel
$\epsilon^{-1}/(\epsilon^{-2}(r-\lambda)+1)$
satisfies
\[
\begin{aligned}
&\tilde\Gamma_a(n,m,r-s)
\int_s^r
\frac{\epsilon^{-1}}
{\epsilon^{-2}(r-\lambda)+1}
a_{\epsilon,s}(m,m,\lambda)\,d\lambda\\
&\qquad\leq
\epsilon^{-1}(r-s)^{1/2}
a_{\epsilon,s}^*(m,m)
\int_s^r
\frac{d\lambda}
{[\epsilon^{-2}(r-\lambda)+1]
(\lambda-s)^{1-\zeta}}\\
&\qquad\lesssim
(r-s)^{1/2}a_{\epsilon,s}^*(m,m)
\int_s^r
\frac{d\lambda}
{(r-\lambda)^{1/2}
(\lambda-s)^{1-\zeta}}\\
&\qquad\lesssim
(r-s)^\zeta a_{\epsilon,s}^*(m,m).
\end{aligned}
\]
Here we used
\[
\frac{1}{\epsilon^{-2}(r-\lambda)+1}
\leq
\frac{\epsilon}{(r-\lambda)^{1/2}}.
\]
Consequently, taking the supremum over
$s<r\leq s+\tilde\Delta$ gives a
coefficient of order $\tilde\Delta^\zeta$,
which can be made arbitrarily small.
The remaining diagonal and off-diagonal
contributions are estimated analogously
using their corresponding weights.

For terms involving correlation functions
of degree at least $N_0$, we apply
\eqref{N*_final}. These contributions are
estimated explicitly and do not enter the
finite system of weighted correlation functions.
For sufficiently small $\zeta>0$, their
weighted bounds are controlled by positive
powers of $\epsilon$ and $\tilde\Delta$. Furthermore, when $n=2$, terms involving
$a_{\epsilon,s}(0,m,\lambda)$ are estimated
separately, since
\[
a_{\epsilon,s}(0,m,\lambda)
=a_\epsilon(m,s)
\lesssim\epsilon^{1-2\zeta}.
\]
Combining these estimates with the
initial-condition bounds, and recalling
that $K$ and $N_0$ are fixed, we obtain
\[
\begin{aligned}
\max_{1\leq n,m\leq N_0-1}
a_{\epsilon,s}^*(n,m)
\lesssim{}&
1+
(\tilde\Delta^\zeta+\epsilon^\zeta)
\max_{1\leq n,m\leq N_0-1}
a_{\epsilon,s}^*(n,m)\\
&+\tilde\Delta^\zeta
\max_{1\leq n,m\leq N_0-1}
d_{\epsilon,s}^*(n,m),
\end{aligned}
\]
and
\[
\begin{aligned}
\max_{1\leq n,m\leq N_0-1}
d_{\epsilon,s}^*(n,m)
\lesssim{}&
\epsilon^{1-7\zeta}+\epsilon^{1-5\zeta}\tilde\Delta^\zeta
\max_{1\leq n,m\leq N_0-1}
a_{\epsilon,s}^*(n,m)\\
&+\epsilon^{1-4\zeta}\tilde\Delta^\zeta
\max_{1\leq n,m\leq N_0-1}
d_{\epsilon,s}^*(n,m).
\end{aligned}
\]

We now close the finite system. Since $N_0$
is fixed, the implicit constants can be
chosen uniformly over
$1\leq n,m\leq N_0-1$.
We first choose $\tilde\Delta>0$
sufficiently small and then $\epsilon>0$
sufficiently small, so that the terms
involving the weighted suprema on the
right-hand sides can be absorbed.
Consequently,
\[
\max_{1\leq n,m\leq N_0-1}
\left\{
a_{\epsilon,s}^*(n,m),
d_{\epsilon,s}^*(n,m)
\right\}
\lesssim1.
\]

Substituting this bound into the inequality
for $d_{\epsilon,s}^*$ gives the sharper estimate
\[
\max_{1\leq n,m\leq N_0-1}
d_{\epsilon,s}^*(n,m)
\lesssim\epsilon^{1-7\zeta}.
\]

Therefore, recalling the definitions of the
weights, for $0<r-s\leq\tilde\Delta$,
\[
a_{\epsilon,s}(n,m,r)
\lesssim
\begin{cases}
\displaystyle
\frac{\epsilon}{(r-s)^{1/2}},
& n+m=2,3,\\[3mm]
\displaystyle
\frac{\epsilon^{1+\zeta}}{(r-s)^{1-\zeta}},
& n+m\geq4,\quad n=m,\\[3mm]
\displaystyle
\frac{\epsilon^{1+\zeta}}{(r-s)^{1/2}},
& n+m\geq4,\quad n\neq m.
\end{cases}
\]
Moreover,
\[
d_{\epsilon,s}(n,m,r)
\lesssim
\frac{\epsilon^{2-4\zeta}}
{(r-s)^{1-\zeta}}.
\]
This completes the short-time estimates.

\medskip

We now extend the estimates to
$r-s>\tilde\Delta$.
The argument follows the same strategy
as in the proof of
Theorem~\ref{thm:vestimate2}.
The key observation is that the short-time
bounds provide suitable initial conditions
at time $s+\tilde\Delta$. Indeed, restarting Duhamel's formula at
$s+\tilde\Delta$, we obtain
\begin{align}
v_{s,r}^{\epsilon,m,n}
(\underline y,\underline x)
={}&
\sum_{\underline w}
\mathbf P_\epsilon
\left(
\underline x
\overset{r-s-\tilde\Delta}{\longrightarrow}
\underline w
\right)
v_{s,s+\tilde\Delta}^{\epsilon,m,n}
(\underline y,\underline w)
\nonumber\\
&+
\int_{s+\tilde\Delta}^{r}d\lambda
\sum_{\underline w}
\mathbf P_\epsilon
\left(
\underline x
\overset{r-\lambda}{\longrightarrow}
\underline w
\right)
(C_\epsilon v_{s,\lambda}^{\epsilon,m,\cdot})
(\underline y,\underline w).
\end{align}

By the short-time estimates, and since
$\tilde\Delta>0$ is fixed,
\[
a_{\epsilon,s}(n,m,s+\tilde\Delta)
\lesssim
\begin{cases}
\epsilon, & n+m=2,3,\\
\epsilon^{1+\zeta}, & n+m\geq4,
\end{cases}
\]
for $1\leq n,m\leq N_0-1$.
Moreover,
\[
d_{\epsilon,s}(n,m,s+\tilde\Delta)
\lesssim\epsilon^{2-4\zeta}.
\]

We apply the same integral-inequality
arguments on successive time intervals,
using the bounds obtained at the end
of each interval as initial conditions
for the next one.

More precisely, we divide
$[s+\tilde\Delta,T]$ into finitely many
subintervals of sufficiently small length,
independent of $\epsilon$, and repeat
the finite-system estimates as in the
proof of Theorem~\ref{thm:vestimate2}.
The higher-degree contributions are
controlled by \eqref{N*_final}.

Since the number of subintervals is finite
and independent of $\epsilon$, the bounds
remain uniform in $\epsilon$.
Consequently, for $r-s>\tilde\Delta$,
\[
a_{\epsilon,s}(n,m,r)
\lesssim
\begin{cases}
\epsilon, & n+m=2,3,\\
\epsilon^{1+\zeta}, & n+m\geq4,
\end{cases}
\]
for $1\leq n,m\leq N_0-1$. This completes the proof.

\begin{corollary}
\label{cor:vestimate3_difference}
Let $K\geq2$ be a fixed integer and let
$\zeta>0$ be sufficiently small.
For any initial product measure $\mu^\epsilon$
satisfying \eqref{Ass1}, there exists
$\tilde\Delta\in(0,1)$ such that the following
estimates hold uniformly for
$0<s<r\leq T$ and
$1\leq n,m\leq N_0-1$.

If $0<r-s\leq\tilde\Delta$, then
\begin{equation}
\label{bound:space_time_difference_v_estimate1}
\begin{aligned}
&\sup_{\underline y\in\Lambda_N^{m,\neq}}
\sup_{\substack{
\underline x\in\Lambda_N^{n,\neq}\\
\underline x+\underline e_1\in\Lambda_N^{n,\neq}
}}
\Big|
v_{s,r}^{\epsilon,m,n}
(\underline y,\underline x\mid\mu^\epsilon)
-
v_{s,r}^{\epsilon,m,n}
(\underline y,\underline x+\underline e_1
\mid\mu^\epsilon)
\Big|\lesssim
\frac{\epsilon^{2-4\zeta}}
{(r-s)^{1-\zeta}}.
\end{aligned}
\end{equation}

If $r-s>\tilde\Delta$, then
\begin{equation}
\label{bound:space_time_difference_v_estimate2}
\begin{aligned}
&\sup_{\underline y\in\Lambda_N^{m,\neq}}
\sup_{\substack{
\underline x\in\Lambda_N^{n,\neq}\\
\underline x+\underline e_1\in\Lambda_N^{n,\neq}
}}
\Big|
v_{s,r}^{\epsilon,m,n}
(\underline y,\underline x\mid\mu^\epsilon)
-
v_{s,r}^{\epsilon,m,n}
(\underline y,\underline x+\underline e_1
\mid\mu^\epsilon)
\Big|\lesssim\epsilon^{2-4\zeta}.
\end{aligned}
\end{equation}
Here $\underline e_1$ denotes the unit vector
in the positive first-coordinate direction.
\end{corollary}

The corollary is analogous to
Corollary~\ref{cor: thm for space v}
for the spatial $v$-functions and follows
from the proof of
Theorem~\ref{thm:vestimate3}.
Indeed, the short-time estimate follows
directly from the bound
\[
d_{\epsilon,s}^*(n,m)
\lesssim\epsilon^{1-7\zeta}
\]
and the definition of $\tilde\Gamma_d$.
The long-time estimate follows by restarting
Duhamel's formula on successive intervals,
as in the proof of
Theorem~\ref{thm:vestimate3}, and propagating
the corresponding bounds for the difference
functions.

\appendix

 \section{A single random walk with reflections}
\label{secN6}

In this section we state some properties of a single random walk in $\Lambda_N$ with reflections
at $\pm N$, referring to \cite{de2011current} for the proofs.

Let $Q_t^{(\epsilon)}(x,y)$  denote the transition probability of a
simple random walk in $\mathbb{Z}$ with jumps of intensity
$\epsilon^{-2}/2$ between nearest neighbor ({\it n.n.}) sites, and let
$P_t^{(\epsilon)}(x,y)$ denote the transition probability of a
corresponding random walk on $\La_N$, with the jumps outside $\La_N$
being suppressed (and the same jump rates within $\La_N$).
$P_t^{(\epsilon)}$ and $Q_t^{(\epsilon)}$ are thus related via the
``reflection map'' $\psi_N:\mathbb Z \to \La_N$ defined as follows.

\begin{itemize}

 \item\; $|x|\le N$:\;\;$\psi_N(x)=x$.

 \item\; $x<-N$:\;\;$\psi_N(x)=-\psi_N(-x)$;

 \item\; $x>N$:\;\;$\psi_N\left(N+j(2N+1)+k\right)=\psi_N\left(N+j(2N+1)-(k-1)\right)$,\;
  $k=1,\dots,2N+1, j=0,1,\dots$.
\end{itemize}

From the last identification we see that the two transition probabilities are related as follows
     \begin{equation}
    \label{a3.6}
P^{(\epsilon)}_t (x,z)= \sum_{y: \psi_N(y)=z} Q^{(\epsilon)}_t (x,y).
    \end{equation}

By the local central limit theorem, \cite{lawler2010random},
there exist positive constants $c_1,...,c_5$ so that for $|x-y|\le (\epsilon^{-2 }t)^{5/8}$ it holds:
    \begin{equation}
    \label{N4.4}
 |Q^{(\epsilon)}_t(x,y)- G_{\epsilon^{-2}t}(x,y)|\le   \frac{c_1}{ \sqrt{\epsilon^{-2}t}}
 G_{\epsilon^{-2}t}(x,y),\quad 
    \end{equation}
 while for $|x-y|> (\epsilon^{-2 }t)^{5/8} $ it holds:
    \begin{equation}
       \label{N4.4.1}
 Q^{(\epsilon)}_t(x,y) \le \min\Big\{ c_2 e^{-c_3|x-y|^{2}/(\epsilon^{-2}t)}, c_4 e^{-|y-x|(\log |y-x|-c_5)}\Big\}.
    \end  {equation}
As a corollary, for any $T>0$ there exists $c$  so that the following holds:
\begin{itemize}
\item For all $\epsilon$, all $t\in (0,T]$ and all $x$, $y$ in $\La_N$
     \begin{equation}
    \label{N4.5}
 P^{(\epsilon)}_t (x,y)    \le c \;G_{\epsilon^{-2}t}(x,y).
    \end{equation}

    \begin{remark}\label{rem:imp_estimate}
    We observe that in fact, due to non-integrability of the heat kernel close to zero, we used along the paper the estimate
\begin{equation}
    \label{N4.5 additional p}
 P^{(\epsilon)}_t (x,y)    \le \frac{2}{\sqrt{\epsilon^{-2}t}+1}.
    \end{equation}    One way to prove this estimate is to do the following. To simplify notation we denote $\lambda:=\epsilon^{-2}t$.
    First we note that since we are dealing with a probability we have that
    \begin{equation}
 P^{(\epsilon)}_t (x,y)    \le \min\{\frac{1}{\sqrt \lambda},1\}.
    \end{equation} To show the bound we do the following. First we split into two cases:
     \begin{itemize}
    \item $\sqrt \lambda>1$, in this case the minimum is attained at $\frac{1}{\sqrt \lambda}$. Now we note that 
    since $\sqrt \lambda>1$ then $2\sqrt \lambda>\sqrt \lambda +1$ from where we get that $\frac{1}{2\sqrt \lambda}< \frac{1}{\sqrt {\lambda}+1}$. From this we get that  \begin{equation}
 P^{(\epsilon)}_t (x,y)    \le \min\{\frac{1}{\sqrt \lambda},1\}=\frac{1}{\sqrt \lambda}<\frac{2}{\sqrt{\lambda}+1}.
    \end{equation} 
    \item $\sqrt \lambda<1$, in this case the minimum is attained at $1$. Now we note that 
    since $\sqrt \lambda<1$ then $2>\sqrt \lambda +1$ from where we get that $\frac{1}{2}< \frac{1}{\sqrt {\lambda}+1}$. From this we get that  \begin{equation}
 P^{(\epsilon)}_t (x,y)    \le \min\{\frac{1}{\sqrt \lambda},1\}=1<\frac{2}{\sqrt{\lambda}+1}.
    \end{equation}
    \end{itemize} 
    This observation also extends to the heat kernel.
    \end{remark}   
\item For all $\epsilon$, all $t\in (0,T]$ and all
 $-N\le x \le N-1$,
     \begin{equation}
    \label{N4.6}
\Big| P^{(\epsilon)}_t (x,y) - P^{(\epsilon)}_t (x+1,y) \Big| \le \frac{c}{\sqrt{\epsilon^{-2}t}} G_{\epsilon^{-2}t}(x,y).
    \end{equation}
Note that 
\begin{equation}
    \label{N4.5 additional}
G_{\epsilon^{-2}t}(x,y)\leq c\frac{1}{[\epsilon^{-2}t]^{1/2}+1}.
    \end{equation}
   As above we can also replace the last estimate by
   \begin{equation}
    \label{N4.6 additional}
\Big| P^{(\epsilon)}_t (x,y) - P^{(\epsilon)}_t (x+1,y) \Big| \le \frac{c}{\sqrt{\epsilon^{-2}t}+1} G_{\epsilon^{-2}t}(x,y).
    \end{equation}
 \end{itemize}

\section{Background on the  stirring process }\label{app: derivation of II for space v}
Before we proceed with the proofs, we start with the following definitions:
\begin{itemize}
\item Active/Passive marks process
\item  Labeled process 
\item Stopping time
 \end{itemize}

\begin{definition} [   The active/passive marks process]\label{def:active/passive marks process}
\quad

$\bullet$\; The active/passive marks process is realized in a probability space denoted by $(\Omega,\mathbf{P}_\epsilon)$.
It is a product of Poisson processes indexed by $\{x,x+1\}$, $x\in \mathbb Z$: for each pair $\{x,x+1\}$
we have a Poisson point process of intensity $\epsilon^{-2}$, its events are called ``marks'' and each mark
is independently given the attribute ``active'' or ``passive'' with probability $1/2$. The processes
relative to different pairs are mutually independent and their common law is $\mathbf{P}_\epsilon$, its expectation
being also denoted by $\mathbf{E}_\epsilon$.

$\bullet$\;  For any $\omega\in \Omega$ we  define the following particle
evolution in $\La_N$: a particle at $x\in \Lambda_N$ moves as soon as an
active mark appears at a pair $x,y$ with $y\in \Lambda_N$. The particle
then moves to $y$ (note that if another particle was at $y$ then it
would simultaneously move to $x$). Passive marks do not play any
role so far as well as all marks  $\{x,x+1\}$ with either $x$ or
$x+1$ not in $\Lambda_N$; they will be used later to construct
couplings. We denote by $X(t)\subset \Lambda_N$ the set of all sites
occupied by the particles at time $t$, so that $\eta(x,t)=1$ iff
$x\in X(t)$.
\end{definition}

\begin{definition} [The labeled process]\label{Def:labeled process}

Given $\omega\in \Omega$, we can follow unambiguously the motion of each
individual particle so that we can give them   labels at time 0
which then remain attached to the particles during their motion. We
shall denote by $\underline x=(x_{i_1},\dots,x_{i_n})$ a labeled
configuration of $n$ particles, ($i_1,\dots,i_n$ the labels,
$x_{i_j}$ the positions); configurations obtained under permutations
of the labels are now considered  distinct. We write $\underline x(t)$ for
the labeled process induced by the active/passive marks and denote
by $X(t)$ the unlabeled configuration obtained from $\underline x(t)$,
(i.e.\ only the positions in $\underline x(t)$ are recorded by $X(t)$). By
an abuse of notation we shall write $\mathbf{P}_\epsilon$ and $\mathbf{E}_\epsilon$ both
for law and expectation in the active/passive marks process and for
the marginal over   $\underline x(t)$, the labeled process realized in
this space.  By adding a subscript $\underline x$ we mean that  the
initial distribution of particles has support on the single labeled
configuration $\underline x$.
\end{definition}

\begin{definition}[The variables $\mathcal T_{x_1,x_2}$, $\tau_{x_1,x_2}$ and  $N_{x_1,x_2,t}$]

Given the initial position  $x_1$ and $x_2$ of particles $1$ and $2$
define in $\Omega$ the random multi-interval $\mathcal
T_{x_1,x_2}=\{s\geq 0: |x_1(s)-x_2(s)|=1\}$,  calling $\mathcal
T_{x_1,x_2,t}:=\mathcal T_{x_1,x_2}\cap [0,t]$ and $I_{x_1,x_2,t}=
\{ (s,y_1,y_2): s \in \mathcal T_{x_1,x_2,t}, y_i=x_i(s)$ and a mark
appears at $s$ between $y_1$ and $y_2$$\} $, we define the stopping
time $\tau_{x_1,x_2}$ as the smallest $s$ in $I_{x_1,x_2,\infty}$
and    $N_{x_1,x_2,t}$   the total number of elements in
$I_{x_1,x_2,t}$. Notice that the presence of other particles does
not affect the values  of $\mathcal T_{x_1,x_2}$, $\tau_{x_1,x_2}$
and  $N_{x_1,x_2,t}$.
\end{definition}
    \begin{lemma}[\cite{de2012truncated}, Lemma 4.3]\label{lem:DPTV_ejp 4.3}
      \label{x}
Let $\underline x=(x_1,\dots,x_n)$, $t>s>0$ and $f(y_1,\dots,y_n)$ a
function antisymmetric under the exchange of $y_1$ and $y_2$. Then
       \begin{equation}
       \label{18a.1}
 \mathbf{E}_{\epsilon,\underline x}\Big[ \mathbf 1_{\tau_{x_1,x_2} \leq s } f(\underline x(t))\Big] = 0,
                \end{equation}
where the suffix $\underline x$ indicates the initial condition i.e.
$\mathbf{E}_{\epsilon,\underline x}(\cdot)=\mathbf{E}_{\epsilon} [\cdot|\underline x(0)=\underline x]$.
\end{lemma}
\begin{definition} [{The stopping time $\tau_{i,j,t_0}$}] \label{def:stoptime}Let $\{\underline x(t)\}_{t\geq 0}$ be the labeled process realized in the active/passive marks process , let $i$ and $j$ be the labels of two of its particles and $t_0\geq 0$.
We then define $\tau_{i,j,t_0}$ as the first time $\tau > t_0$
when  {\em(i)} $|x_i(\tau)-x_j(\tau)|=1$  and {\em(ii)} at $\tau$ there is a mark (either active or  passive) between  $x_i(\tau)$ and $x_j(\tau)$; otherwise we set $\tau=\infty$.
When $t_0=0$ we just write $\tau_{i,j}$.
\end{definition}
  \begin{theorem}[Theorem 4.4, \cite{de2012truncated}]
   \label{3.01}
There is  $c$ so that for all $\epsilon>0$
                    \begin{equation}
              \label{19.e1}
\sup_{|x_1-x_2|=1} \mathbb{P}_\epsilon\big[\tau_{x_1,x_2} \geq t\big] \leq \frac{c}{(\epsilon^{-2}t)^{1/2}+1}.
             \end{equation}
Moreover, given any $T>0$, for any $\zeta>0$ and $k\ge 1$ there is
$c$ so that for all $t\leq T$ and  for all $\epsilon>0$
                      \begin{equation}
              \label{19.e2}
\sup_{x_1,x_2} \mathbb{P}_\epsilon\big[N_{x_1,x_2,t} \geq (\epsilon^{-2}t)^{1/2+\zeta}\big] \le  c (\epsilon^{-2}t)^{-k}.
             \end{equation}

\end{theorem}
\subsection{Useful estimates for transition probabilities}

\begin{lemma}\label{lem: bound T1}
Let $n\geq 2$. For any  $\underline{z}=(z_1,\dots,z_n)$ and $\underline{y}=(y_1,\dots,y_n)$ in $\Lambda_{N}^{n,\neq}$, the transition probability to go from $\underline{z}$ to $\underline{y}$ in a time $\lambda>0$ is bounded as:
\begin{equation}\label{bound T1 display}
\mathbf{P}_{\epsilon}(\underline{y}\overset{\lambda}\rightarrow\underline{z})\leq c_n \frac{1}{{{[\epsilon^{-2}\lambda]}^{n/2}+1}}\leq c_n \frac{1}{{{[\epsilon^{-2}\lambda]}^{1-\zeta}}}
\end{equation}
for some  constant $c_n>0$ which depends on $n$ and for some small $\zeta>0$. For $n=1$
\begin{equation}\label{bound T1 display'}
\mathbf{P}_{\epsilon}(y\overset{\lambda}\rightarrow z)\leq c_1 \frac{1}{{[\epsilon^{-2}\lambda]}^{1/2}+1}.
\end{equation}
\end{lemma}
\begin{proof}
 By applying Liggett's inequality \cite{de2012truncated}, we have that
\begin{equation}\label{Liggett}
\mathbf{P}_{\epsilon}(\underline{y}\overset{\lambda}\rightarrow\underline{z})\leq \prod_{i=1}^{n}\sum_{j=1}^{n}P^{(\epsilon)}_{\lambda}(y_i,z_j).
\end{equation}
  Recall \eqref{N4.5 additional p} we then get  \[
 \mathbf{P}_{\epsilon}(\underline{y}\overset{\lambda}\rightarrow\underline{z})\leq c_n\frac{1}{[\epsilon^{-2}\lambda]^{n/2}+1}
 \]
To finish the proof it is enough to note that 
if $\epsilon^{-2}\lambda\geq 1$, since $n\geq 2$ then  $(\epsilon^{-2}\lambda)^{\frac{n}{2}}+1\geq (\epsilon^{-2}\lambda)^{\frac{n}{2}}\geq  (\epsilon^{-2}\lambda)^{1-\zeta}$ for some $\zeta>0$. From this  the estimate follows. Now,  if $\epsilon^{-2}\lambda\leq 1$, then 
$ (\epsilon^{-2}\lambda)^{\frac{n}{2}}\leq  (\epsilon^{-2}\lambda)^{1-\zeta}\leq 1\leq (\epsilon^{-2}\lambda)^{\frac{n}{2}}+1$ and from this the proof ends. 
\end{proof}
\begin{remark}\label{rem:Liggett_ineq}
Let $S_n$ be the set of all permutations of the set $\{1,\dots,n\}$. We denote $\pi\in S_n$ a permutation of $\{1,\dots,n\}$. Then, for   $\underline{z}=(z_1,\dots,z_n)$ and $\underline{y}=(y_1,\dots,y_n)$ in $\Lambda_{N}^{n,\neq}$, the right-hand side of \eqref{Liggett}, can be written as $$\prod_{i=1}^{n}\sum_{j=1}^{n}P^{(\epsilon)}_{\lambda}(y_i,z_j)=\sum_{\pi\in S_n}\prod_{i=1}^n P^{(\epsilon)}_{\lambda}(y_{\pi(i)},z_i)$$ which gives us 
\begin{equation}\label{Liggett_new}
\mathbf{P}_{\epsilon}(\underline{y}\overset{\lambda}\rightarrow\underline{z})\leq \sum_{\pi\in S_n}\prod_{i=1}^n P^{(\epsilon)}_{\lambda}(y_{\pi(i)},z_i).
\end{equation}
\end{remark}

\subsection{Important bounds for space $v$-functions} \label{proof 13.1.111}

For the difference of $v$'s, we use  the next lemma:
   \begin{lemma}[\cite{de2012truncated}, Lemma 5.3]\label{lem:17.3}
We have
                     \begin{eqnarray}
                     \label{17.3}
|v^{\eps}_{n-1}(\underline x^{(i)},t)-v^{\eps}_{n-1}(\underline x^{(j)},t)| &\leq &
\int_0^t ds\; \mathbf{E}_\epsilon \Big[\mathbf 1_{\{\tau_{i,j}\geq \frac s2\}}\sum_{\underline y} \Big\{\mathbf{P}_\epsilon\Big(\underline x^{(j)}\Big(\tfrac{s}{2}\Big)\overset{s/2}\rightarrow \underline y\Big)
\nonumber\\ &+& \mathbf{P}_\epsilon\Big(\underline x^{(i)}\Big(\tfrac{s}{2}\Big)\overset{s/2}\rightarrow \underline y\Big)\Big\}  |(C_\epsilon v^{\eps}_{n-1})(\underline y,t-s)| \Big].
                   \end{eqnarray}
            \end{lemma}

\section{On the solutions to the discrete PDE (2.10)} 
\label{ap:useful_estimates}
Here we collected some results about the solution to \eqref{eq:linearized}. The first one was derived in \cite{de2012truncated} and the second one is a refinement of the former using a different method.

\begin{proposition}[\cite{de2012truncated}, eq. (3.8)--(3.9)]\label{prop:disc_grad_rho_eps}
For any $T>0$ there exists a finite constant $c$ so that for any solution $\rho_\epsilon$ of \eqref{eq:linearized} with $\rho_\epsilon(\cdot,0)$ the following holds. For any $x\in[-N,\dots,N-1]$, any $t\in(0,T]$ and any $\epsilon >0$
\begin{equation}\label{eq:dir_rho_lin}
|\rho_\epsilon(x,t)-\rho_\epsilon(x+1,t)|\leq \min\Big\{1,c\Big(\epsilon\log_+(\epsilon^{-2}t)+\frac{1}{\sqrt {\epsilon^{-2}t}}\Big)\Big\},
\end{equation}
where $\log_+(u)=\max \{\log(u),1\}$. We shall use a weaker version of \eqref{eq:dir_rho_lin}, namely that
  for any $\zeta>0$ and $\tau>0$ there is $c$ so that
       \begin{equation}
       \label{18a}
 \sup_{x,y\in \La_N: |x-y|\leq 1}|\rho_\epsilon(x,t)-\rho_\epsilon(y,t)| \le \frac {c}{(\epsilon^{-2}t)^{1/2-\zeta}+1}
\quad \text{{\rm for any} $t\le \tau \log \epsilon^{-1}$}.
                \end{equation}
\end{proposition}
{Now we note that for our purposes the previous results
are not enough, and we had to obtain more refined estimates as given in Lemma  \ref{lem: difference of rhos} whose proof we present below. }

\begin{corollary}\label{corollary: difference of rhos}
Let $x\in \{-N,\dots, N\}$ and $0\leq t\leq T$. There exists a constant $C_T$ which depends on $T$ such that for $0<\eps<1$
\begin{eqnarray*}
\sup_{x,y\in\Lambda_N:|x-y|\leq 1}|\rho_\epsilon(x,t)-\rho_\epsilon(y,t)|\leq C_T \epsilon^{1-2\zeta}
\end{eqnarray*}
\end{corollary}
\begin{proof}
From the proof Lemma \ref{lem: difference of rhos}, we can bound $\sup_{x,y\in\Lambda_N:|x-y|\leq 1}|\rho_\epsilon(x,t)-\rho_\epsilon(y,t)|$ by  
\begin{eqnarray}
{\epsilon c_T}+\epsilon^{-1}\frac j2 \int_{0}^t d\lambda \frac{1}{\sqrt{\epsilon^{-2}\lambda}+1}\sum_{z\in I_\pm}G_{\epsilon^{-2}\lambda}(x,z).\nonumber 
\end{eqnarray}
Now we treat the integral term above. To that end first we use the estimate in \eqref{N4.5 additional p} that as mentioned in Remark \ref{rem:imp_estimate} also holds for the heat kernel, to bound the integral in the last display by a constant times
\begin{eqnarray}
\epsilon^{-1}\frac j2 \int_{0}^t d\lambda \frac{1}{(\sqrt{\epsilon^{-2}\lambda}+1)^2}.\nonumber 
\end{eqnarray}
Last integral is bounded from above by 
\begin{eqnarray}
\epsilon^{-1}\frac j2 \int_{0}^t d\lambda \frac{1}{(\sqrt{\epsilon^{-2}\lambda}+1)^{2(1-\zeta)}}\leq \epsilon^{-1}\frac j2 \int_{0}^t d\lambda \frac{1}{(\sqrt{\epsilon^{-2}\lambda})^{2(1-\zeta)}}.\nonumber 
\end{eqnarray}
Now by computing the integral explicitly we see that it equals to 
$\zeta^{-1}\epsilon ^{1-2\zeta}t^\zeta$.
Putting together all the estimates we see that
$$\sup_{x,y\in\Lambda_N:|x-y|\leq 1}|\rho_\epsilon(x,t)-\rho_\epsilon(y,t)|\leq C(\epsilon^2 t+\epsilon+\epsilon^{1-2\zeta}t^{\zeta}).$$ 
Then the bound follows since $t\leq T$ and $\epsilon<1$.
\end{proof}



\bibliographystyle{plain}  
\bibliography{references} 

\end{document}